\documentclass{article}
\usepackage{amsmath}
\usepackage{caption}
\usepackage[utf8x]{inputenc}
\usepackage[mathscr]{eucal}
\usepackage{amsfonts, amssymb, amsmath, amscd, eucal, amsthm}
\usepackage{algorithm}
\usepackage{algpseudocode}

\usepackage{array}
\usepackage{graphicx}
\usepackage{subcaption} 
\usepackage{fancyhdr}
\usepackage{color}
\usepackage{extarrows}
\usepackage{enumerate}
\usepackage{multirow}
\usepackage[numbers]{natbib}
\usepackage{geometry}
\usepackage{CJK} 
\usepackage{bm}
\usepackage{float}
\usepackage{hyperref}
\hypersetup{
	colorlinks=true,
	linkcolor=green,
	filecolor=magenta,      
	urlcolor=cyan,
	pdftitle={Overleaf Example},
	pdfpagemode=FullScreen,
}
\usepackage{cleveref} 
\usepackage{comment} 
\usepackage{xurl}

\usepackage{setspace}
\def\NN{\mathbb{N}}

\def\SS{\mathbb{S}}

\newtheorem{theorem}{Theorem}

\theoremstyle{remark}
\newtheorem{remark}{Remark}
\newtheorem{example}{\textbf{Example}}

\title{Solving High-Dimensional PDEs Using Linearized Neural Networks}

\author{
Tong Mao\footnotemark[1]
\and
Jinchao Xu\thanks{
Applied Mathematics and Computational Sciences, Computer, Electrical and Mathematical Science and Engineering Division (CEMSE), King Abdullah University of Science and Technology, Thuwal, 23955, Saudi Arabia.
(\texttt{tong.mao@kaust.edu.sa, jinchao.xu@kaust.edu.sa})
}
\and 
Xiaofeng Xu\thanks{
Department of Mathematics, Pennsylvania State University, State College, PA 16801, USA.
(\texttt{xkx5060@psu.edu})
}
}
\date{}
\begin{document}
\maketitle

\begin{abstract}
Linearized shallow neural networks that are constructed by fixing the hidden-layer parameters have recently shown strong performance in solving partial differential equations (PDEs). 
Such models, widely used in the random feature method (RFM) and extreme learning machines (ELM), transform network training into a linear least-squares problem. 
In this paper, we conduct a numerical study of the variational (Galerkin) and collocation formulations for these linearized networks. 
Our numerical results reveal that, in the variational formulation, the associated linear systems are severely ill-conditioned, forming the primary computational bottleneck in scaling the neural network size, even when direct solvers are employed. 
In contrast, collocation methods combined with robust least-squares solvers exhibit better numerical stability and achieve higher accuracy as we increase neuron numbers.
This behavior is consistently observed for both ReLU$^k$ and $\tanh$ activations, with $\tanh$ networks exhibiting even worse conditioning. 
Furthermore, we demonstrate that random sampling of the hidden layer parameters, commonly used in RFM and ELM, is not necessary for achieving high accuracy.
For ReLU$^k$ activations, this follows from existing theory and is verified numerically in this paper, while for $\tanh$ activations, we introduce two deterministic schemes that achieve comparable accuracy.
\end{abstract}

\section{Introduction}
Neural networks have emerged as powerful tools in scientific computing, 
particularly for the numerical solution of partial differential equations~(PDEs). 
Representative approaches include physics-informed neural networks (PINNs)~\cite{Cuomo:2022,RPK:2019}, 
the deep Ritz method~\cite{EY:2018}, and the finite neuron method~\cite{Xu:2020}. 
All of these methods share a common principle: they seek optimal parameters for a neural 
network by minimizing a suitable energy functional, that is, by solving an optimization 
problem of the form
\begin{equation}
\label{NN_optimization}
\min_{v \in \mathcal{F}} E(v),
\end{equation}
where $\mathcal{F}$ denotes a neural network function class and $E \colon \mathcal{F} \to \mathbb{R}$ 
is the chosen energy functional.

Despite their success in various applications, the training of neural networks for PDEs remains challenging due to non-convexity and ill-conditioning~\cite{HSTX:2022}. 
A variety of training algorithms specifically tailored for PDEs 
have been developed beyond conventional gradient descent~\cite{GBC:2016}. 
Examples include the hybrid least–squares gradient descent method~\cite{CGPPT:2020}, 
motivated by the adaptive basis viewpoint, and the active neuron least–squares method~\cite{AS:2021,AS:2022}, 
which adjusts neuron positions adaptively to escape training plateaus. 
The neuron-wise parallel subspace correction method~\cite{PXX:2022} provides a 
preconditioning strategy for linear layers and updates neurons in parallel. 
More recently, greedy algorithms~\cite{DT:1996,Jones:1992,MZ:1993,Temlyakov:2011} 
have been adapted to neural network optimization and have shown strong performance; 
see, for example,~\cite{DPW:2021,SHJHX:2023,SX:2022a}. 
The orthogonal greedy algorithm (OGA)~\cite{SHJHX:2023,SX:2022a} in particular has been 
proven to achieve the optimal approximation rates of shallow neural networks~\cite{SX:2024}. 
The randomized variant of OGA~\cite{XX:2025} further improves computational 
efficiency by reducing the size of the discrete dictionary in each iteration, 
while still achieving optimal rates.
Other approaches avoid iterative training altogether by employing fixed feature maps. 
Random feature methods (RFM)~\cite{CCEY:2022,EMW:2020,LL:2024,RR:2007} and extreme learning machines (ELM)~\cite{DL:2021,HZS:2006,LLR:2024} have attracted increasing attention in this context and have achieved remarkably high numerical accuracy, despite the lack of rigorous mathematical analysis.
Only recently, it is established in \cite{LMX2025} that, for ReLU$^k$ shallow neural networks, the same optimal 
approximation rate
\[
O\!\left(n^{-\frac{1}{2} - \frac{2(k-m)+1}{2d}}\right)
\]
can be obtained if the neuron parameters $(\omega_i,b_i)$ are chosen to form a quasi-uniform grid on $S^d$, with the corresponding neurons directly serving 
as basis functions.

Motivated by the theoretical results of~\cite{LMX2025}, this paper conducts a detailed numerical study of linearized shallow neural networks with pre-determined neurons for solving high-dimensional PDEs.
We focus on two questions regarding the behavior of linearized neural networks.
The first concerns the comparison between the variational (Galerkin) and collocation (least-squares) formulations.
In traditional finite element methods, the variational approach is standard, while in extreme learning machines and random feature methods, the collocation approach is almost exclusively used.
Understanding their relative performance in terms of numerical stability and accuracy is therefore of both theoretical and practical interest.
The second question is whether deterministic parameter constructions can effectively replace random sampling in setting the neuron parameters, and whether such deterministic designs can achieve comparable accuracy, particularly for $\tanh$-activated shallow neural networks, where theoretical results are not yet available.
We believe this investigation will also shed light on various other activation functions used in shallow neural networks. 

\medskip
\noindent
The main contributions of this paper are as follows:
\begin{itemize}
  \item \textbf{Linearized $\text{ReLU}^k$ shallow networks with pre-determined neurons.}  
  We study ReLU$^k$ shallow neural networks with hidden-layer parameters chosen deterministically as quasi-uniform grids on $S^d$. 
  We numerically demonstrate that the optimal approximation rate can be achieved using this approach without the need to optimize the hidden-layer parameters. 
    In low dimensions, neuron parameters can be arranged in quasi-uniform grids on the sphere. 
    In higher dimensions, however, constructing such grids is challenging. 
    To address this, we consider two alternatives: 
    (i) setting hidden-layer parameters by i.i.d.\ uniform sampling on the sphere, 
    which corresponds to a random feature method for ReLU$^k$ shallow networks, and 
    (ii) employing quasi–Monte Carlo (QMC) constructions to generate points on spheres, which are then used to fix the neuron parameters. 
Our numerical experiments demonstrate that QMC-based parameter selection 
provides higher accuracy than purely random sampling, indicating the necessity of designing high-quality quasi-uniform grids on hyperspheres. 
  \item \textbf{Linearized $\tanh$ shallow neural networks with pre-determined neurons.}
As an alternative to random sampling of neuron parameters from $[-R,R]^{d+1}$, with an appropriate value of $R$, we introduce two deterministic schemes for fixing the nonlinear layer parameters:
(i) a \emph{Petrushev-type grid scheme}, which selects neuron parameters  from the tensor product of evenly spaced points on a scaled sphere and a one-dimensional interval; and
(ii) a \emph{sphere-based scheme}, which distributes the parameters quasi-uniformly on a scaled spherical surface of an appropriate radius. 
Both schemes eliminate randomness, ensure reproducible results, and achieve accuracy comparable to random sampling from $[-R,R]^{d+1}$.

  \item \textbf{Comparison between variational and collocation formulations.} 
  We conduct a numerical comparison of variational (Galerkin) and collocation formulations through $L^2$-minimization examples.  
  We observe that the numerical solution obtained from the variational formulation is less stable than that from the collocation method when increasing the number of neurons. 
  We investigate the conditioning of the resulting linear systems from the variational formulation 
and identify it as a primary bottleneck in scaling neural network size. 
While variational (Galerkin) formulations offer theoretical clarity, 
they suffer from severe ill-conditioning as the number of neurons grows, 
leading to numerical instability. 
In contrast, despite a lack of theoretical analysis, collocation methods combined with robust least-squares solvers exhibit better numerical stability in our numerical experiments.

\end{itemize}

\medskip
\noindent
The remainder of the paper is organized as follows. 
Section~\ref{sec:theoretical-results} introduces the ReLU$^k$ and $\tanh$ activation functions and recalls approximation results for the corresponding shallow neural networks.
Section~\ref{sec:methods} presents the variational and collocation formulations used in our numerical study. 
Section~\ref{sec:numerical-experiments} contains the main numerical experiments, including $L^2$-minimization problems, PDE benchmarks in multiple dimensions, and a study of $\tanh$ activation. 
We conclude in Section~\ref{sec:conclusion} with a summary of the main findings and directions for future research.

\section{Theoretical results}\label{sec:theoretical-results}

\subsection{ReLU$^k$ shallow neural network}
We recall fundamental approximation results for ReLU$^k$ shallow neural networks, which serve as the theoretical basis of our study. 
We examine the function class $\Sigma_{n, M} (\mathbb{D})$, defined for some $M > 0$, as
\begin{equation}\label{eq:sigma_n_M}
\Sigma_{n,M} (\mathbb{D}) = \left\{ \sum_{i=1}^n a_i d_i : a_i \in \mathbb{R}, \text{ } d_i \in \mathbb{D}, \text{ } \sum_{i = 1}^n |a_i| \leq M \right\},
\end{equation}
where the dictionary $\mathbb{D}$ is symmetric and is given by
\begin{equation}
\label{dictionary}
\mathbb{D} = \mathbb{P}_k^d:=
\left\{ \pm \text{ReLU}^k (\omega \cdot x + b ) :  (\omega,b) \in S^{d} \right\}.
\end{equation}

\begin{theorem}{\cite{SX:2024}}
Given $u \in \mathcal{K}_1(\mathbb{D})$, there exists a positive constant $M$ depending on $k$ and $d$ such that
\[
\inf_{u_n \in \Sigma_{n, M}(\mathbb{D})} \| u - u_n \|_{H^m(\Omega)} \leq C \| u \|_{\mathcal{K}_1(\mathbb{D})} n^{-\frac{1}{2} - \frac{2(k-m)+1}{2d}},
\]
where $\Sigma_{n, M}(\mathbb{D})$ and $\mathcal{K}_1(\mathbb{D})$ are the shallow neural network function space, and the variation space with respect to the $\text{ReLU}^k$ dictionary $\mathbb{D}$, respectively, and $C$ is a constant independent of $u$ and $n$.
\end{theorem}
This result implies that ReLU$^k$ shallow networks achieve approximation rates of $O(n^{-\frac{1}{2} - \frac{2(k-m)+1}{2d}})$ in the variation space $\mathcal{K}_1(\mathbb{D})$. 

Recently, it is proved in~\cite{LMX2025} that the same rate can be obtained when the neuron parameters $(\omega_i,b_i)$ are fixed to form a quasi-uniform grid on $S^d$.
Given a predetermined set of parameters $\left\{\theta_j\right\}_{j=1}^n \subset \mathbb{S}^d$, we define the corresponding basis functions
$$
\phi_j(x)=\sigma_k\left(\theta_j \cdot \tilde{x}\right)=\sigma_k\left(w_j \cdot x+b_j\right), \quad j=1, \ldots, n, 
$$
and the finite neuron space as
$$
L_n^k=L_n^k\left(\left\{\theta_j\right\}_{j=1}^n\right)=\operatorname{span}\left\{\phi_1, \ldots, \phi_n\right\}. 
$$

Analogous to $\Sigma_{n, M}(\mathbb{D})$ in \eqref{eq:sigma_n_M}, we define the constrained version of $L_n^k$ as
$$
L_{n, M}^k=\left\{\sum_{j=1}^n a_j \phi_j:\left(n \sum_{j=1}^n a_j^2\right)^{\frac{1}{2}} \leq M\right\}.
$$



\begin{theorem}{\cite{LMX2025}}\label{thm:lfns}
 Let $k \geq m$, $f \in H^{\frac{d+2 k+1}{2}}(\Omega)$ and $\{\theta_j\}_{j=1}^n$ being quasi-uniform, we have 
 \begin{equation}
\begin{aligned}
\inf _{f_n \in L_{n, M}^k}\left\|f-f_n\right\|_{H^m(\Omega)} \lesssim n^{-\frac{1}{2}-\frac{2 (k-m)+1}{2 d}}\|f\|_{ H^{\frac{d+2 k+1}{2}}(\Omega)}.
\end{aligned}
\end{equation}    
\end{theorem}
The theoretical finding implies that we can use ReLU$^k$ shallow neural networks with pre-fixed deterministic neurons in approximation functions and solving PDEs, which we will verify numerically in Section \ref{sec:numerical-experiments}. 

\begin{remark}
    The above theorem applies to functions in Sobolev spaces. Notably, one can further show that such Sobolev spaces and the Barron space are of comparable complexity in terms of metric entropy.
\end{remark}

\subsection{Tanh shallow neural networks}
\label{sec:tanh_theory}

We now recall several classical approximation results for shallow neural networks
with the $\tanh$ activation function which is very commonly used in 
the random feature method (RFM) and extreme learning machine (ELM) frameworks. 
Let $\Sigma_n(\tanh)$ denote the set of all single-hidden-layer neural networks
with $n$ hidden neurons and the $\tanh$ activation function
\begin{equation}
    \Sigma_n(\tanh) = \left\{ \sum_{i = 1 }^n a_i \tanh(\omega_i \cdot x + b_i), \omega_i \in \mathbb{R}^d, b_i \in \mathbb{R}, x \in \mathbb{R}^d \right\}.  
\end{equation}

\medskip
\noindent
In the work of Mhaskar~\cite{mhaskar1996neural}, 
the following convergence results were established.
For target functions $f$ belonging to the Sobolev space
$W^{r,p}([-1,1]^d)$ with $1 \leq p \leq \infty$, one has
\begin{equation}
\label{eq:tanh_poly_rate}
\inf_{f_n \in \Sigma_n(\tanh)} 
\|f - f_n\|_{L^p} 
\leq C\, n^{-r/d}\, \|f\|_{W^{r,p}}.
\end{equation}
This estimate shows that $\tanh$-activated shallow networks can achieve
algebraic convergence of order $\mathcal{O}(n^{-r/d})$ for functions of Sobolev regularity~$r$.


\medskip
\noindent
Earlier, Barron~\cite{barron1993universal} established an $L^2$ approximation rate
based on Fourier integral representations.
If $f$ has finite Barron norm
\[
\|f\|_{\mathcal{B}_1}
:= \int_{\mathbb{R}^d} \|\omega\|_2\, |\hat{f}(\omega)|\, d\omega < \infty,
\]
then
\begin{equation}
\label{eq:barron_rate}
\inf_{f_n \in \Sigma_n(\tanh)} 
\|f - f_n\|_{L^2}
\leq
C\, \frac{\|f\|_{\mathcal{B}_1}}{\sqrt{n}}.
\end{equation}
This $\mathcal{O}(n^{-1/2})$ rate was later extended to Sobolev norms 
by Siegel and Xu~\cite{siegel2020approximation}, who proved that
\begin{equation}
\label{eq:siegelxu_rate}
\inf_{f_n \in \Sigma_n(\tanh)} 
\|f - f_n\|_{H^m(\Omega)} 
\leq
C\, n^{-1/2}\, \|f\|_{B^{m+1}},
\end{equation}
where $B^{m+1}$ denotes a suitable Barron-type space.

\medskip
\noindent
These theoretical results collectively establish that $\tanh$ activations possess 
strong approximation capabilities when the neuron parameters are optimized.
However, none of these analyses provide an explicit rule for constructing or pre-fixing the nonlinear parameters $(\omega_i,b_i)$.
In practice, RFM and ELM typically assign these parameters randomly from a uniform distribution on an interval, without theoretical guidance on how such choices affect the convergence rate.
In this paper, motivated by the result of ReLU$^k$ shallow neural networks as shown in Theorem \ref{thm:lfns}, we investigate whether deterministic constructions 
of fixed neuron parameters can achieve comparable approximation performance to random sampling.
Two deterministic schemes are proposed and studied numerically in 
Section~\ref{sec:numerical-experiments}.

\section{Methodology}\label{sec:methods}

In this section, we describe the two formulations used in our numerical study of linearized shallow neural networks for solving PDEs, namely, the variational and collocation formulations. 
The corresponding numerical methods are referred to as variational and collocation methods, respectively.

\subsection{Variational Method}

We introduce the variational method in the context of an $L^2$-minimization problem and an elliptic partial differential equation. 
For more general problem settings, we refer the reader to~\cite{Xu:2020}. 

We first consider the $L^2$-minimization problem. 
Given a target function $u \in L^2(\Omega)$ and a subspace $V_n \subset L^2(\Omega)$, 
the variational problem seeks
\[
u_n = \arg\min_{v \in V_n} \|u - v\|_{L^2(\Omega)}.
\]
Equivalently, $u_n \in V_n$ satisfies the Galerkin orthogonality condition
\[
\langle u_n, v \rangle_{L^2(\Omega)} = \langle u, v \rangle_{L^2(\Omega)}, \quad \forall v \in V_n,
\]
where $\langle f,g \rangle_{L^2(\Omega)} := \int_\Omega f(x) g(x)\,dx$. 
Choosing basis functions $\{\phi_j\}_{j=1}^n \subset V_n$ 
and writing $u_n = \sum_{j=1}^n a_j \phi_j$, 
the coefficients $a=(a_1,\dots,a_n)^\top$ are determined by the linear system
\[
M a = b,
\]
with entries
\[
M_{ij} = \int_\Omega \phi_j(x)\phi_i(x)\,dx, 
\qquad
b_i = \int_\Omega u(x)\phi_i(x)\,dx .
\]

We now consider the elliptic PDE
\begin{align*}
- \Delta u + u &= f \quad \text{in } \Omega, \\
\frac{\partial u}{\partial n} &= 0 \quad \text{on } \partial \Omega.
\end{align*}
The variational method seeks $u_n \in V_n$ such that
\begin{equation} \label{eq:variational-neumann}
a(u_n,v) = \ell(v), \quad \forall v \in V_n,    
\end{equation}
with bilinear form and linear functional
\[
a(w,v) = \int_\Omega \nabla w(x)\cdot \nabla v(x)\,dx 
+ \int_\Omega w(x) v(x)\,dx, 
\qquad 
\ell(v) = \int_\Omega f(x) v(x)\,dx .
\]
In other words, $u_n$ is the Galerkin projection of the true solution $u$ 
with respect to the energy inner product induced by $a(\cdot,\cdot)$. 

Similarly, writing $u_n = \sum_{j=1}^n a_j \phi_j$, 
the coefficients $a=(a_1,\dots,a_n)^\top$ are determined by the linear system
\[
A a = b,
\]
with entries
\[
A_{ij} = \int_\Omega \nabla \phi_j(x)\cdot \nabla \phi_i(x)\,dx 
+ \int_\Omega \phi_j(x)\phi_i(x)\,dx, 
\qquad
b_i = \int_\Omega f(x)\phi_i(x)\,dx .
\]

\begin{remark}
For many PDEs, the variational formulation admits an equivalent energy minimization formulation, which is widely used in neural-network-based PDE solvers.
This principle underlies the \emph{deep Ritz method}~\cite{EY:2018},
the \emph{finite neuron method}~\cite{Xu:2020},
and recent greedy algorithms for neural network optimization in PDEs~\cite{SHJHX:2023,XX:2025}.
\end{remark}

\subsection{Collocation method}
We consider the following linear boundary-value problem:
\begin{equation}\label{model_pde}
\begin{aligned}
     \mathcal{L}u(x) & = f(x), \quad && x \in \Omega, \\
     \mathcal{B}u(x) & = g(x), \quad && x \in \partial \Omega, 
\end{aligned}
\end{equation}
where $\mathcal{L}$ and $\mathcal{B}$ denote the linear differential and boundary operators, respectively. 
For simplicity, we restrict our attention to scalar-valued solutions. 

The collocation method enforces the PDE residual at interior points of the domain and the boundary condition at boundary points, using a finite set of points called collocation points.
It leads to a least-squares problem and is widely used in PINNs~\cite{raissi2019physics}, random feature methods~\cite{chen2022bridging}, and extreme learning machines~\cite{HZS:2006,wang2024extreme}.
Let $u_n(x) := \sum_{i = 1}^n a_i \phi_i(x)$ be the shallow neural network with $n$ pre-fixed neurons $\{ \phi_i(x)\}_{i = 1}^n$. 
Let $C_I$ be the set of interior collocation points ($|C_I|=N_I$) 
and $C_B$ the set of boundary collocation points ($|C_B|=N_B$), 
with $C := C_I \cup C_B$. 
The loss function is given by
\begin{equation}\label{loss}
    L(\mathbf{a})= \tfrac{1}{2} \sum_{x_i \in C_I} \lambda_I^i \bigl( \mathcal{L} u_n(x_i) - f(x_i) \bigr)^2 
    +  \tfrac{1}{2} \sum_{x_j^b \in C_B} \lambda_B^j  \bigl( \mathcal{B} u_n(x_j^b) - g(x_j^b) \bigr)^2,
\end{equation}
where $\mathbf{a}$ denotes the coefficients $a_{j}$ in~\eqref{eq:sigma_n_M}, 
and $\lambda_I^i, \lambda_B^j>0$ are penalty parameters that scale the contributions 
of different collocation points. 

In matrix form, this becomes
\begin{equation}\label{loss-matrix}
    \min_{\mathbf{a}} \; \tfrac{1}{2}\|\Lambda_I(D_I \mathbf{a} - \mathbf{f}) \|_2^2 
    + \tfrac{1}{2}\| \Lambda_B (D_B \mathbf{a}- \mathbf{g} ) \|_2^2. 
\end{equation}
In the above matrix form, $\Lambda_I = \text{diag}(\lambda_I^1,...,\lambda_I^{N_I})$, $\Lambda_B = \text{diag}(\lambda_B^1,...,\lambda_B^{N_B})$ are two diagonal matrices that scale the least-squares problem. 
The matrices $D_I$ and $D_B$ are given by $[D_I]_{i,j} = \mathcal{L}\phi_{j}(x_i) $, and $[D_B]_{i,j} = \mathcal{B}\phi_{j}(x_i) (x^b_i)$. 
The two vectors $\mathbf{f}$ and $\mathbf{g}$ are the values of the functions $f$ and $g$ evaluated at the collocation points, respectively.

In the special case of $L^2$-minimization, the problem reduces to a standard discrete $\ell^2$-regression.
The minimizer $\mathbf{a}$ is computed using a least-squares solver in both RFM and ELM.


\section{Numerical experiments}\label{sec:numerical-experiments}
In this section, we provide a numerical study of the variational and collocation formulations.
For ReLU$^k$ shallow neural networks, we numerically confirm that prescribing nonlinear layer parameters via a quasi-uniform grid on the sphere yields the optimal approximation rate under both formulations.
This verifies the most recent theoretical results on linearized ReLU$^k$ shallow neural networks, as stated in Theorem \ref{thm:lfns}.
Our experiments further reveal that, for both ReLU$^k$ and $\tanh$ activations, the linear system arising in the variational formulation is extremely ill-conditioned, which quickly becomes a bottleneck to achieving high accuracy when scaling the neural network size.

Moreover, using a second-order elliptic PDE example, we verify the convergence of linearized ReLU$^k$ shallow neural networks. 
We also demonstrate that for ReLU$^k$ networks, quasi–Monte Carlo (QMC) constructions of hidden-layer parameters outperform random sampling from the uniform distribution on the sphere $S^d$, indicating the need for high-quality quasi-uniform grids on $S^d$.
This approach can be applied in dimensions higher than three, where constructing a quasi-uniform grid on the hypersphere is nontrivial.
For neural networks with a $\tanh$ activation function, we show that linearized neural networks can be constructed using two deterministic schemes that we propose.
The resulting linearized neural networks achieve accuracy comparable to those constructed through random sampling, as commonly seen in ELM and RFM.
These results indicate that randomness may not be necessary to achieve high accuracy when using linearized neural networks.

In our experiments, for the variational formulation, numerical integration is required to assemble the linear system. 
In one, two, and three dimensions, we employ piecewise Gauss–Legendre quadrature on uniform subdomains, while in higher dimensions we rely on quasi–Monte Carlo methods with a large number of samples~\cite{Caflisch:1998,MC:1995}. 
The number of integration points is chosen to guarantee sufficient accuracy and is specified for each experiment. 
For the collocation formulation, the number of collocation points is likewise detailed in the corresponding experiments.

\subsection{ReLU$^k$ Shallow Neural Networks}
\label{subsec:relul2}

We first consider $L^2$-minimization problems in the domain $\Omega = (-1,1)^d$ 
for dimensions $d=1$--$6$. 
Throughout this subsection the target function is taken to be
\begin{equation}
    u(x) = \prod_{i=1}^d \sin\!\left(\tfrac{\pi x_i}{2}\right).
\end{equation}
 
 Before proceeding to the numerical results, we introduce how to construct a quasi-uniform grid points on a sphere $S^d$.
 Constructing quasi-uniform grid points on the sphere $S^d$ is straightforward for $d = 1, 2$.
In one dimension, we can use the polar coordinate representation. By dividing the angular variable uniformly, we obtain a uniform grid on the circle $S^1$.
In two dimensions, we can generate a uniform grid on $S^2$ using the following procedure \cite{gonzalez2010measurement}. For $n$ points, we first define the indices $i = 0.5, 1.5, \ldots, n - 0.5$. Then, for each index $i$, we compute:
\begin{align*}
\phi = \arccos(1 - 2i/n), \quad 
\theta = \pi(1 + \sqrt{5})i.
\end{align*}
The coordinates of the corresponding points are then given by:
\begin{align*}
x = \sin(\phi)\cos(\theta), \quad 
y = \sin(\phi)\sin(\theta), \quad 
z = \cos(\phi).
\end{align*}
This construction produces a spiral pattern of points that forms an approximately uniform grid on the sphere. The method is based on the golden ratio $\varphi = \frac{1 + \sqrt{5}}{2}$, which helps ensure good spacing between points.
However, for higher dimensions ($d \geq 3$), constructing a quasi-uniform grid on $S^d$ becomes a nontrivial task.
Therefore, for convenience, we instead employ a uniform distribution on $S^d$ to randomly sample grid points, similar to the random feature method. In this case, the grid points are not strictly quasi-uniform distributed, but the same approximation rate still holds (up to logarithm) with high probability (see \cite[Theorem 6.2]{LMX2025}).

\subsubsection{Continuous $L^2$-minimization}
In this experiment, we use ReLU$^2$ shallow neural networks for solving the $L^2$-minimization problem in dimensions 1-6. 
In the variational formulation, the approximation $u_n \in L_n^k$ is obtained 
as the $L^2$-projection of $u$, requiring the evaluation of continuous inner products. 
For numerical integration, in 1D, we divide the interval into 1024 uniform subintervals and use a Gauss quadrature of order 5 in each subinterval.
In 2D, we divide the square domain into $400^2$ uniform subdomains and use a Gauss quadrature of order 5 in each subdomain.
In 3D, we divide the cubic domain into $50^3$ uniform subdomains and use a Gauss quadrature of order 3 in each subdomain.
In higher dimensions, we use the Quasi-Monte Carlo method with integration points generated by Sobol sequences.
The numbers of integration points are $8\times10^5$, $8\times10^5$, and $2\times10^6$ in 4D, 5D, and 6D, respectively.

Numerical results for dimensions $d=1$--$6$ are shown in 
Figure~\ref{fig:L2-variational}.  
Across all the dimensions tested, the numerical error decreases as expected with 
increasing number of neurons, matching the theoretically guaranteed approximation rate of ReLU$^2$ shallow neural networks. 

\begin{figure}[H]
    \centering
    \includegraphics[width=0.325\linewidth]{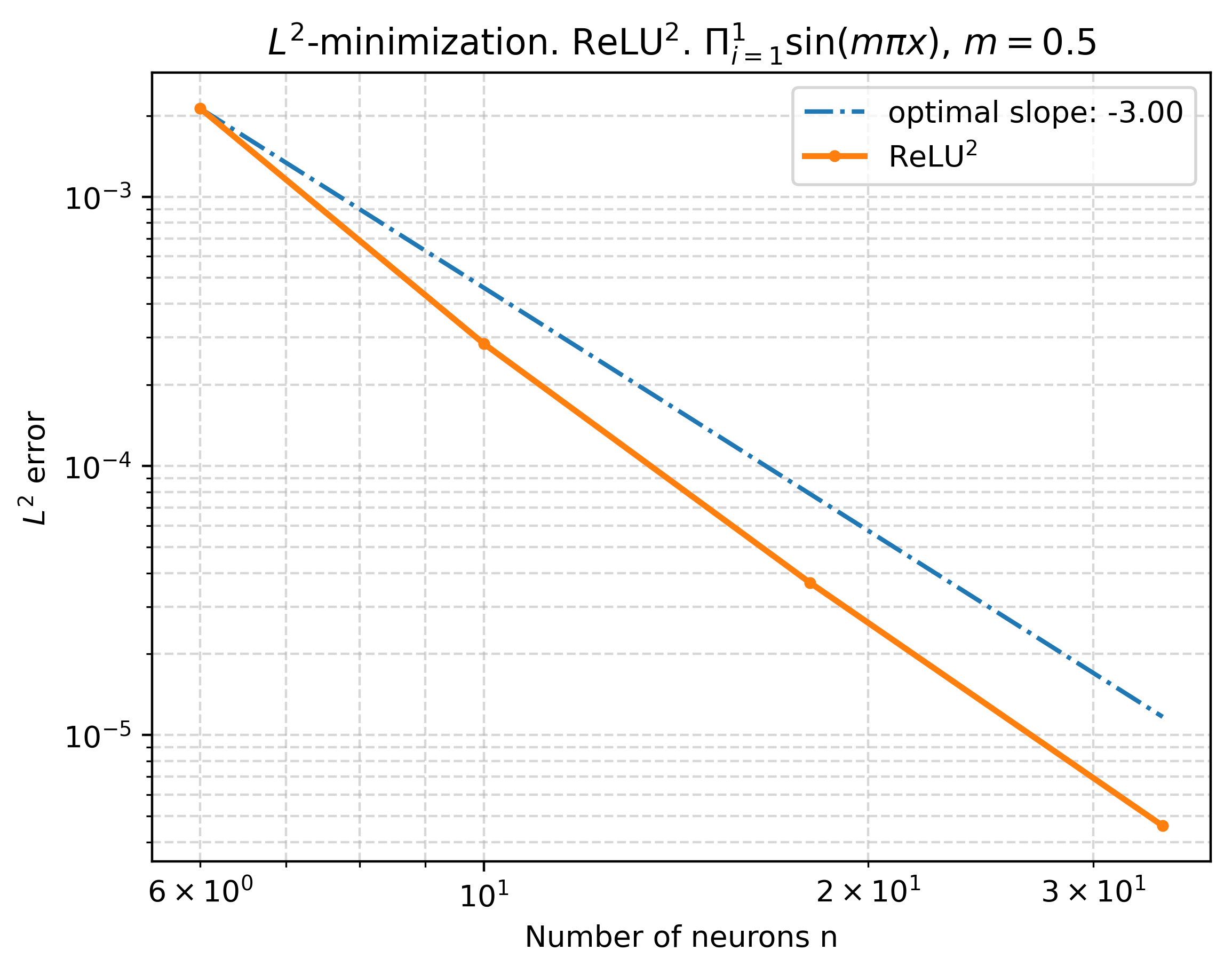}
    \includegraphics[width=0.325\linewidth]{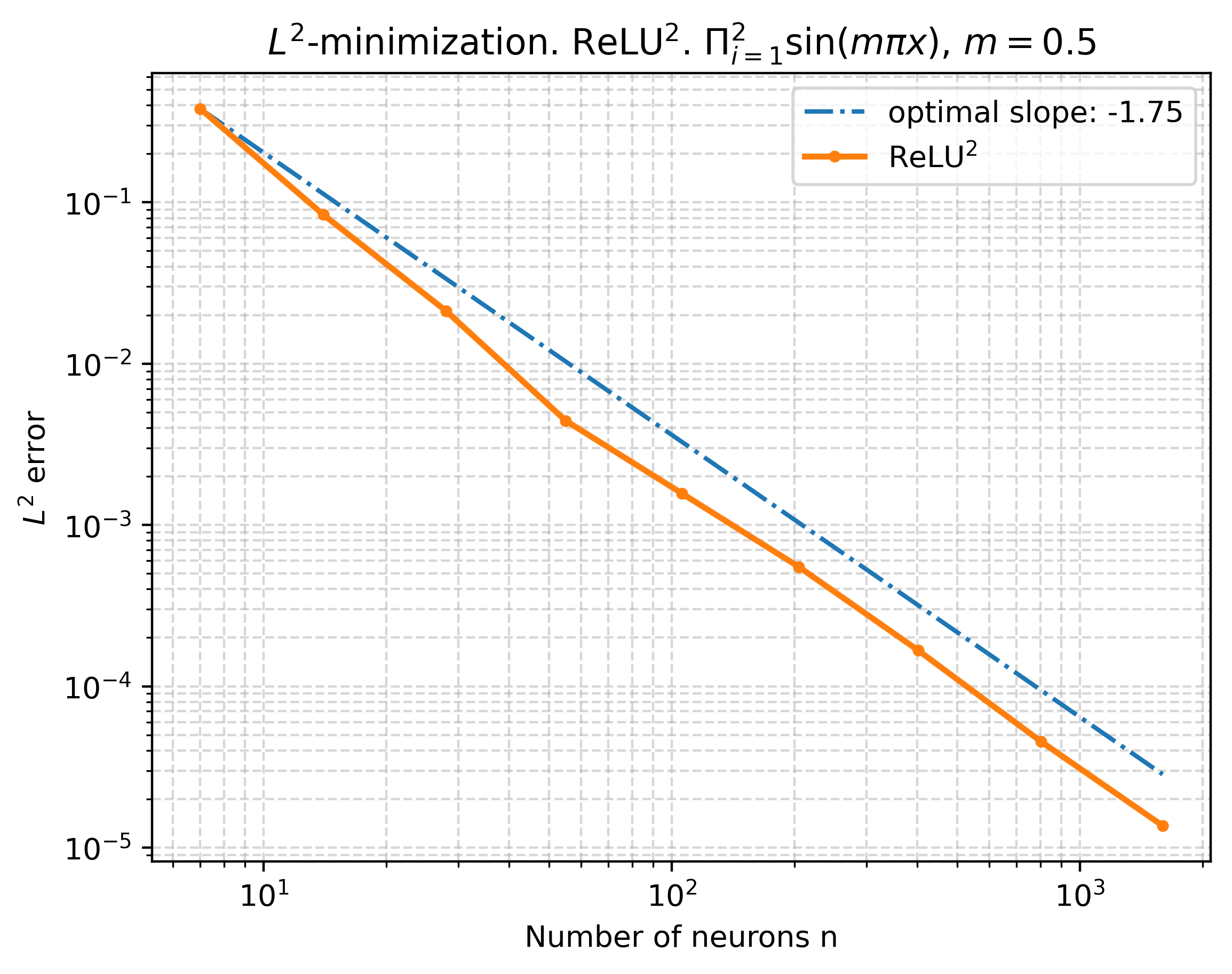}
    \includegraphics[width=0.325\linewidth]{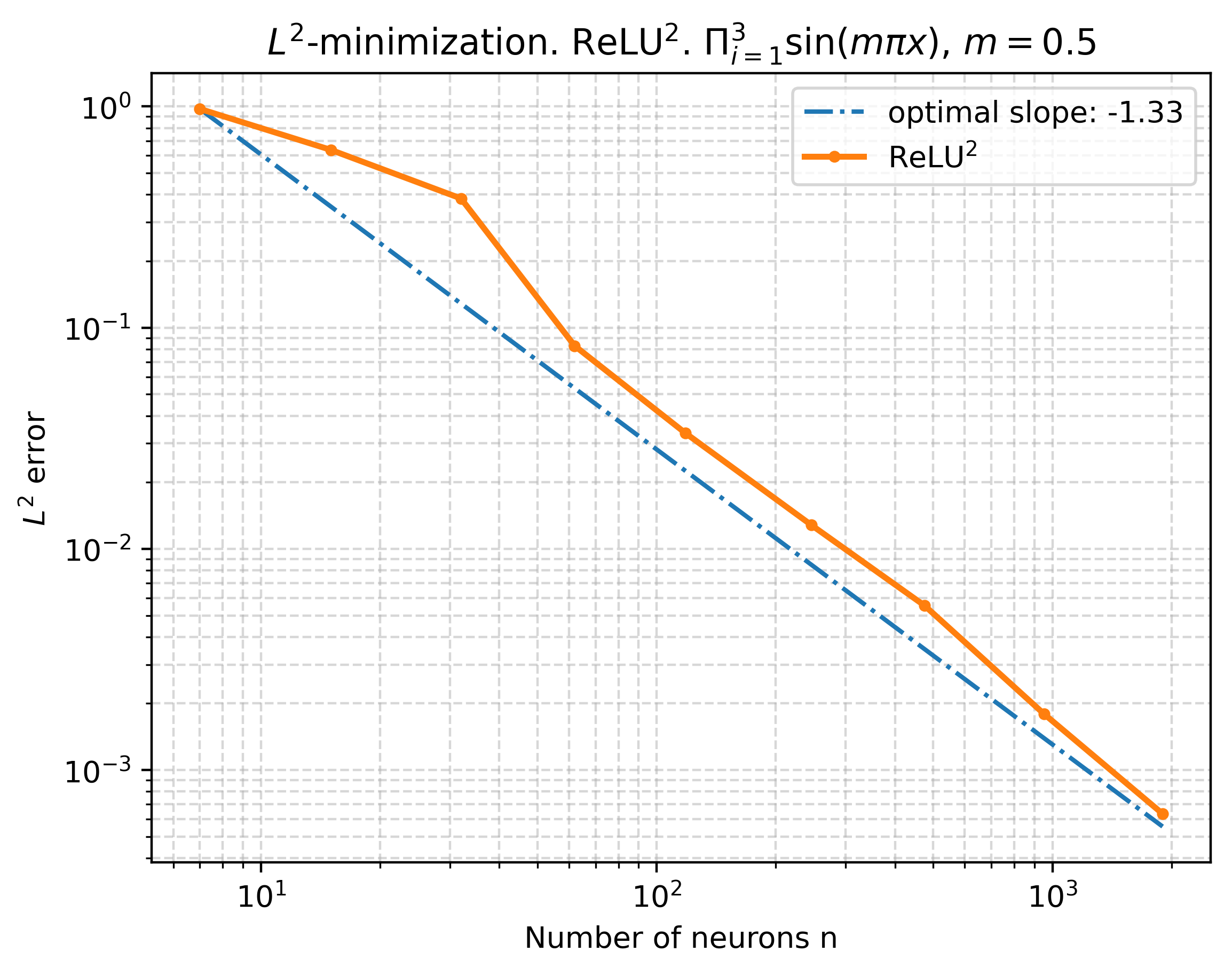} \\
    \includegraphics[width=0.325\linewidth]{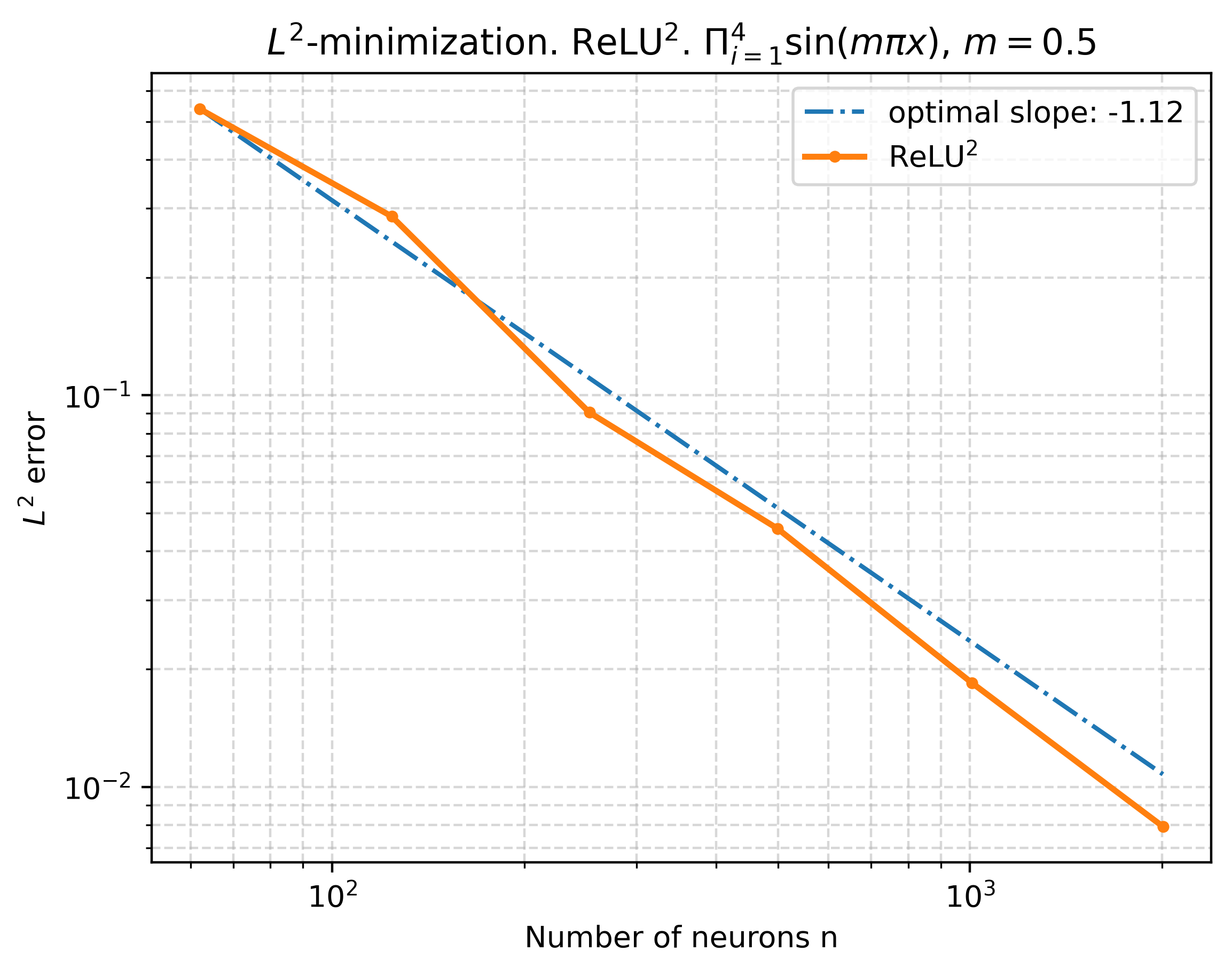}
    \includegraphics[width=0.325\linewidth]{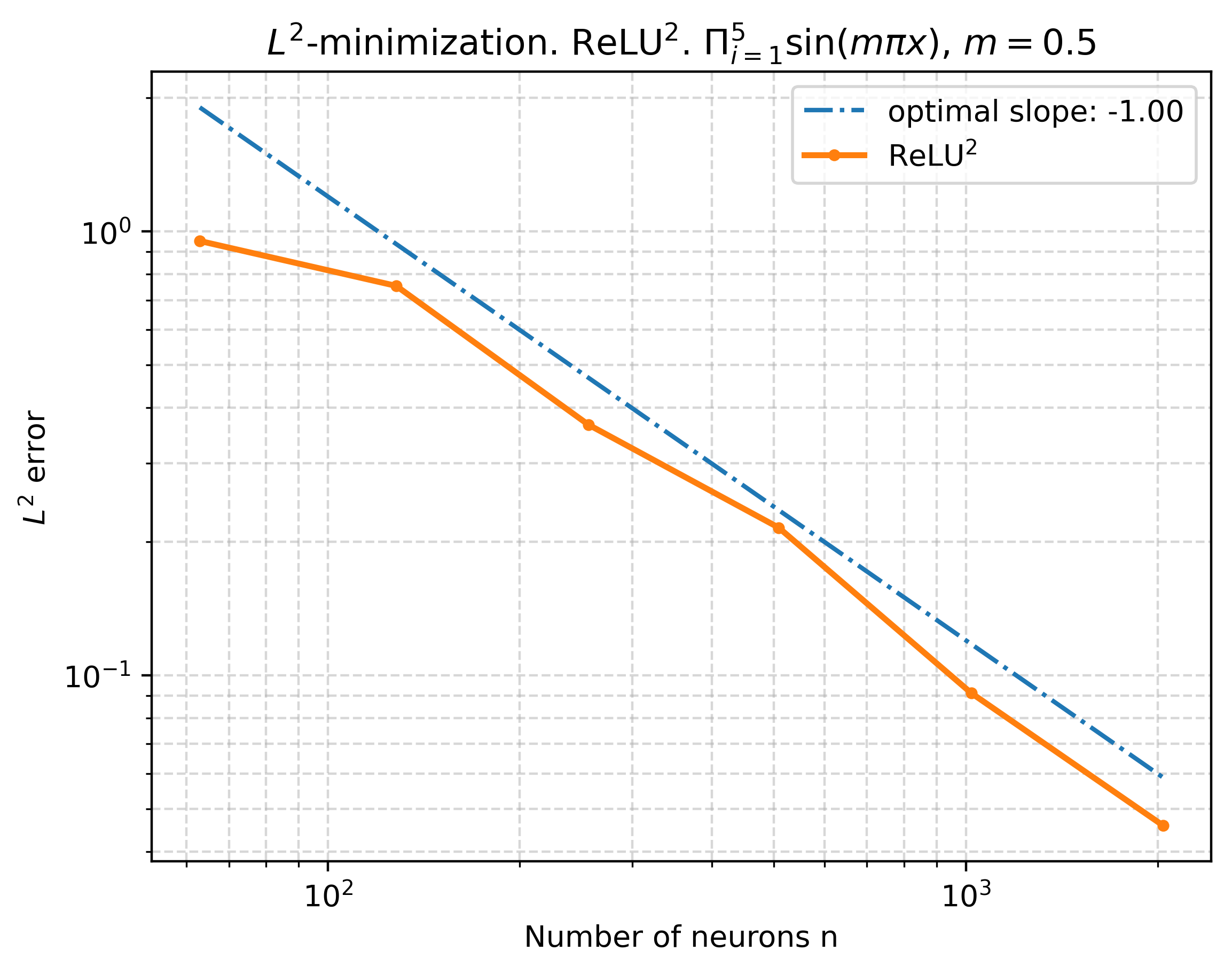}
    \includegraphics[width=0.325\linewidth]{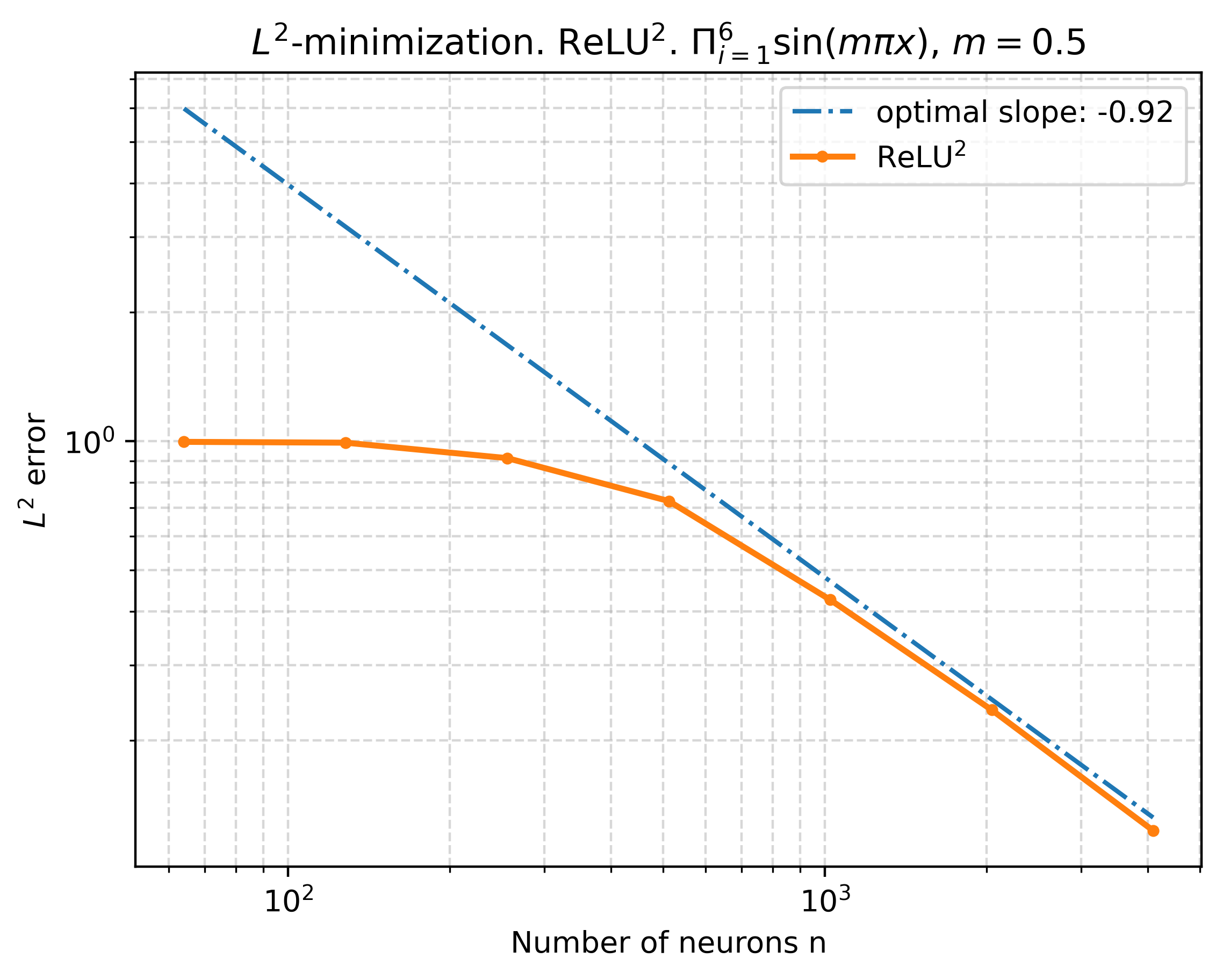} 
    \caption{Error decay for continuous $L^2$-minimization (variational formulation) 
    in dimensions $d=1$--$6$.}
    \label{fig:L2-variational}
\end{figure}

Additional experiments with ReLU activation are reported in Appendix~\ref{app:supp-l2}.  

\subsubsection{Discrete $\ell^2$-regression}  
We next consider the discrete least-squares regression problem, 
which corresponds to the collocation formulation introduced in 
Section~\ref{sec:methods}.  
In one dimension, the collocation points are chosen as a uniform grid of size $2048$ over the interval, including the boundary points. 
In two dimensions, a uniform $100\times100$ tensor-product grid is used over the domain, including the boundary points. 
In three dimensions, the collocation points are taken to be a uniform $100\times100\times100$ tensor-product grid over the domain, including the boundary points. 
For higher-dimensional problems, the collocation points are generated by a Quasi--Monte Carlo method using Sobol sequences. 
The numbers of integration points are $8\times10^5$, $8\times10^5$, and $2\times10^6$ in 4D, 5D, and 6D, respectively. 
The numerical errors are evaluated in the same manner as in the continuous $L^2$-minimization setting.
The least–squares problems are solved using \texttt{scipy.linalg.lstsq} 
with its default driver \texttt{gelsd}, which is based on a singular value 
decomposition.

The same target function $u$ is used in dimensions $d=1$--$6$, and ReLU$^2$ networks are used for the task. 
The results are reported in Figure~\ref{fig:L2-regression}.  
The observed error decay is consistent with the theoretically guaranteed 
approximation rate of ReLU$^2$ shallow neural networks across all tested dimensions.

\begin{figure}[H]
    \centering
    \includegraphics[width=0.325\linewidth]{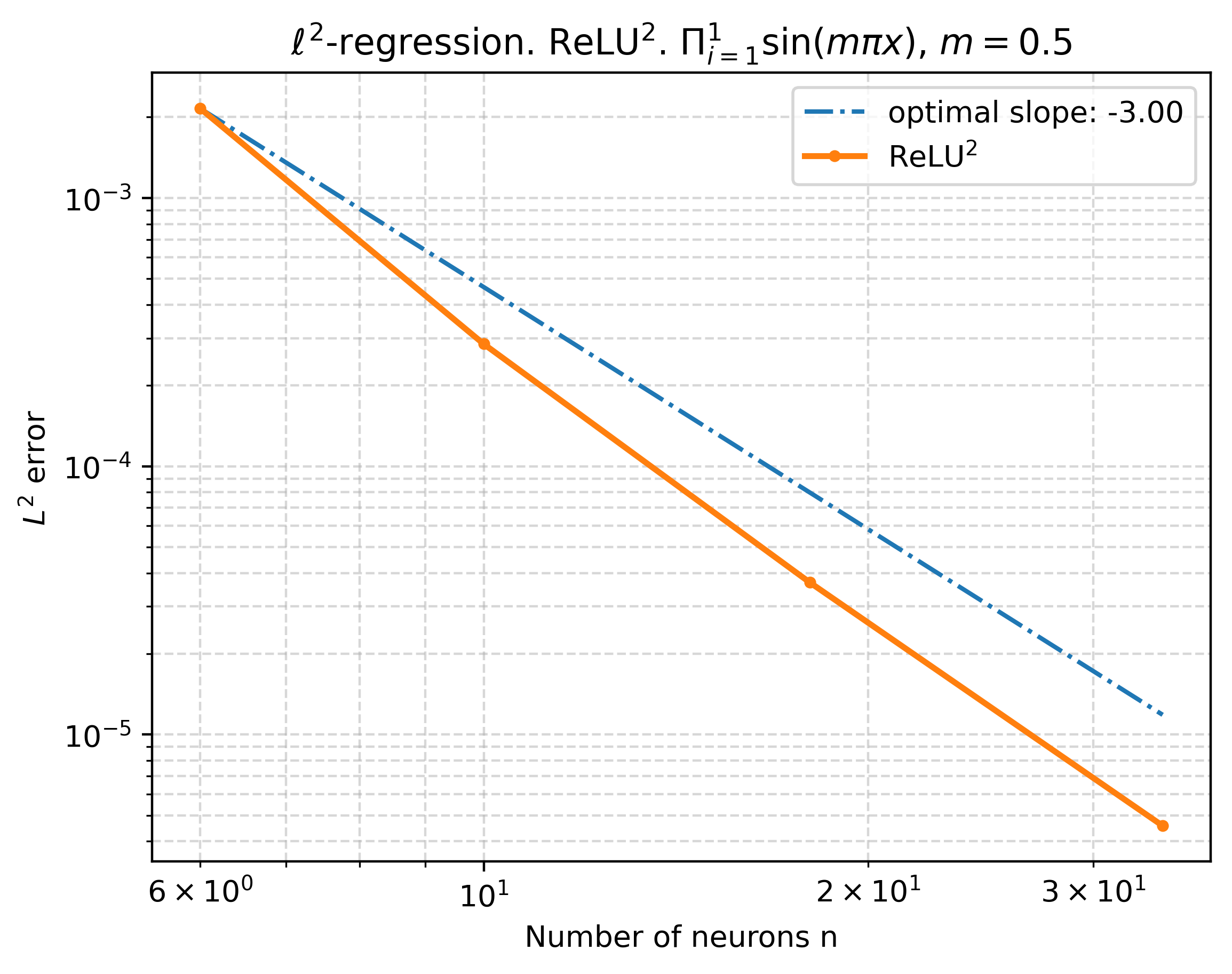}
    \includegraphics[width=0.325\linewidth]{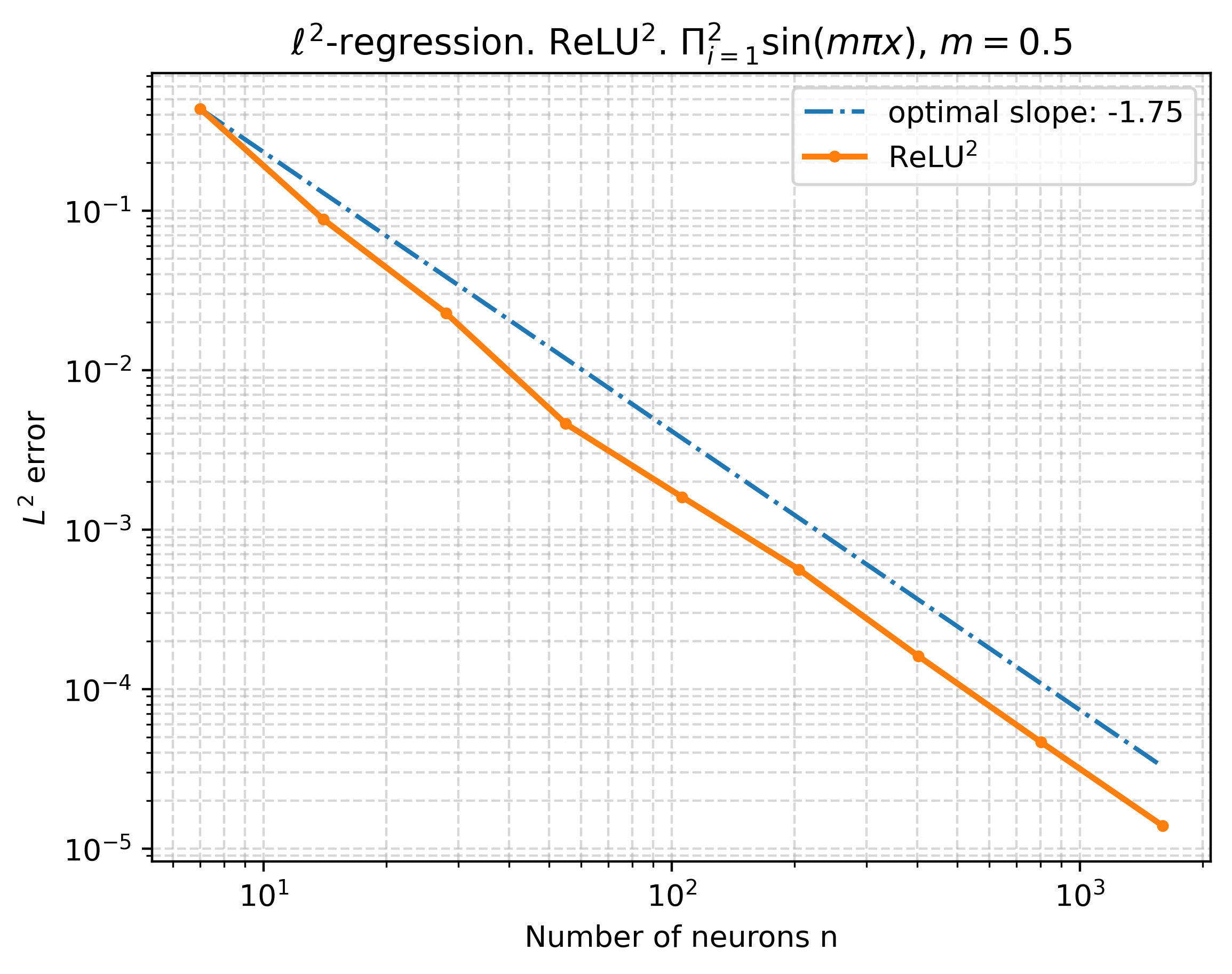}
    \includegraphics[width=0.325\linewidth]{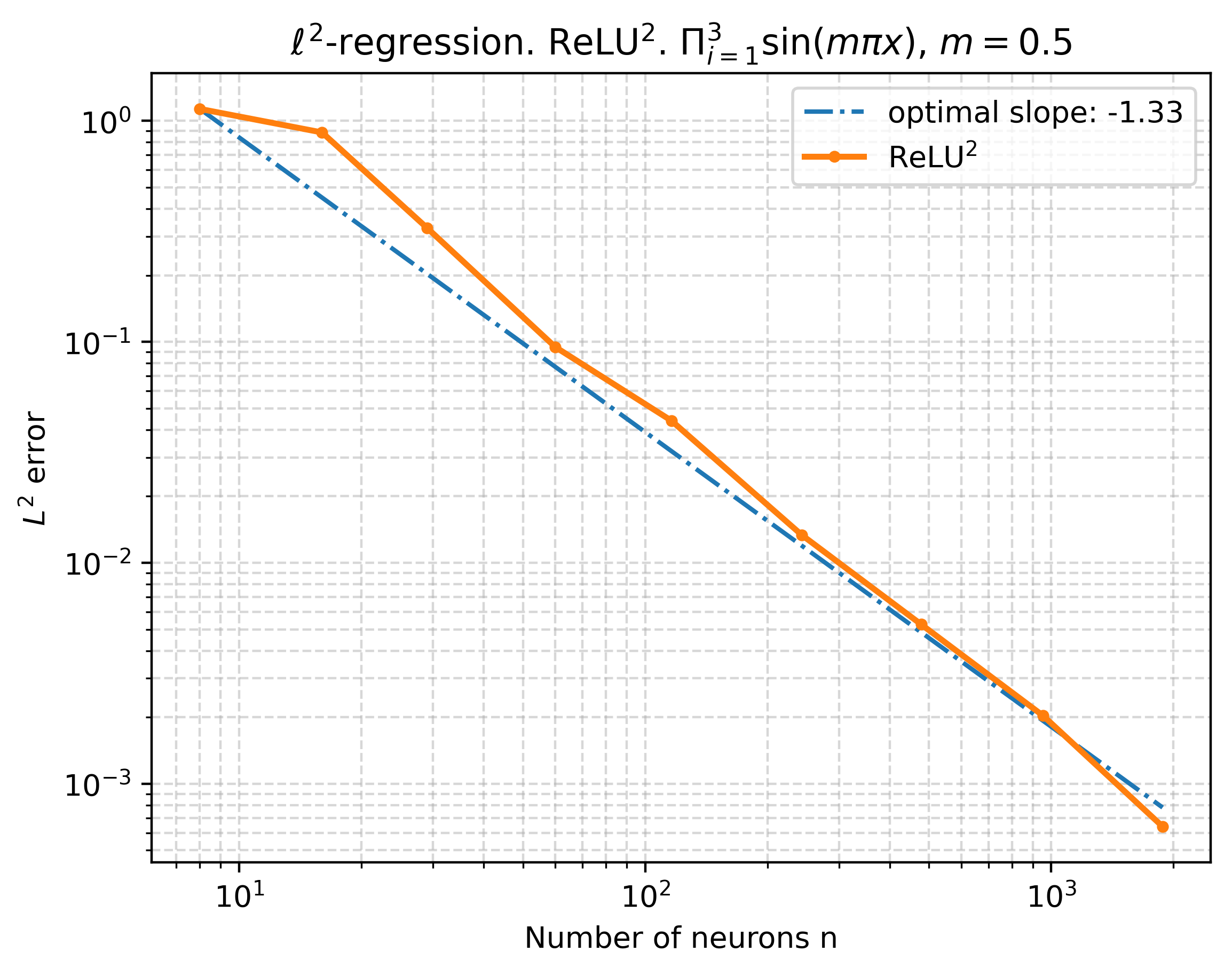} \\
    \includegraphics[width=0.325\linewidth]{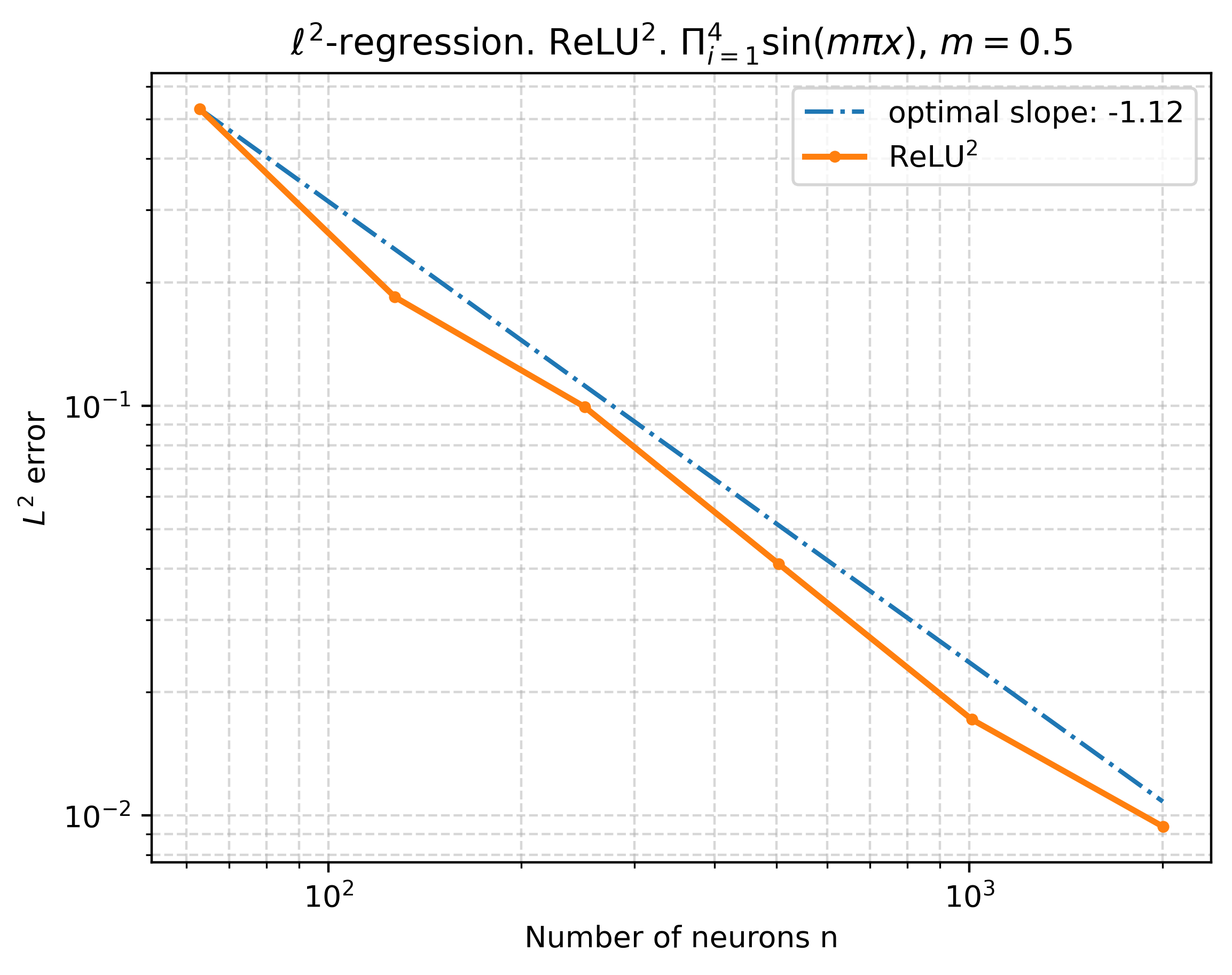}
    \includegraphics[width=0.325\linewidth]{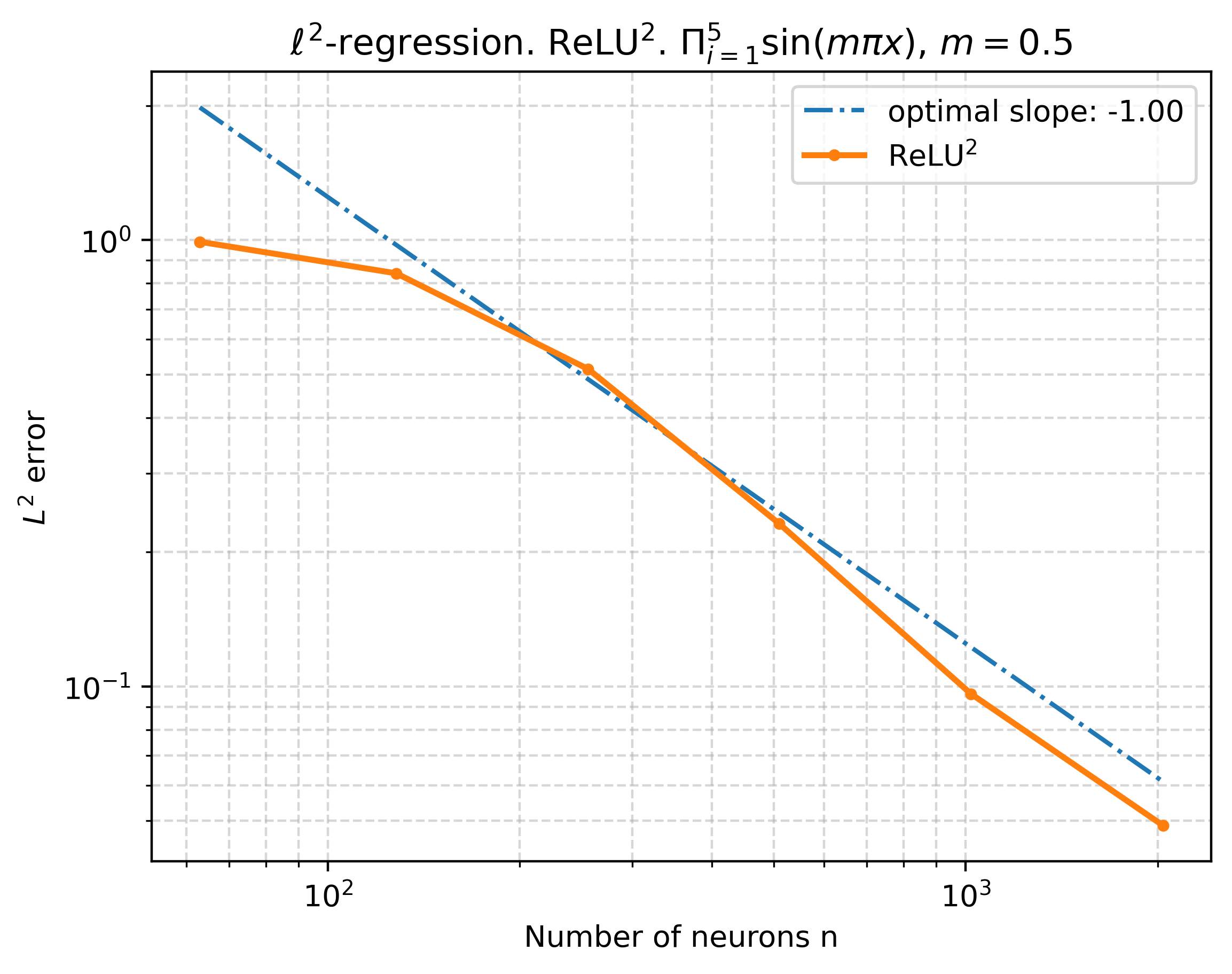}
    \includegraphics[width=0.325\linewidth]{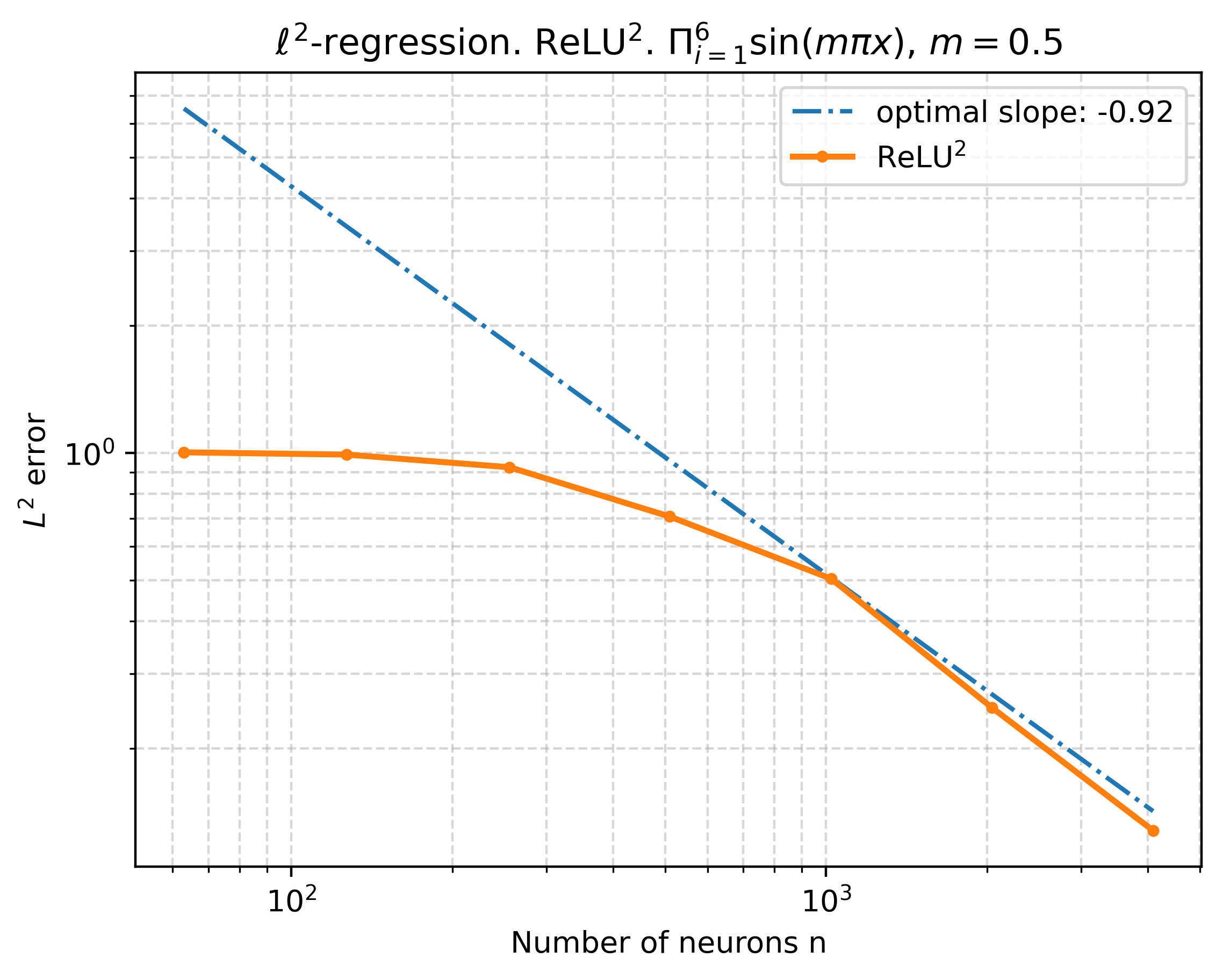} 
    \caption{Error decay for discrete $\ell^2$-regression (collocation formulation) 
    in dimensions $d=1$--$6$.}
    \label{fig:L2-regression}
\end{figure}

Additional experiments with ReLU activation are reported in Appendix~\ref{app:supp-l2}.  

\subsubsection{Condition numbers and numerical instability}
To better understand the numerical challenges associated with linearized neural networks of pre-determined neurons, we analyze the conditioning of the linear systems arising from the continuous $L^2$-minimization problems through examining the condition number of the mass matrix 
in different dimensions. 

Figures~\ref{fig:relu-condition} and~\ref{fig:relu2-condition} report the growth of condition numbers 
for ReLU and ReLU$^2$ activations, respectively. 
In all cases, the condition number increases rapidly with the number of neurons, and the fitted slopes (shown as dashed lines) are reported for reference. 

\begin{figure}[H]
    \centering
    \includegraphics[width=0.32\linewidth]{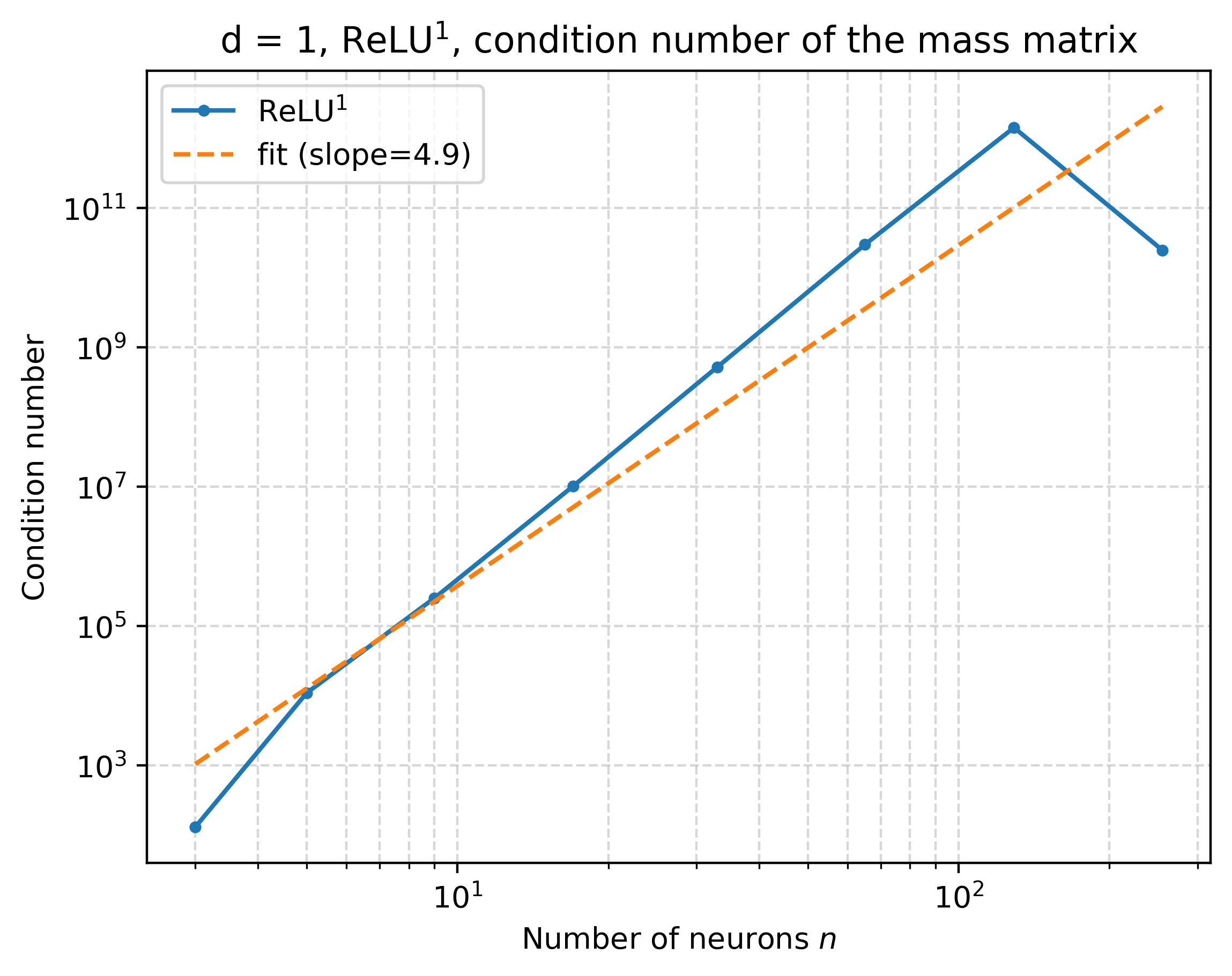}
    \includegraphics[width=0.32\linewidth]{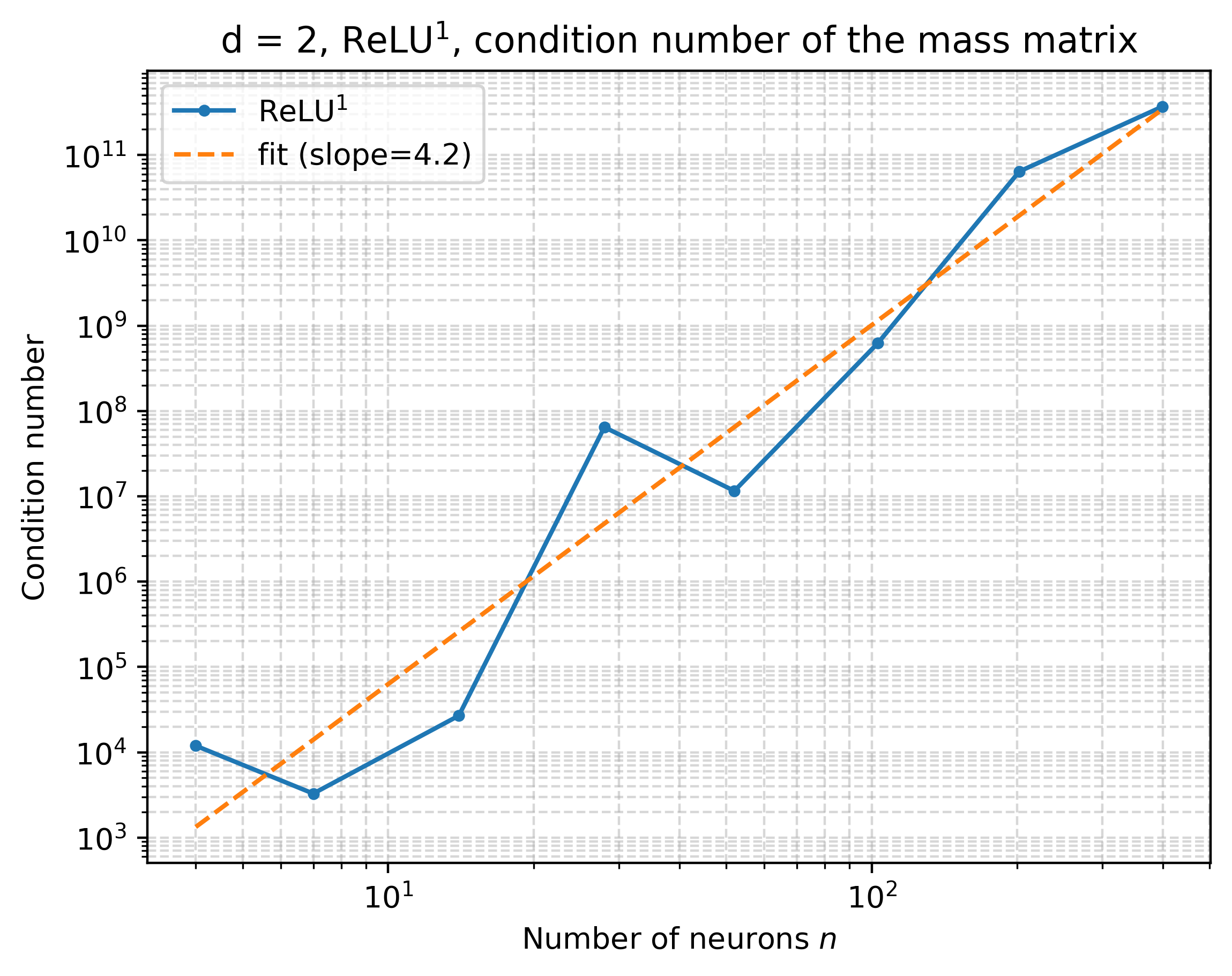}
    \includegraphics[width=0.32\linewidth]{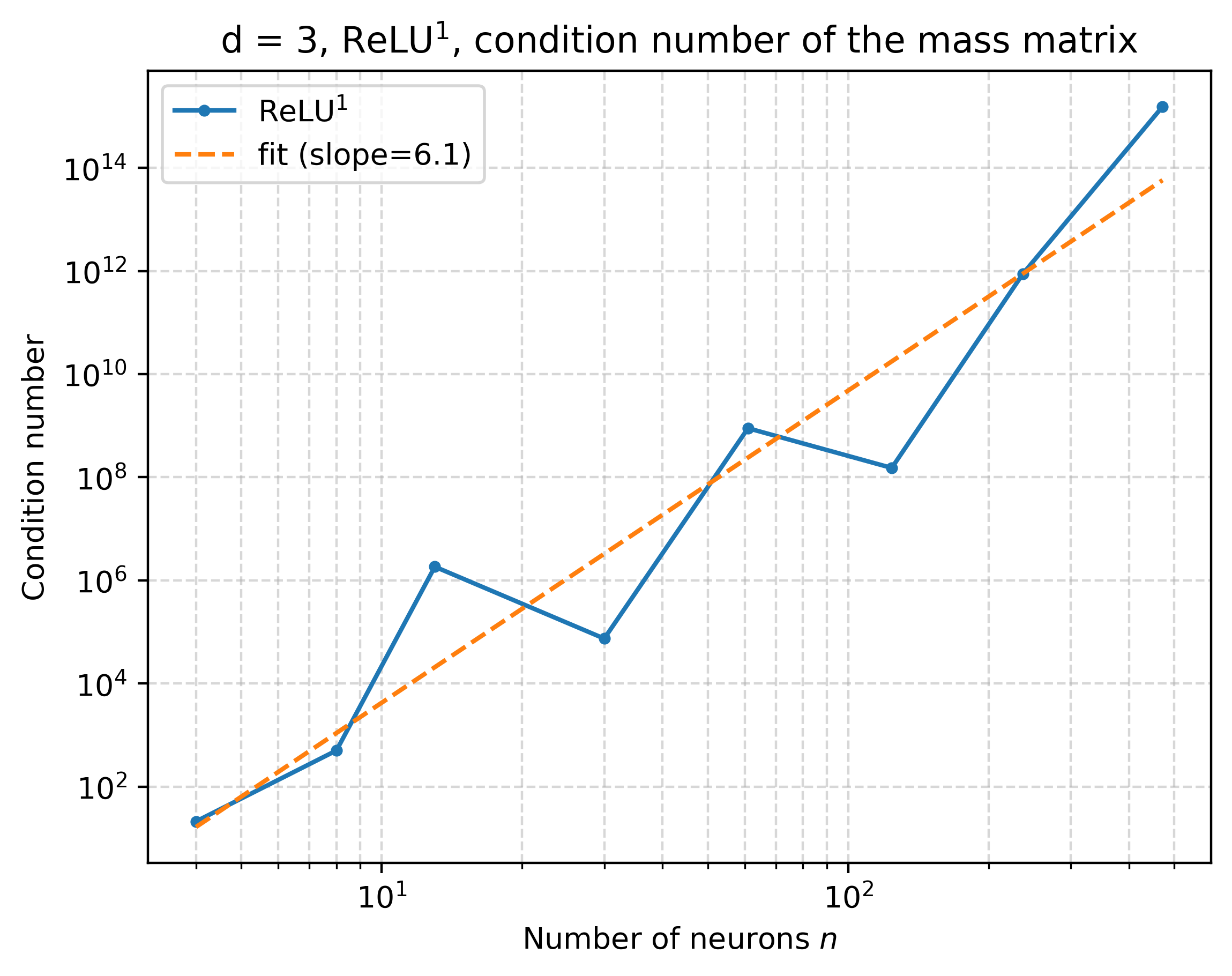} \\ 
    \includegraphics[width=0.32\linewidth]{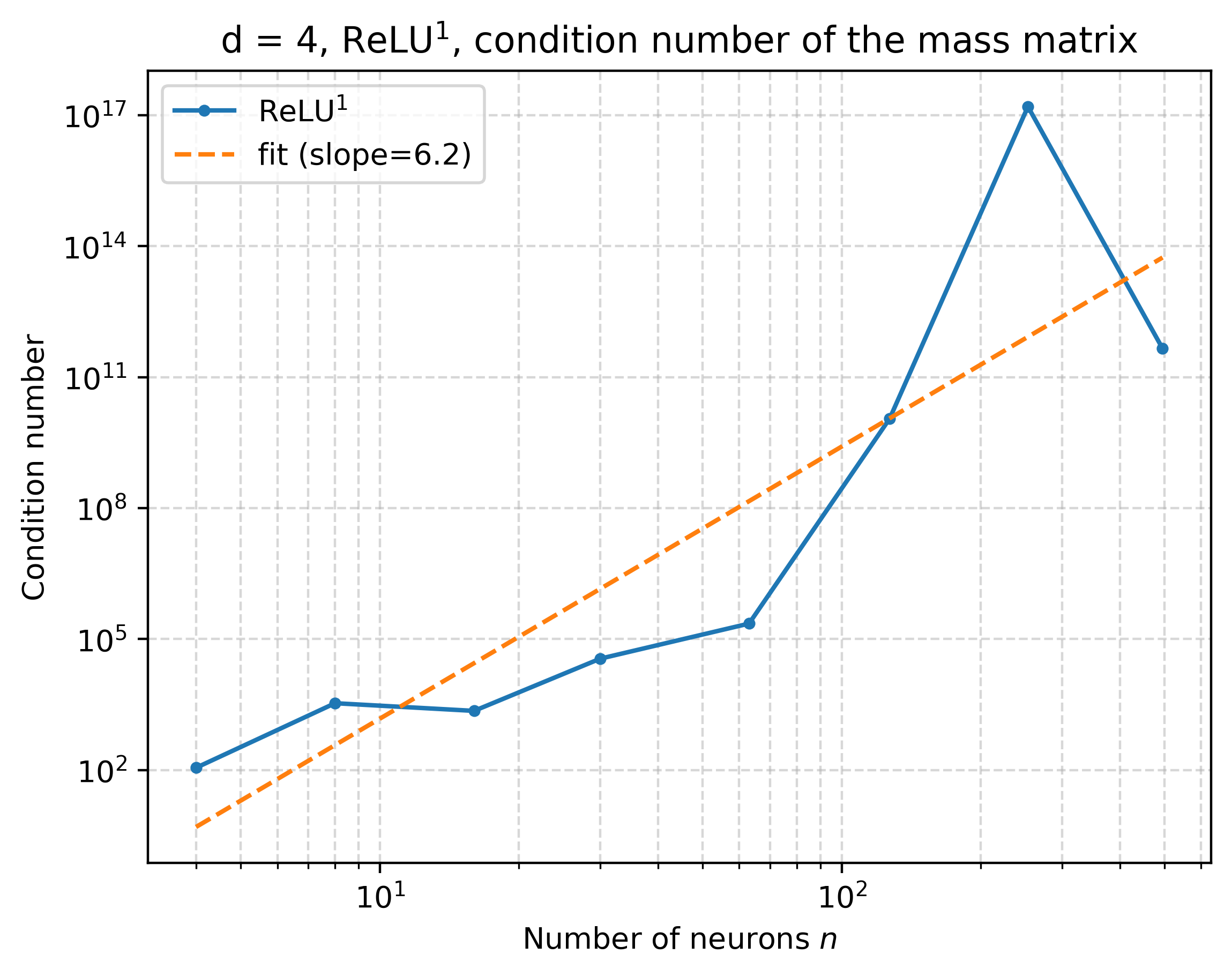}
    \includegraphics[width=0.32\linewidth]{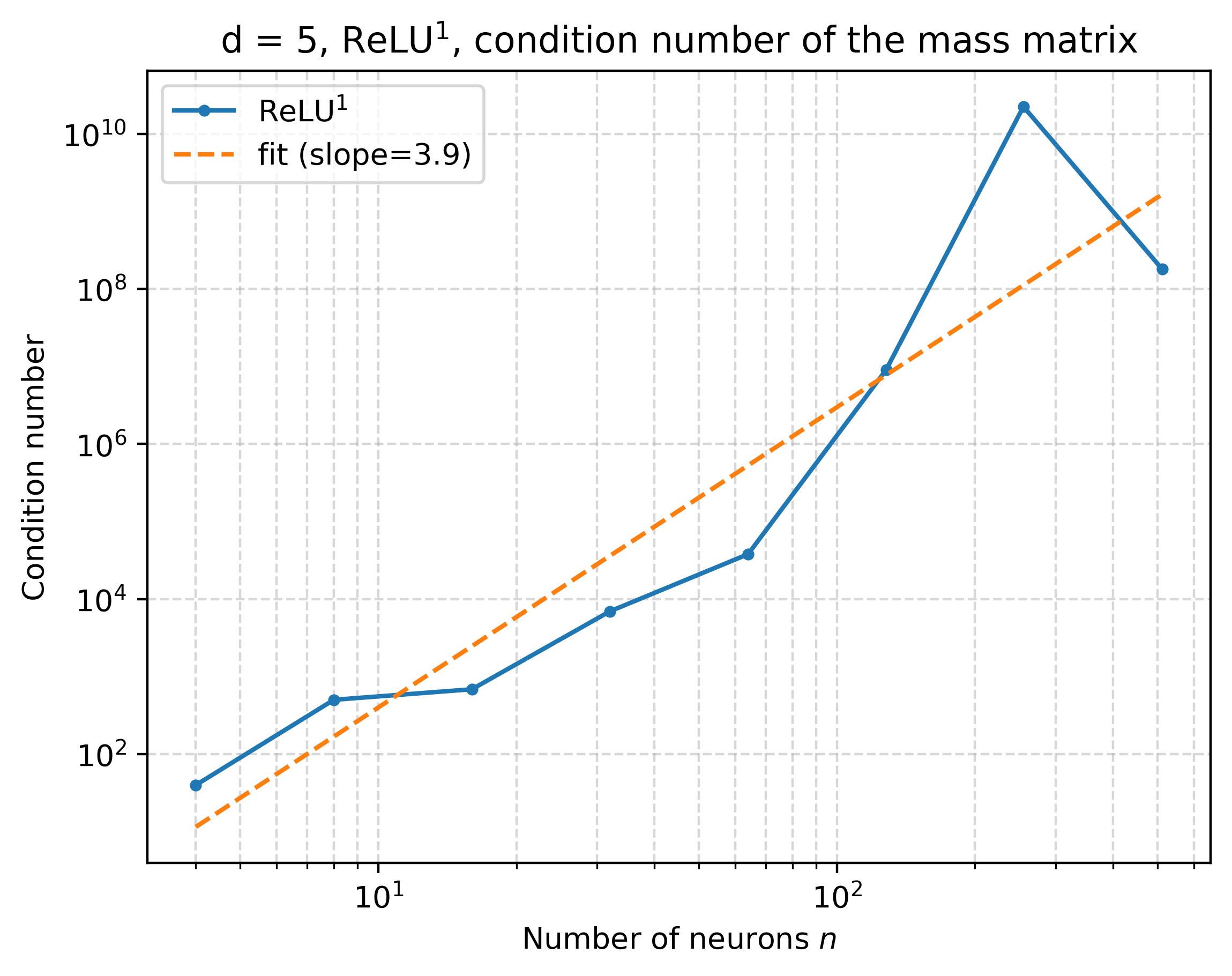}
    \includegraphics[width=0.32\linewidth]{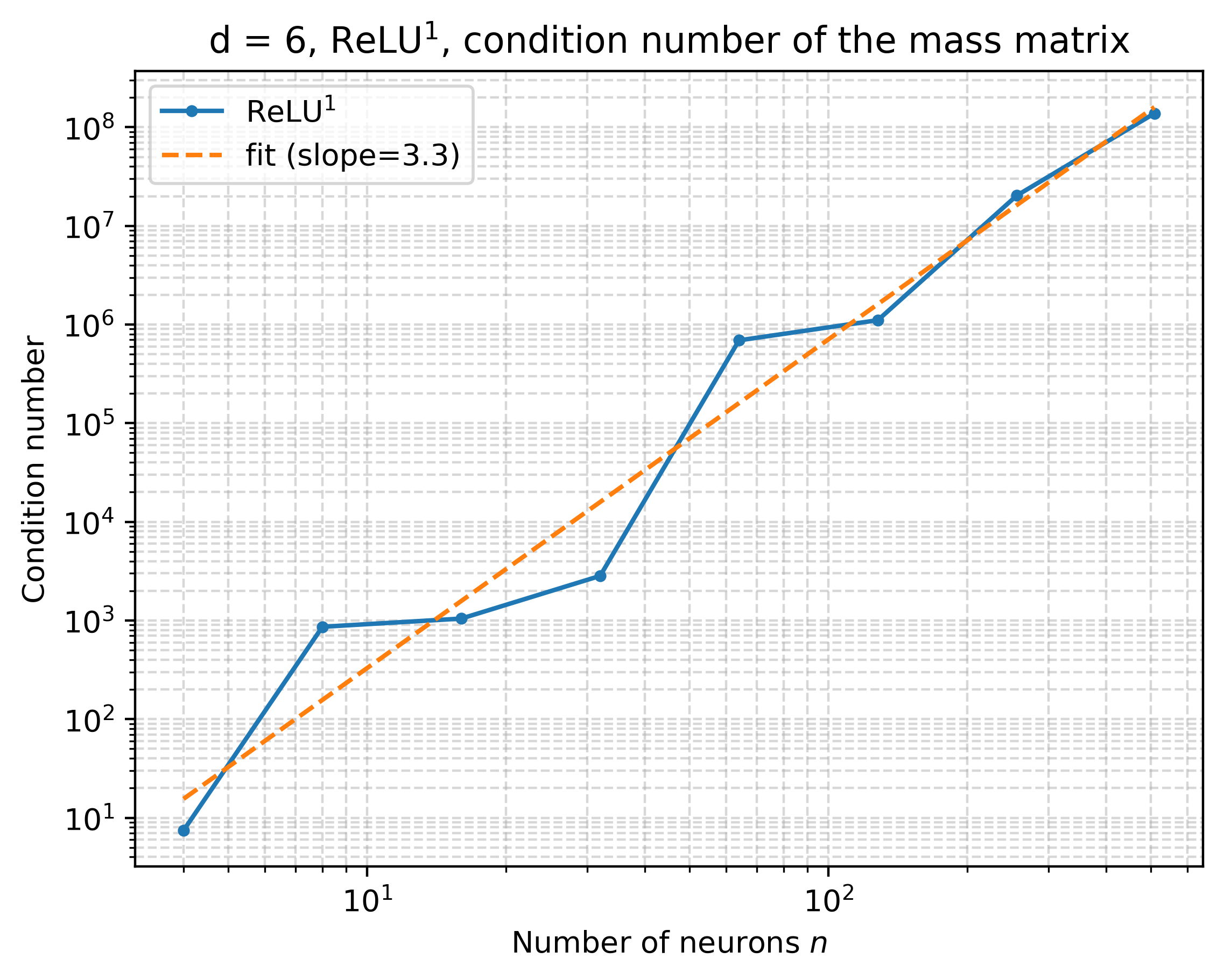} 
    \caption{Condition numbers of the mass matrices for ReLU$^1$ in dimensions $d=1$--$6$. 
    The dashed lines show least-squares fits on a log--log scale.}
    \label{fig:relu-condition}
\end{figure}

\begin{figure}[H]
    \centering
    \includegraphics[width=0.32\linewidth]{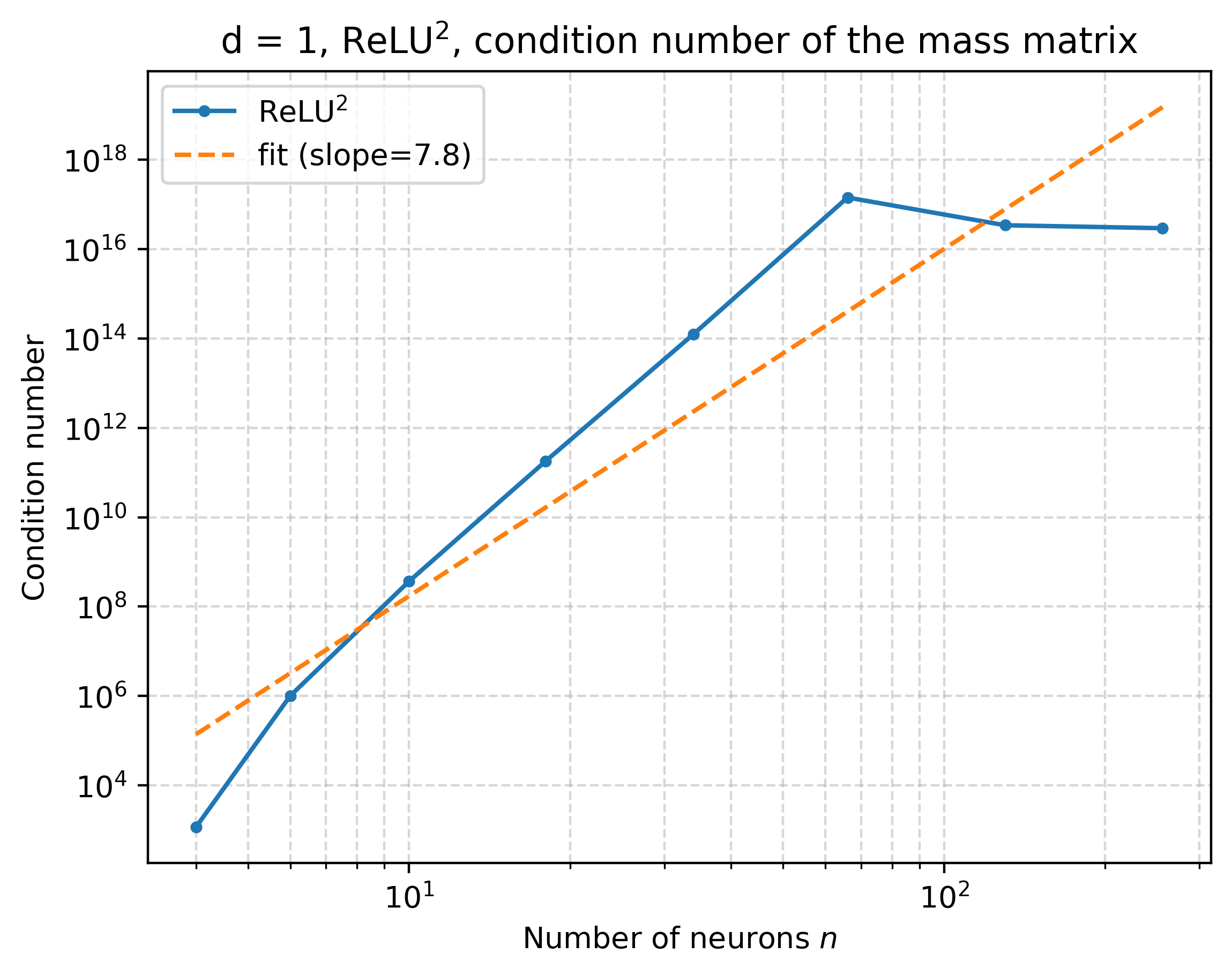}
    \includegraphics[width=0.32\linewidth]{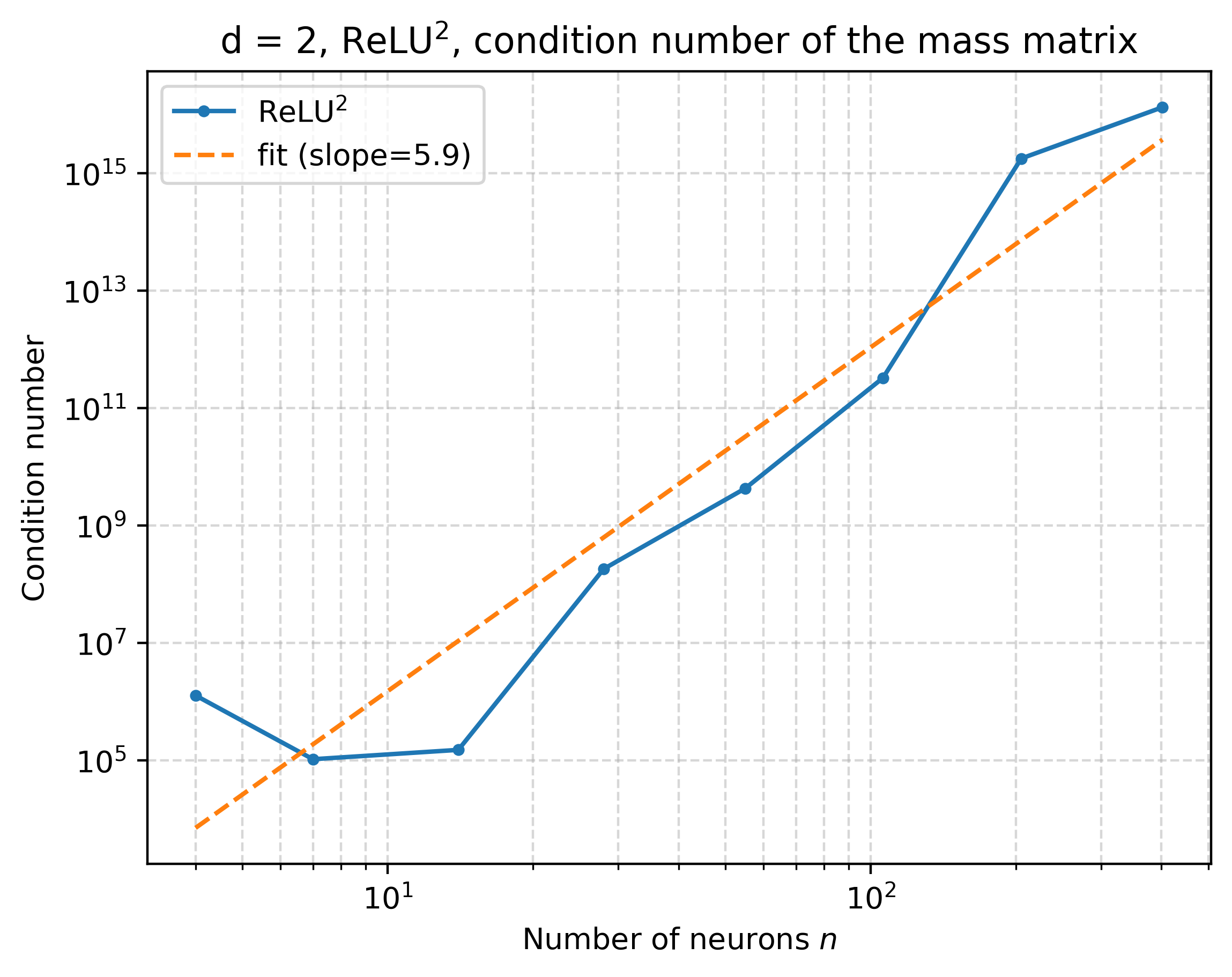}
    \includegraphics[width=0.32\linewidth]{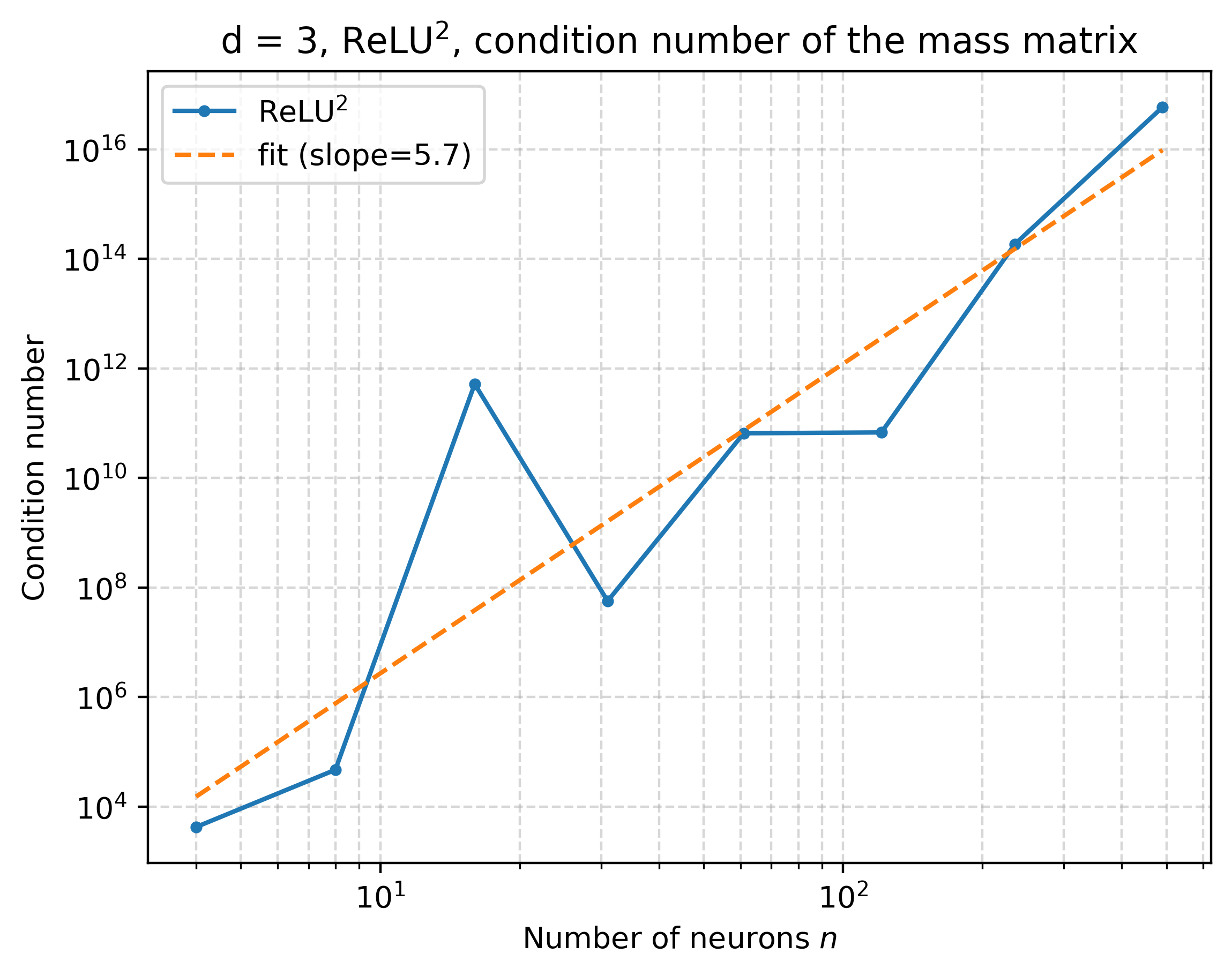} \\ 
    \includegraphics[width=0.32\linewidth]{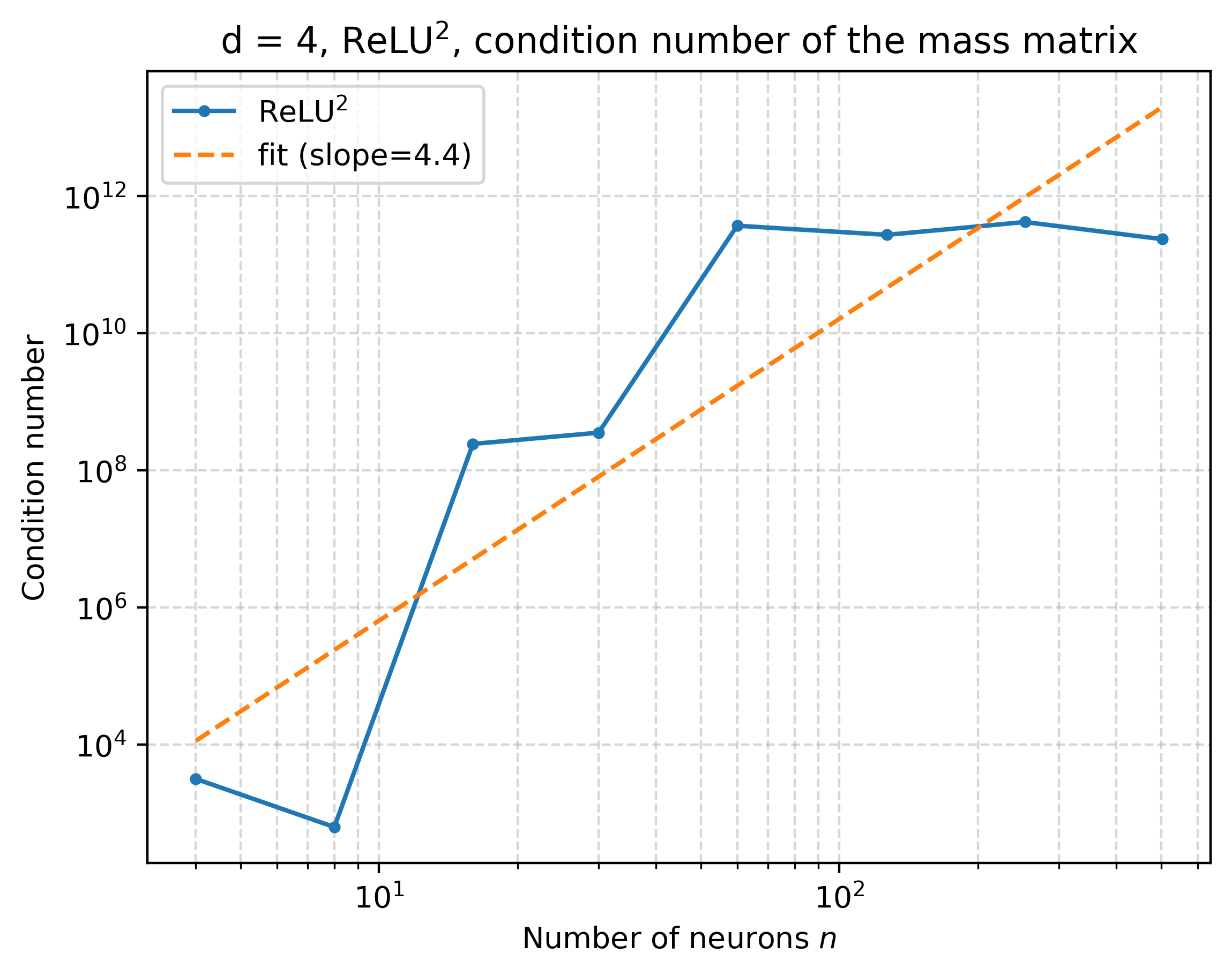}
    \includegraphics[width=0.32\linewidth]{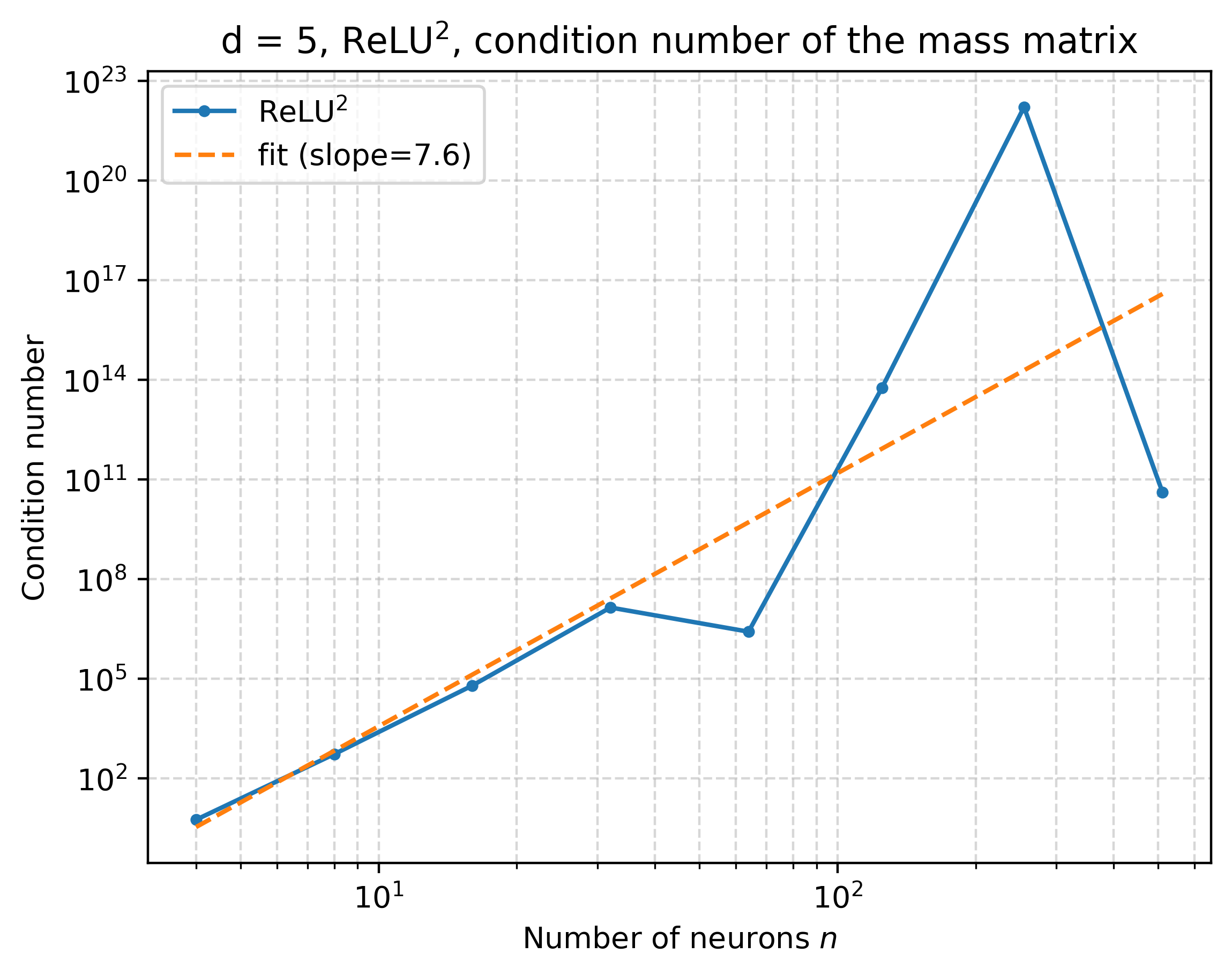}
    \includegraphics[width=0.32\linewidth]{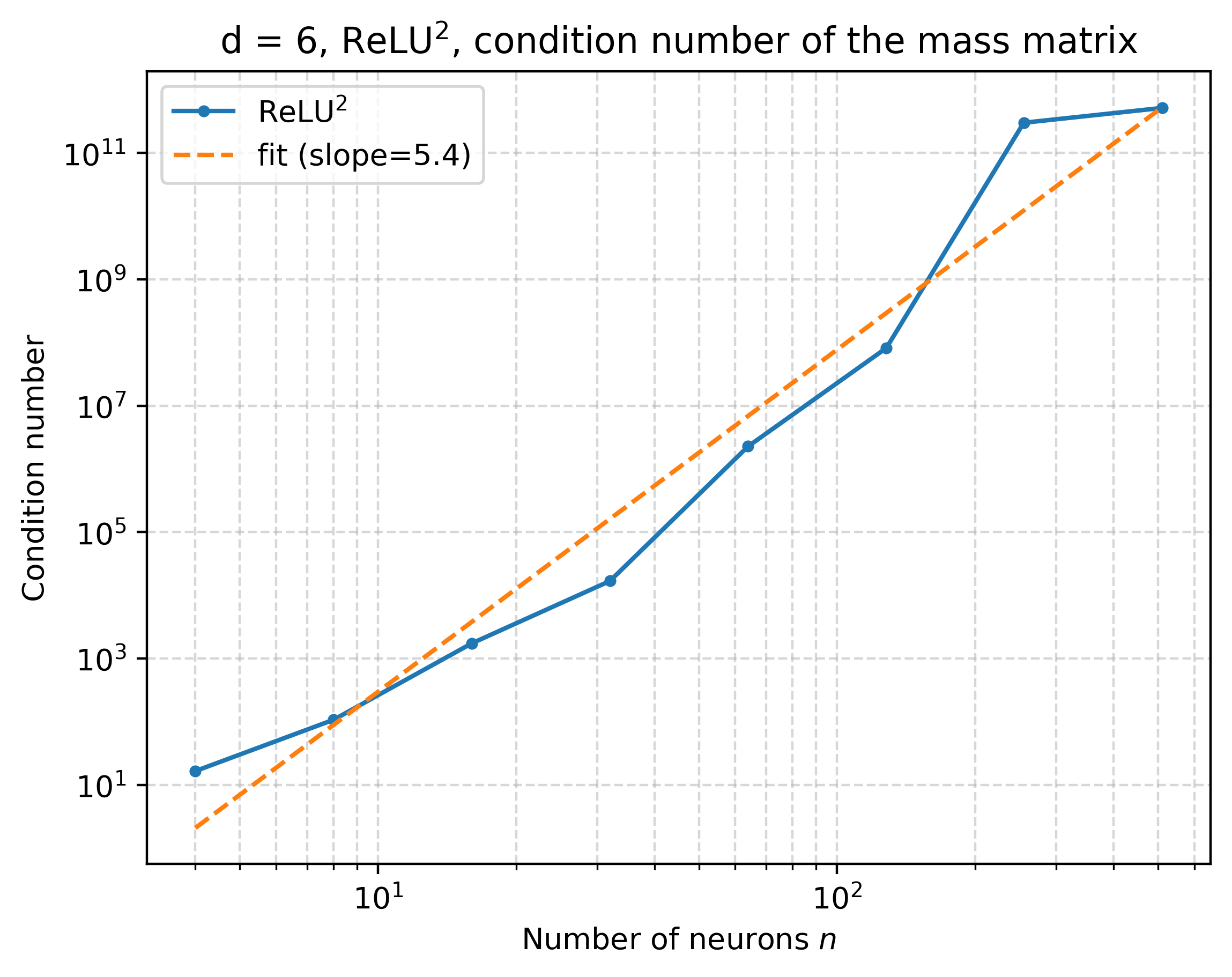} 
    \caption{Condition numbers of the mass matrices for ReLU$^2$ in dimensions $d=1$--$6$. 
    The dashed lines show least-squares fits on a log--log scale.}
    \label{fig:relu2-condition}
\end{figure}

In summary, we observe that the mass matrices for ReLU and ReLU$^2$ neural networks are ill-conditioned. 
This could lead to the instability in the continuous variational formulation, as the number of neurons increases. 
Indeed, in the continuous $L^2$-minimization setting, we observe that the numerical accuracy 
deteriorates once the number of neurons becomes large even if we use sufficiently many quadrature points to ensure accuracy of the numerical integration.  
Figure~\ref{fig:L2-minimization-instability} shows a representative one-dimensional example 
with ReLU$^2$ activation, where the error initially decreases at the expected rate but 
then exhibits irregular behavior as $n$ increases. 
This loss of convergence indicates the presence of numerical instability.
In the variational formulation, the error decay becomes unstable once the number of neurons $n$ exceeds a moderate threshold, due to severe ill-conditioning of the mass matrix. 

\begin{figure}[H]
    \centering
        \includegraphics[width=0.60\linewidth]{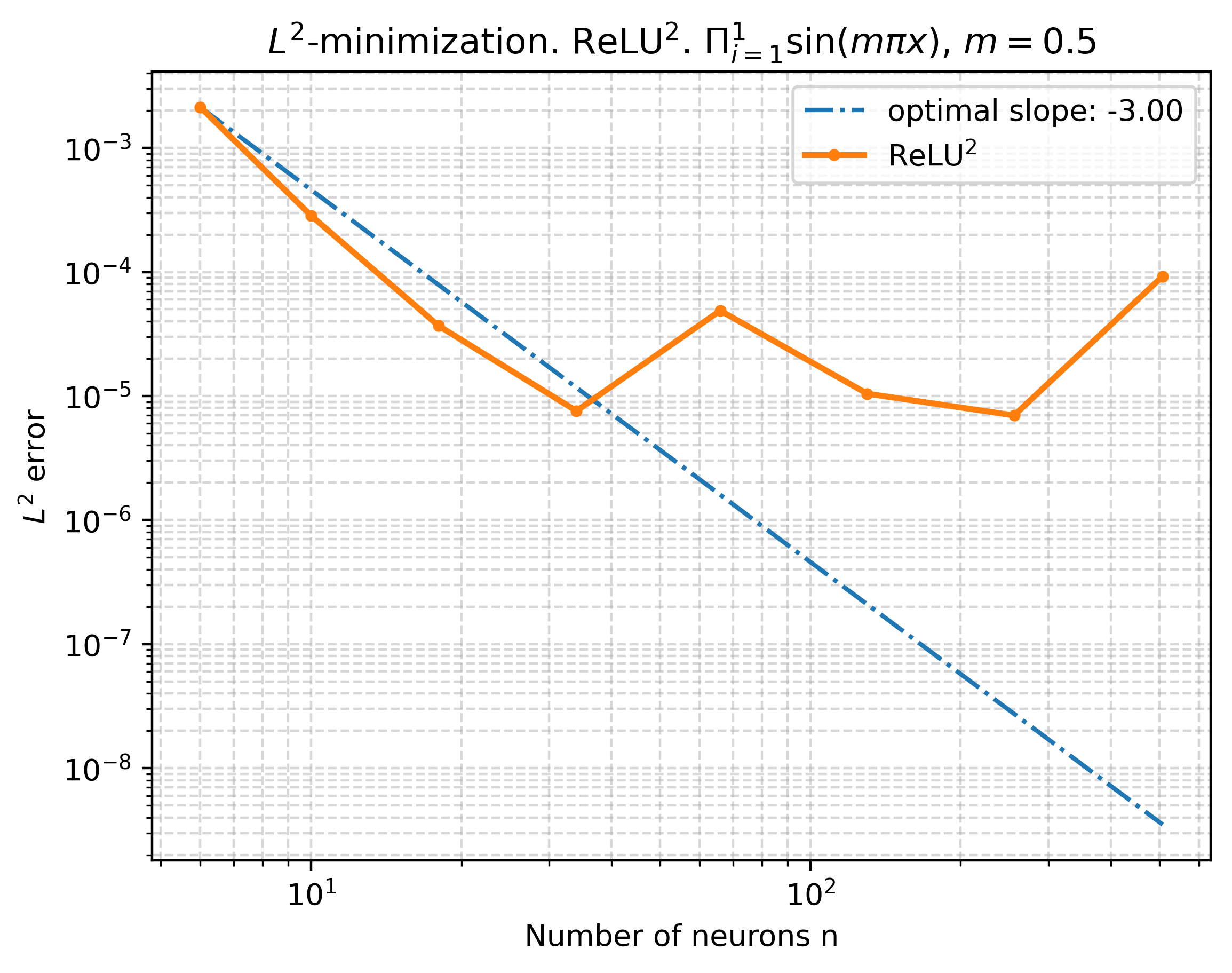}
    \caption{Error decay plots for continuous $L^2$-minimization with ReLU$^2$ in one dimension. 
    The error decay becomes unstable as the number of neurons increases.}
    \label{fig:L2-minimization-instability}
\end{figure}

To alleviate this issue, it is therefore important to design effective preconditioning strategies for such linear systems. 
Developing such preconditioners for linearized neural network bases is nontrivial, and rigorous results are currently limited. Existing optimal preconditioners are known only in the one-dimensional ReLU setting, see, for example, \cite{PXX:2022}.
For the $L^2$-minimization problem, one natural alternative is to solve the variational problem in a least-squares form rather than assembling and solving the normal equations directly.
The details are described in Appendix \ref{app:least-square}.
We next demonstrate, through a simple one-dimensional example, that this variational least-squares formulation leads to slightly improved numerical stability.

Specifically, we compare two approaches for solving the $L^2$-minimization problem using a one-dimensional ReLU$^2$ shallow neural network. The first approach is the direct assembly and
solution of the mass matrix arising from the variational formulation, while the second approach
solves the same variational problem using a weighted least-squares formulation. Both approaches
employ identical quadrature rules and the same set of sampling points. Therefore, any discrepancy
in numerical performance can be attributed solely to the conditioning of the resulting linear
systems rather than to differences in discretization.

The numerical results are shown in Figure~\ref{fig:L2min-compare-1drelu2}. 
We observe that the variational least-squares formulation yields slightly more stable numerical accuracy as the number of neurons increases, compared to direct mass matrix assembly.

\begin{figure}[h]
    \centering
    \includegraphics[width=0.6\linewidth]{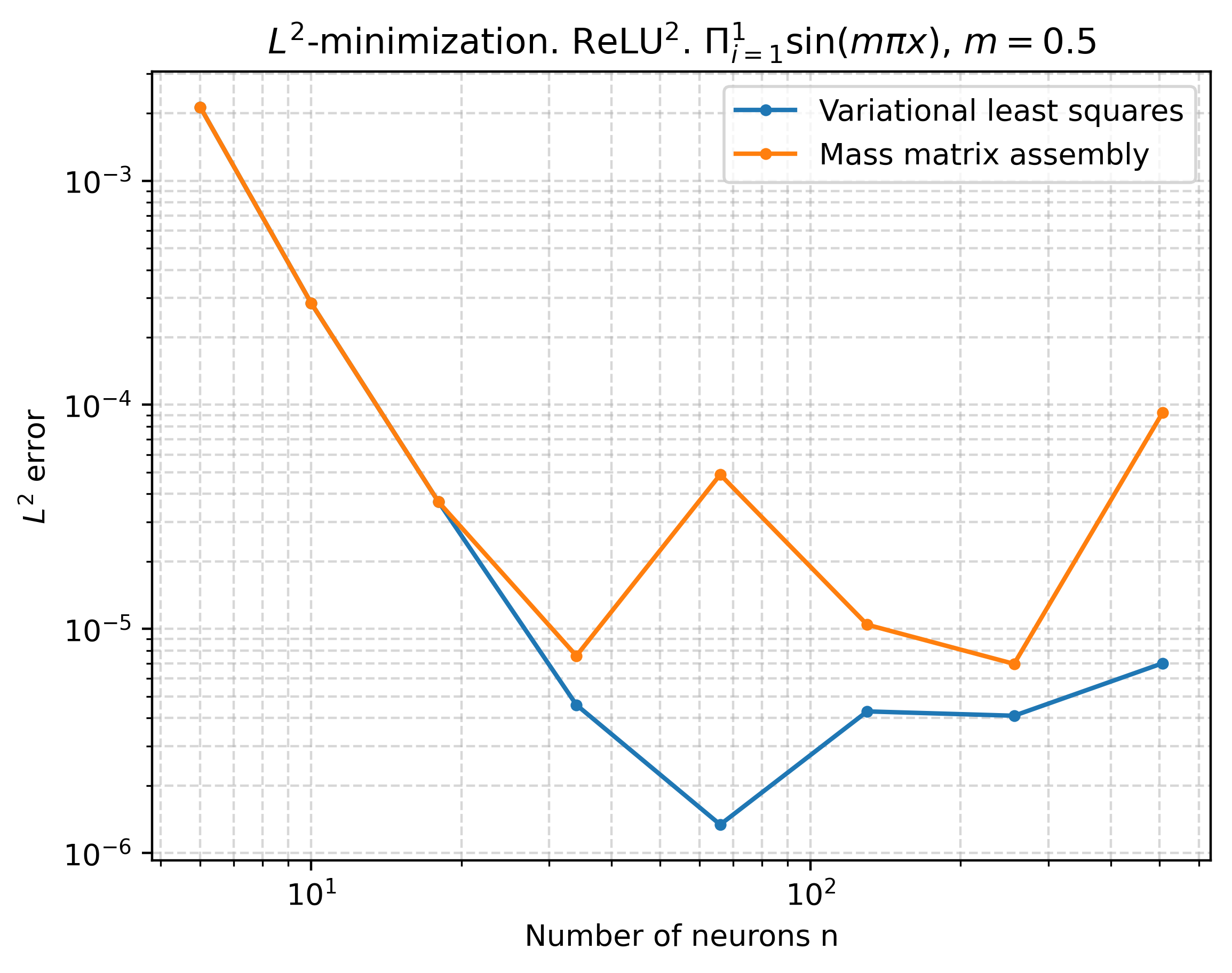}
    \caption{Comparison of mass matrix assembly and variational least–squares formulations for 1D ReLU$^2$ networks. }
    \label{fig:L2min-compare-1drelu2}
\end{figure}


\subsubsection{Second-order elliptic PDEs}
We next consider the Neumann problem
\begin{equation}
\label{eq:Neumann}
\begin{aligned}
        -\Delta u +  u & = f \quad \text{ in } \Omega = (-1,1)^d,\\ 
        \frac{\partial u}{\partial n} & = 0 \quad \text{ on } \partial \Omega. 
\end{aligned}    
\end{equation} 
For convenience, we impose the zero Neumann boundary condition and choose the exact solution $ u(x) = \prod_{i=1}^d \sin\!\left(\tfrac{\pi x_i}{2}\right) $.
This function satisfies the Neumann boundary condition, with
$\|u\|_{L^2(\Omega)} = 1$, $\|\nabla u\|_{L^2(\Omega)} = \tfrac{\pi}{2}\sqrt{d}$
for all $d$. 

We solve the equation using ReLU$^k$ shallow neural networks 
with predetermined neurons under the variational formulation, that is, we find a numerical solution $u_n \in L_n^k$ such that it satisfies the variational problem of the form \eqref{eq:variational-neumann}. 
The way to construct a quasi-uniform grid points on a sphere $S^d$ for this problem is same as that for the $L^2$-minimization problem. 

For numerical integration, in 2D, we divide the square domain into $100^2$ uniform subdomains and use a Gauss quadrature of order $3$ in each subdomain.
In 3D, we divide the cubic domain into $50^3$ uniform subdomains and use a Gauss quadrature of order $3$ in each subdomain.
In higher dimensions, we use the Quasi-Monte Carlo method with integration points generated by Sobol sequences.
The numbers of integration points are $1\times10^6$, $6\times10^6$, and $5\times10^7$ in 4D, 5D, and 6D, respectively.

The results are shown in Tables \ref{tab:Neuman-RFM-ex1-2d}, \ref{tab:Neuman-RFM-ex1-3d}, \ref{tab:Neuman-RFM-ex1-4d}, \ref{tab:Neuman-RFM-ex1-5d} and \ref{tab:Neuman-RFM-ex1-6d}.

    \begin{table}[H]
        \centering
        \begin{tabular}{c|c|c|c|c}
$n$ & 	 $\|u-u_n \|_{L^2}$ & 	 order $O(n^{-2.25})$  & $ |u-u_n |_{H^1}$ & 	 order $O(n^{-1.75})$\\ \hline \hline 
88 		 &  5.542e-04 &  		*   &  8.075e-03 &  *   \\ \hline  
165 		 &  8.224e-05 &  		 3.04  &  1.905e-03 &  2.30  \\ \hline  
318 		 &  2.070e-05 &  		 2.10  &  6.226e-04 &  1.70  \\ \hline  
631 		 &  4.214e-06 &  		 2.32  &  1.839e-04 &  1.78  \\ \hline  
1256 		 &  8.508e-07 &  		 2.32  &  5.139e-05 &  1.85  \\ \hline  
        \end{tabular}
        \caption{2D Neumann problem solved by the variational method. ReLU$^3$ shallow neural network with predetermined neurons.}
        \label{tab:Neuman-RFM-ex1-2d}
    \end{table}

      \begin{table}[H]
        \centering
        \begin{tabular}{c|c|c|c|c}
$n$ & 	 $\|u-u_n \|_{L^2}$ & 	 order $O(n^{-1.67})$ & $ |u-u_n |_{H^1}$ & 	 order $O(n^{-1.33})$\\ \hline \hline 
95 		 &  3.859e-02 &  		 *  &  2.930e-01 &   * \\ \hline  
192 		 &  8.457e-03 &  		 2.16  &  8.397e-02 &  1.78  \\ \hline  

383 		 &  2.164e-03 &  		 1.97  &  2.769e-02 &  1.61  \\ \hline  

751 		 &  6.656e-04 &  		 1.75  &  1.085e-02 &  1.39  \\ \hline  

1481 		 &  1.876e-04 &  		 1.86  &  3.961e-03 &  1.48  \\ \hline  
        \end{tabular}
        \caption{3D Neumann problem solved by the variational method. ReLU$^3$ shallow neural network with predetermined neurons. }
        \label{tab:Neuman-RFM-ex1-3d}
    \end{table}
    
      \begin{table}[H]
        \centering
        \begin{tabular}{c|c|c|c|c}
$n$ & 	 $\|u-u_n \|_{L^2}$ & 	 order $O(n^{-1.375})$ & $ |u-u_n |_{H^1}$ & 	 order $O(n^{-1.125})$\\ \hline \hline 
95 		 &  3.122e-01 &  		*  &  1.500e+00 &  *\\ \hline  
199 		 &  9.747e-02 &  		 1.57  &  6.570e-01 &  1.12  \\ \hline  
393 		 &  3.315e-02 &  		 1.58  &  2.883e-01 &  1.21  \\ \hline  
791 		 &  1.109e-02 &  		 1.57  &  1.190e-01 &  1.27  \\ \hline  
1575 		 &  4.142e-03 &  		 1.43  &  5.322e-02 &  1.17  \\ \hline 
        \end{tabular}
        \caption{4D Neumann problem solved by the variational method. ReLU$^3$ shallow neural network with predetermined neurons.}
        \label{tab:Neuman-RFM-ex1-4d}
    \end{table}

      \begin{table}[H]
        \centering
        \begin{tabular}{c|c|c|c|c}
$n$ & 	 $\|u-u_n \|_{L^2}$ & 	 order $O(n^{-1.2})$ & $ |u-u_n |_{H^1}$ & 	 order $O(n^{-1.0})$\\ \hline \hline 
99 		 &  9.353e-01 &  		*  &  3.356e+00 &  *   \\ \hline  
199 		 &  6.036e-01 &  		 0.63  &  2.533e+00 &  0.40  \\ \hline  
399 		 &  2.408e-01 &  		 1.32  &  1.367e+00 &  0.89  \\ \hline  
798 		 &  7.664e-02 &  		 1.65  &  5.802e-01 &  1.24  \\ \hline  
1590 		 &  3.730e-02 &  		 1.04  &  3.400e-01 &  0.78  \\ \hline  
3182 		 &  1.582e-02 &  		 1.24  &  1.751e-01 &  0.96  \\ \hline 
        \end{tabular}
        \caption{5D Neumann problem solved by the variational method. ReLU$^3$ shallow neural network with predetermined neurons. }
        \label{tab:Neuman-RFM-ex1-5d}
    \end{table}
    
      \begin{table}[H]
        \centering
        \begin{tabular}{c|c|c|c|c}
$n$ & 	 $\|u-u_n \|_{L^2}$ & 	 order $O(n^{-1.08})$  & $ |u-u_n |_{H^1}$ & 	 order $O(n^{-0.92})$\\ \hline \hline 

100 		 &  9.983e-01 &  		 *  &  3.842e+00 &  *  \\ \hline  
200 		 &  9.878e-01 &  		 0.02  &  3.811e+00 &  0.01  \\ \hline  

399 		 &  8.384e-01 &  		 0.24  &  3.422e+00 &  0.16  \\ \hline  

798 		 &  4.036e-01 &  		 1.05  &  2.108e+00 &  0.70  \\ \hline  

1598 		 &  1.688e-01 &  		 1.26  &  1.146e+00 &  0.88  \\ \hline  

3194 		 &  7.788e-02 &  		 1.12  &  6.421e-01 &  0.84  \\ \hline  
        \end{tabular}
        \caption{6D Neumann problem solved by the variational method. ReLU$^3$ shallow neural network with predetermined neurons. }
        \label{tab:Neuman-RFM-ex1-6d}
    \end{table}

\subsubsection{Deterministic features using quasi-Monte Carlo points in high dimensions}
For higher dimensions ($d \geq 3$), constructing a quasi--uniform grid on $S^d$ 
is nontrivial. 
In the previous subsections, for convenience, we fix the hidden layer parameters by sampling from a uniform distribution on the sphere $S^d$. 
While simple to implement, purely random sampling may produce clusters or gaps, 
and hence does not guarantee quasi--uniform coverage of the sphere.  
To alleviate this difficulty, we propose to use quasi--Monte Carlo (QMC) point sets on the sphere to determine the hidden layer parameters. 
The use of QMC sequences ensures a more uniform distribution than purely random sampling.  
The procedure to generate QMC points uniformly distributed on $S^d$ is described below.



\begin{enumerate}
    \item \textbf{Generate Low-Discrepancy Points:} \\
    Generate points in the unit cube \([0,1]^{d+1}\) using a QMC sequence, e.g., Sobol or Halton sequences.

    \item \textbf{Clip the Points:} \\
    Clip each coordinate \(u\) to the interval \([\varepsilon, 1-\varepsilon]\) to avoid boundary values that yield infinities, where \(\varepsilon\) is a small positive constant.

    \item \textbf{Map to the Standard Normal Distribution:} \\
    Transform each coordinate using the inverse cumulative distribution function (CDF) of the standard normal distribution:
    \[
    z = \Phi^{-1}(u),
    \]
    where \(\Phi^{-1}\) is the quantile function of the standard Gaussian.

    \item \textbf{Project onto the Unit Sphere:} \\
    Normalize the resulting vector \(\mathbf{z} \in \mathbb{R}^{d+1}\) by its Euclidean norm:
    \[
    \mathbf{x} = \frac{\mathbf{z}}{\|\mathbf{z}\|_2}.
    \]
    
    \item \textbf{Final Result:} \\
    The points \(\mathbf{x}\) are approximately uniformly distributed on the sphere \(S^d\) and inherit the low-discrepancy properties from the original QMC sequence.
\end{enumerate}

Using this approach, we solve the Neumann problem in dimensions three and five. 
The corresponding numerical results are reported in 
Tables~\ref{tab:Neuman-RFM-ex1-3d-qmc} and~\ref{tab:Neuman-RFM-ex1-6d-qmc}. 
We use the QMC points generated by the Sobol sequence. 
For both cases, we also compare against the baseline where hidden layer parameters 
are chosen by independent random sampling. 
The comparisons are shown in Figures~\ref{fig:3dNeumann-RFM-qmc} and~\ref{fig:5dNeumann-RFM-qmc}. 
The results indicate that employing QMC points for the hidden parameters improves the numerical accuracy relative to purely random sampling.

\begin{table}[H]
        \centering
        \begin{tabular}{c|c|c|c|c}
$n$ & 	 $\|u-u_n \|_{L^2}$ & 	 order $O(n^{-1.67})$ & $ |u-u_n |_{H^1}$ & 	 order $O(n^{-1.33})$\\ \hline \hline 
23 		 & 4.217e-01 &		 * & 1.550e+00 &  *  \\ \hline 
47 		 &  1.196e-01 &  		 1.76  &  6.404e-01 &  1.24  \\ \hline  
96 		 &  2.883e-02 &  		 1.99  &  2.212e-01 &  1.49  \\ \hline  
191 		 &  7.638e-03 &  		 1.93  &  7.680e-02 &  1.54  \\ \hline  
385 		 &  1.706e-03 &  		 2.14  &  2.228e-02 &  1.77  \\ \hline  
750 		 &  4.588e-04 &  		 1.97  &  7.994e-03 &  1.54  \\ \hline  
1483 		 &  1.526e-04 &  		 1.61  &  3.332e-03 &  1.28  \\ \hline 
        \end{tabular}
        \caption{3D Neumann problem solved by the variational method. ReLU$^3$ shallow neural network with predetermined neurons constructed from QMC points.  Numerical integration: $50^3$, order $3$.   }
        \label{tab:Neuman-RFM-ex1-3d-qmc}
    \end{table}

      \begin{table}[H]
        \centering
        \begin{tabular}{c|c|c|c|c}
$n$ & 	 $\|u-u_n \|_{L^2}$ & 	 order $O(n^{-1.2})$ & $ |u-u_n |_{H^1}$ & 	 order $O(n^{-1.0})$\\ \hline \hline 


98 		 &  9.074e-01 &  	*  &  3.288e+00 &  * \\ \hline  
197 		 &  5.577e-01 &  		 0.70  &  2.414e+00 &  0.44  \\ \hline  
397 		 &  2.025e-01 &  		 1.45  &  1.218e+00 &  0.98  \\ \hline  
794 		 &  7.349e-02 &  		 1.46  &  5.689e-01 &  1.10  \\ \hline  
1591 		 &  3.304e-02 &  		 1.15  &  3.038e-01 &  0.90  \\ \hline  
3185 		 &  1.428e-02 &  		 1.21  &  1.579e-01 &  0.94  \\ \hline
        \end{tabular}
        \caption{5D Neumann problem solved by the variational method. ReLU$^3$ shallow neural network with predetermined neurons constructed from QMC points. Numerical integration: $6 \times 10^6$ quasi-Monte Carlo points.   }
        \label{tab:Neuman-RFM-ex1-6d-qmc}
    \end{table}

\begin{figure}[H]
    \centering
\includegraphics[width=0.95\linewidth]{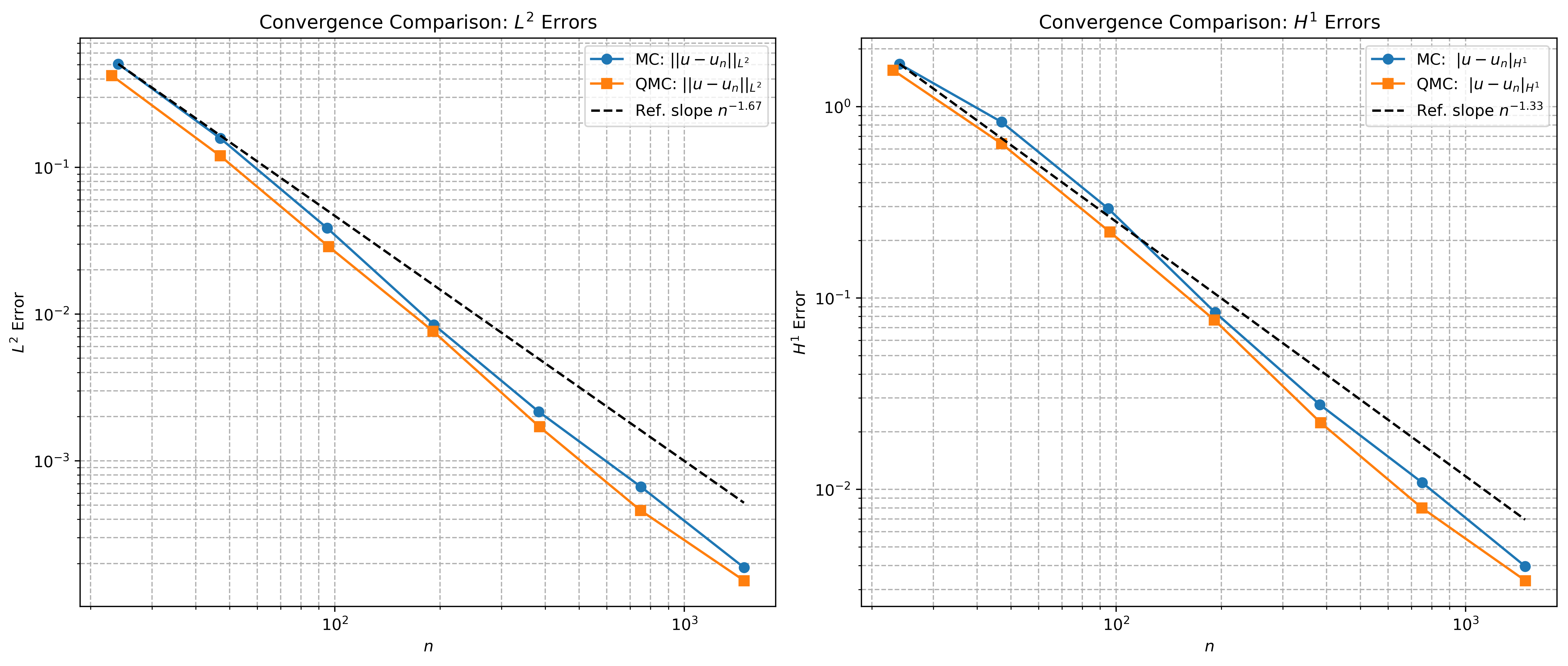}
    \caption{3D Neumann problem. ReLU$^3$. Comparison between using QMC points on $S^d$ and randomly sampled points. }
    \label{fig:3dNeumann-RFM-qmc}
\end{figure}

\begin{figure}[H]
    \centering
\includegraphics[width=0.95\linewidth]{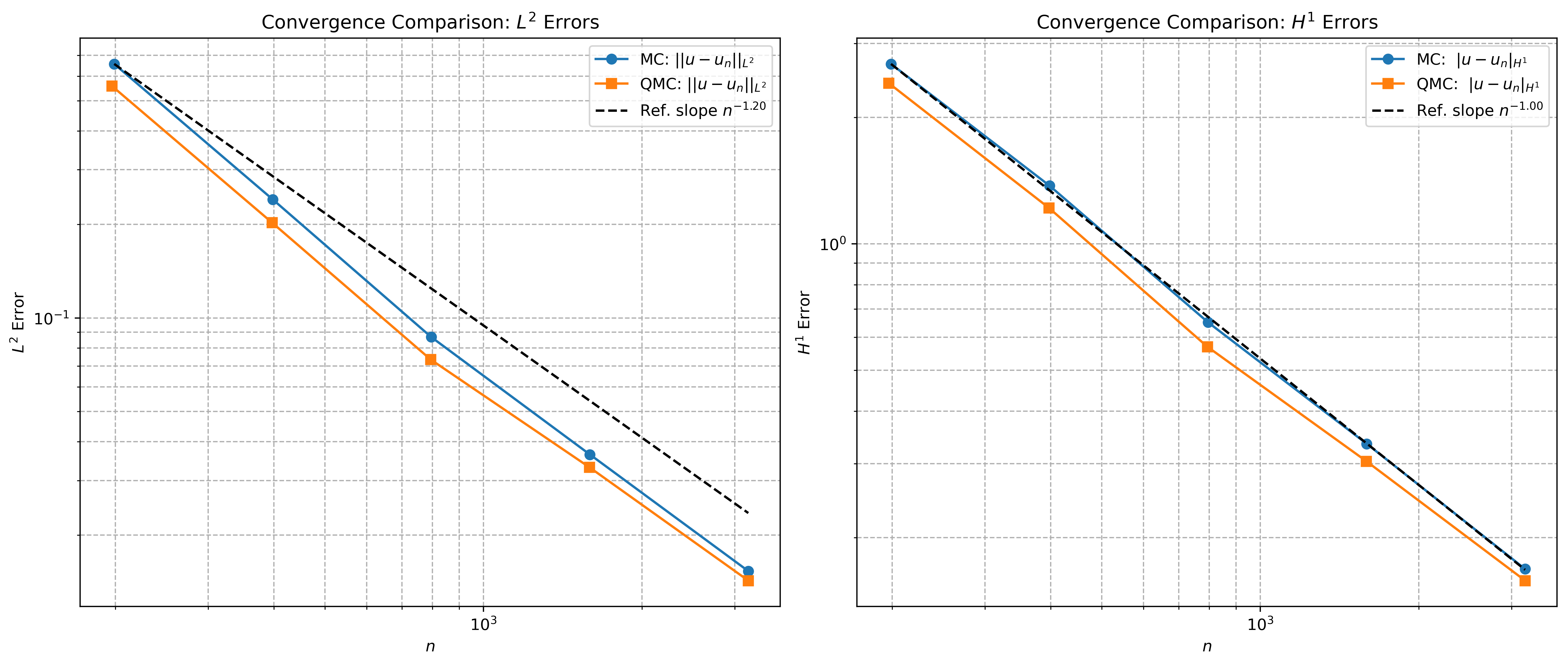}
    \caption{5D Neumann problem. ReLU$^3$. Comparison between using QMC points on $S^d$ and randomly sampled points.}
    \label{fig:5dNeumann-RFM-qmc}
\end{figure}

\begin{remark}
    Using the quasi-Monte Carlo points yields a more uniform distribution on hyperspheres, which can lead to a slight improvement in accuracy. 
   Alternatively, to get a quasi-uniform grid, one can also compute a centroidal Voronoi tessellation (CVT) on the hypersphere $S^d$ using Lloyd’s algorithm, following the framework in \cite{du1999centroidal, du2003constrained}. 
\end{remark}

\subsection{Tanh networks and deterministic schemes}
We have so far focused on deterministic ReLU$^k$ networks. 
Here, we also investigate the classical $\tanh$ activation, which is widely used in scientific machine learning.
Similar to the experiments conducted on ReLU$^k$ neural networks, we first compare the variational formulation with the collocation method.

We show that the variational formulation leads to a linear system that is numerically less stable due to its extremely poor conditioning, which severely limits its practical applicability.
In contrast, in the collocation formulation, $\tanh$ activations can still achieve high numerical accuracy, provided that numerically robust least-squares solvers are employed.
These observations align with common practice: $\tanh$ is typically used in collocation formulations, whereas ReLU-type activations are more common in variational or Galerkin formulations.

Furthermore, we propose two deterministic schemes for setting the parameters in the nonlinear layer of the $\tanh$ neural network.
Both schemes achieve accuracy comparable to that of the commonly used random sampling strategies in ELM and RFM. 
This demonstrates that randomness may not be essential for achieving high accuracy in these methods.

We first consider $L^2$ minimization problems in one and two dimensions, 
with target functions taken from the sinusoidal family. 
In one dimension, we use
\begin{equation} \label{eq:tanh-l2min-1d}
u(x) = \sin(m \pi x), \quad m \in \{1,2,4\},
\end{equation}
and in two dimensions
\begin{equation}\label{eq:tanh-l2min-2d}
u(x) = \sin(m \pi x_1)\,\sin(m \pi x_2), \quad m \in \{1,2,4\}.    
\end{equation}
The trial space is chosen as $V_n = \mathrm{span}\{\phi_j\}_{j=1}^n$, 
with neurons
\[
\phi_j(x) = \tanh(\omega_j \cdot x + b_j).
\]

To fix the nonlinear layer parameters, we sample $\{(\omega_j,b_j)\}_{j=1}^n$ 
independently from the uniform distribution on $[-R,R]^{d+1}$ for some 
appropriate parameter $R$ that requires manual tuning. 
This initialization is widely used in the extreme learning machine (ELM) literature and in random feature methods. 
It has been demonstrated empirically to yield fast, in many cases 
exponential, convergence in practical problems, although a rigorous 
theoretical justification of this phenomenon remains open.

In the following, we solve the $L^2$-minimization problem using the variational formulation and the collocation formulation,  respectively. 
\subsubsection{Continuous $L^2$-minimization}
In the variational formulation, to assemble the mass matrix, we use a piecewise Gauss quadrature.
In 1D, we divide the interval into 1024 uniform subintervals and use a Gauss quadrature of order 5 in each subinterval.
In 2D, we divide the square domain into $100^2$ uniform subdomains and use a Gauss quadrature of order 5 in each subdomain.
\begin{figure}[H]
    \centering
    \includegraphics[width=0.325\linewidth]{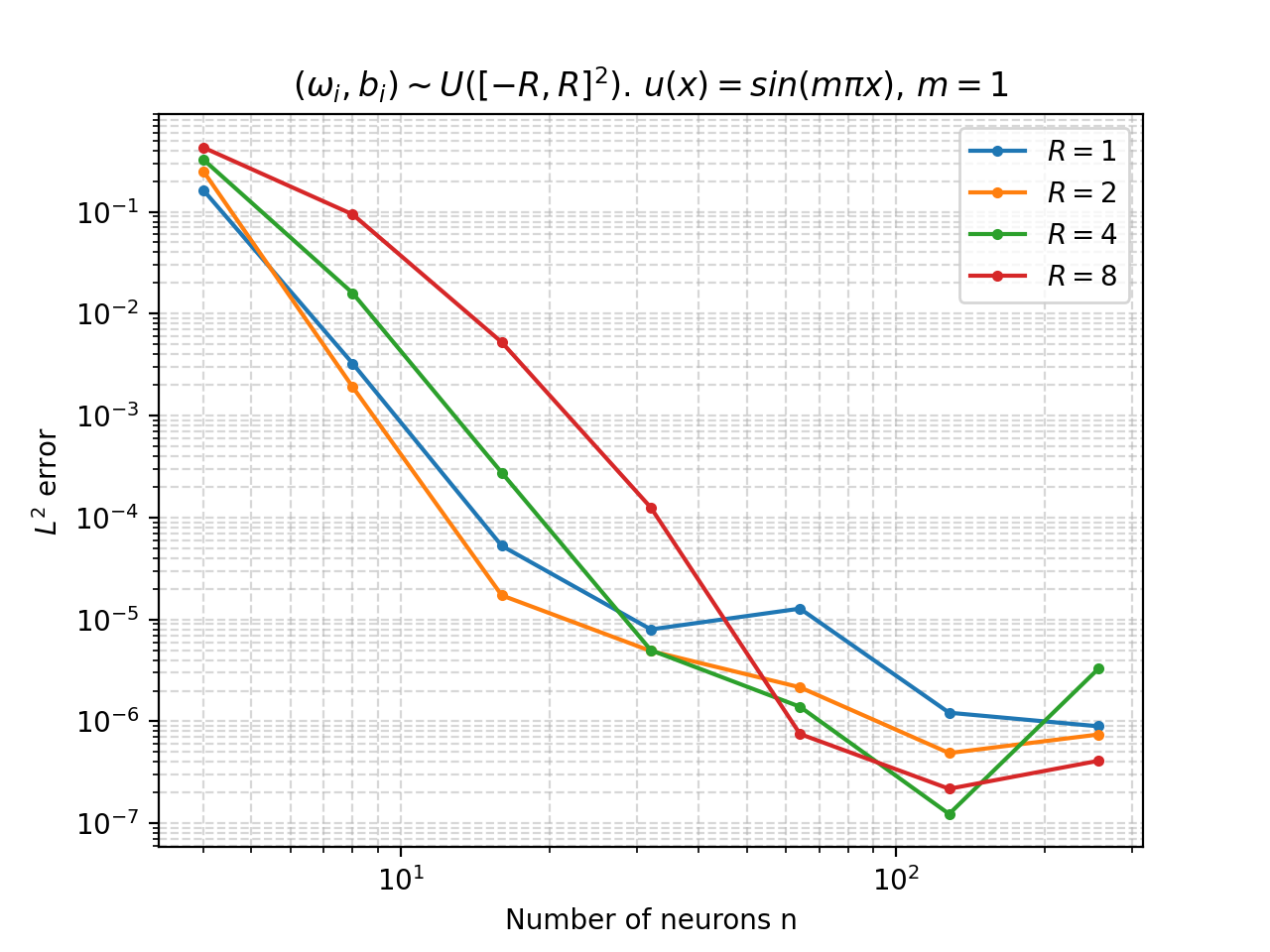}
    \includegraphics[width=0.325\linewidth]{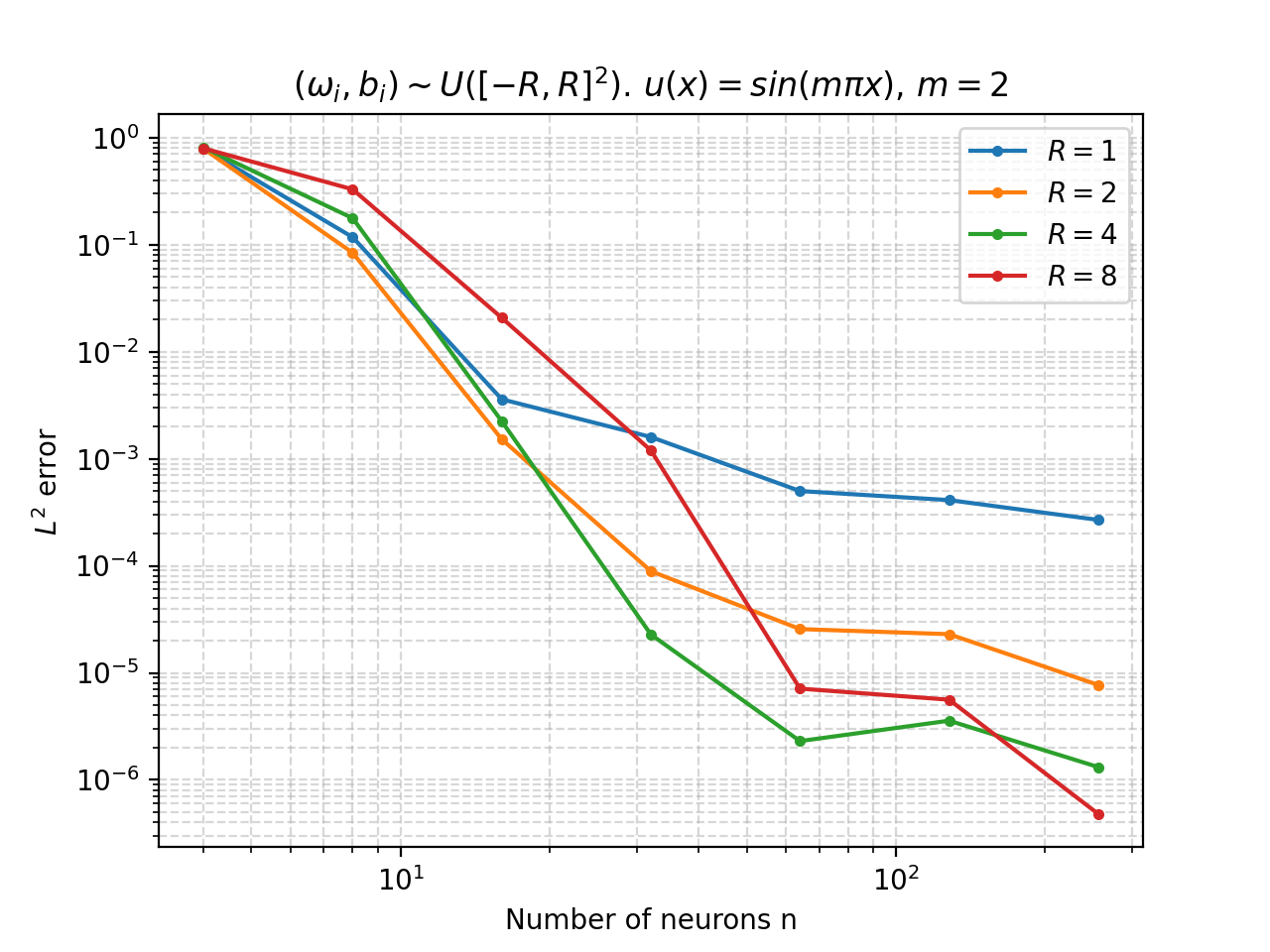}
    \includegraphics[width=0.325\linewidth]{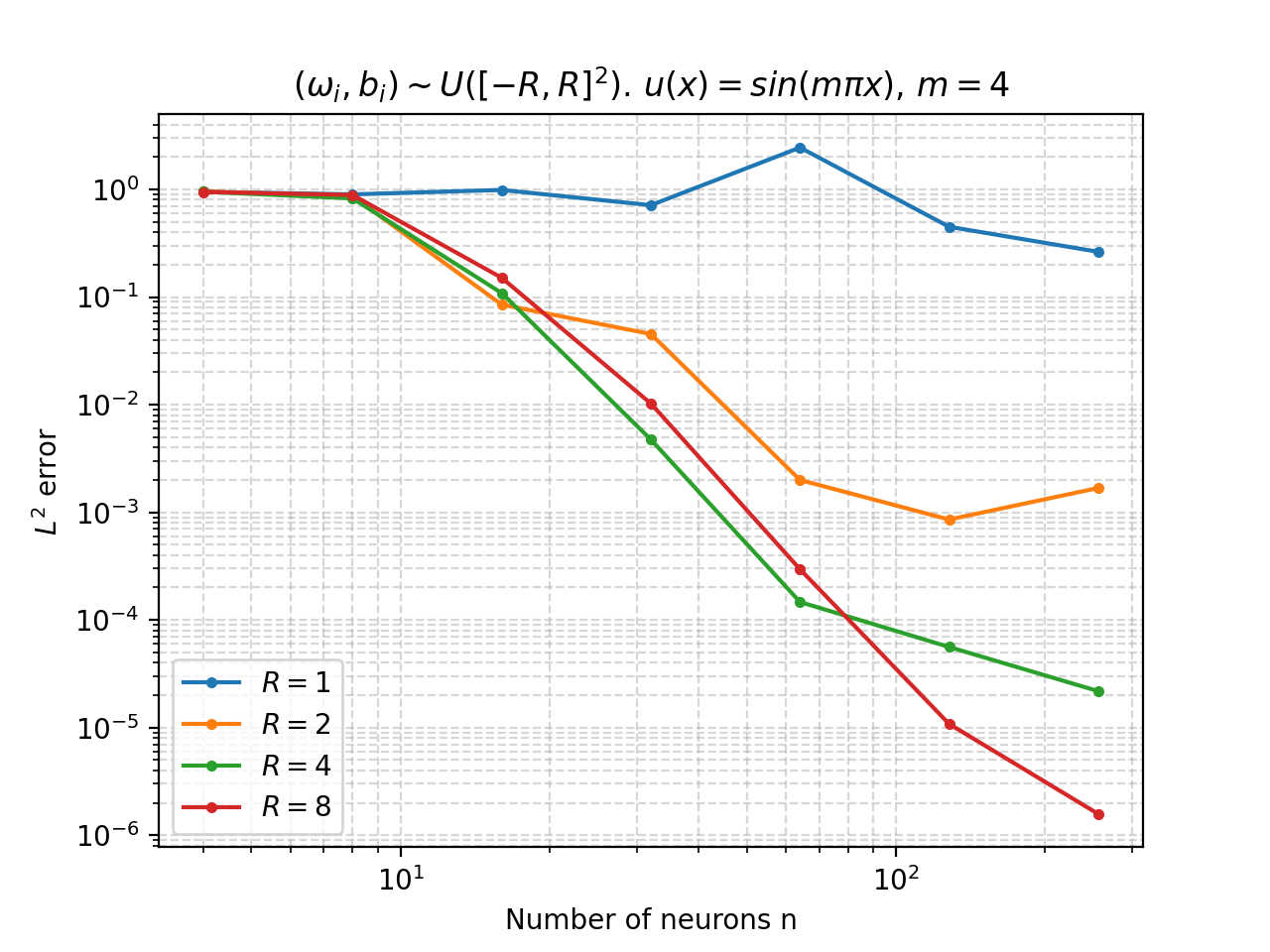} 
    \caption{1D $L^2$ minimization with $\tanh$ activations. Variational formulation. 
    Target function $u(x) = \sin(m\pi x)$ for $m=1,2,4$.}
    \label{fig:1dl2-tanh-var}
\end{figure}

\begin{figure}[H]
    \centering
    \includegraphics[width=0.325\linewidth]{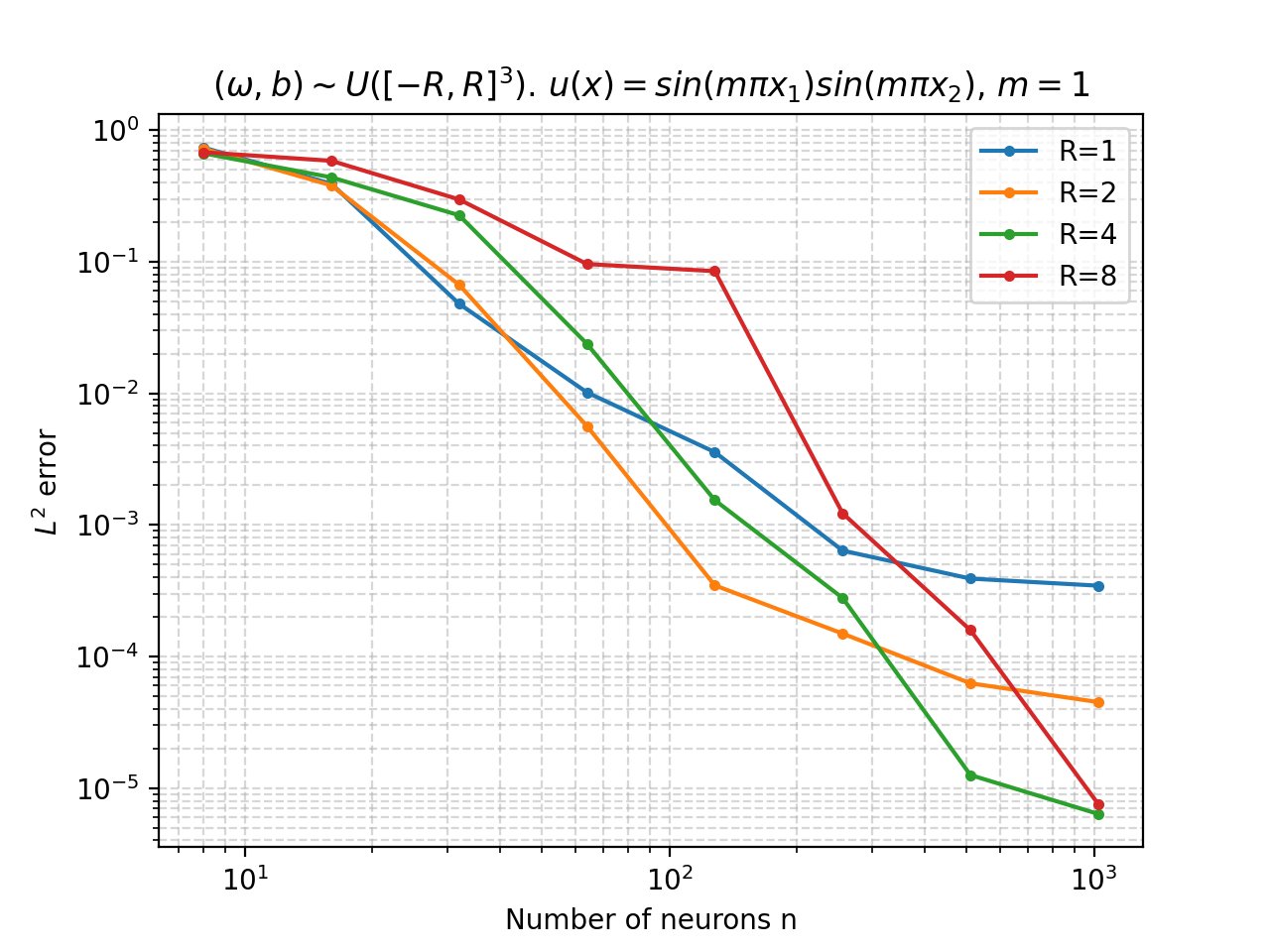}
    \includegraphics[width=0.325\linewidth]{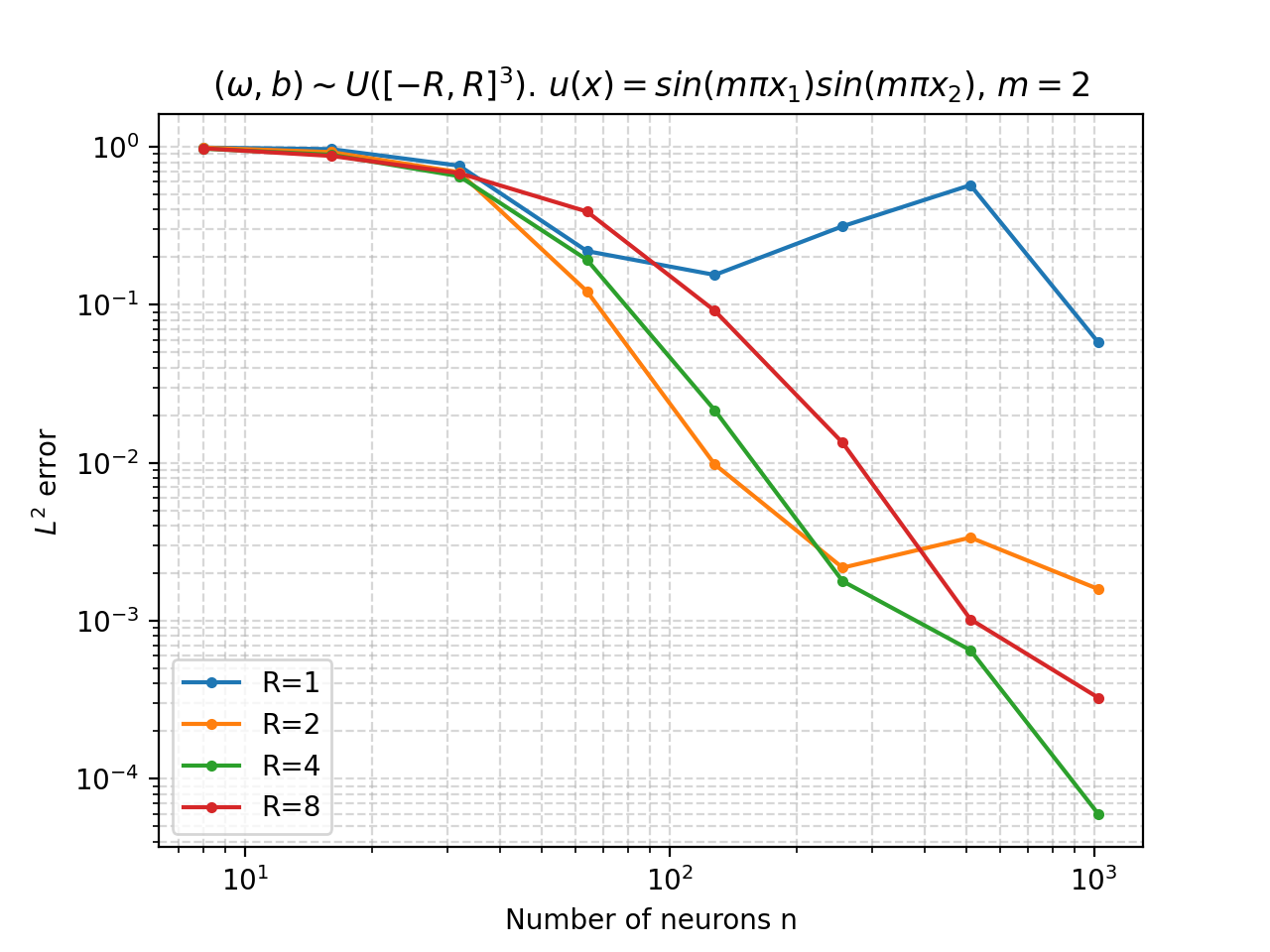}
    \includegraphics[width=0.325\linewidth]{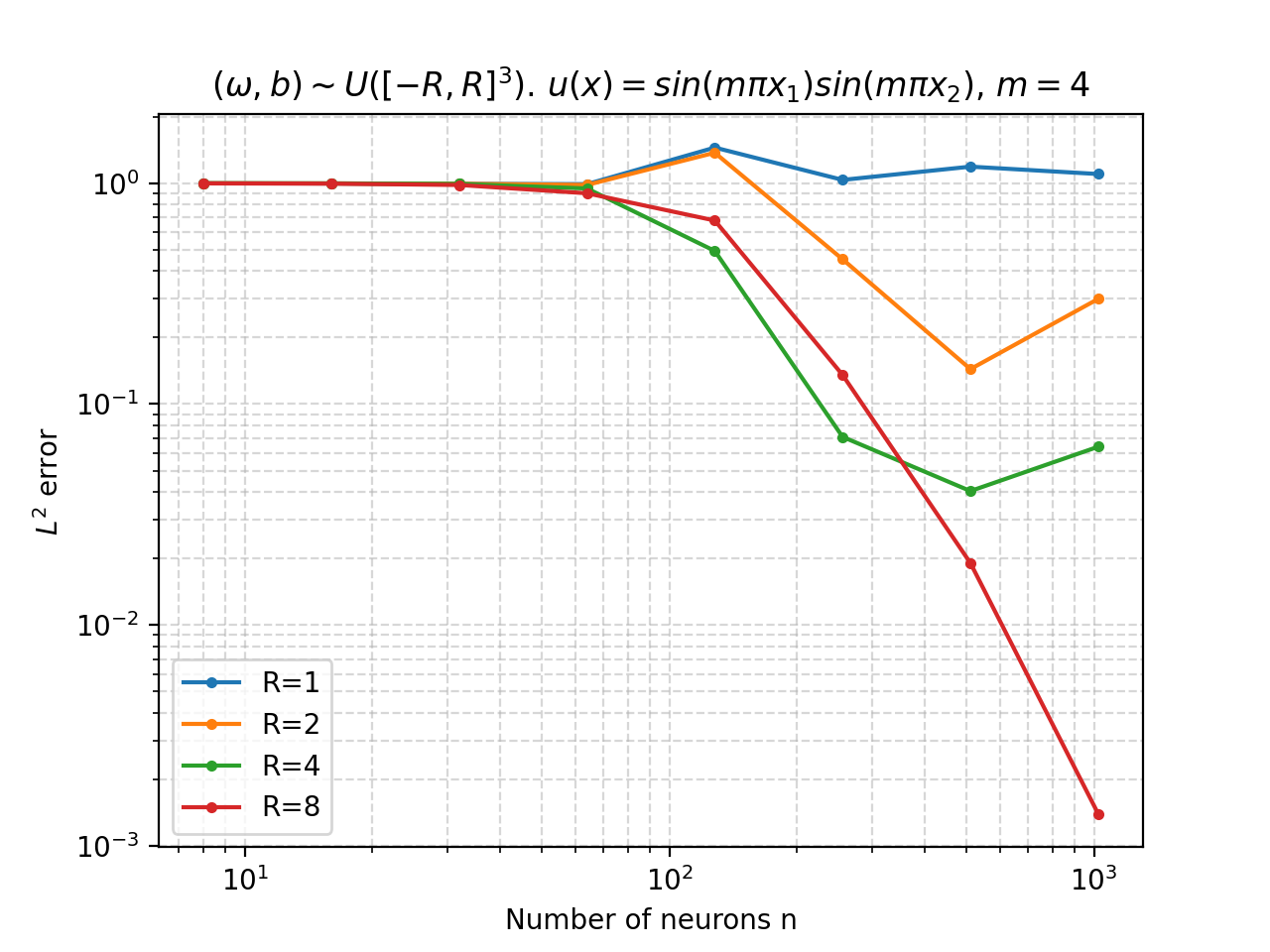} 
    \caption{2D $L^2$ minimization with $\tanh$ activations. Variational formulation. 
    Target function $u(x) = \sin(m\pi x_1)\sin(m\pi x_2)$ for $m=1,2,4$.}
    \label{fig:2dl2-tanh-var}
\end{figure}

\subsubsection{Discrete $\ell^2$-regression}

We now consider the discrete least-squares regression problem, which corresponds to the collocation formulation introduced in 
Section~\ref{sec:methods}.  
In one dimension, the collocation points are chosen as a uniform grid of size $1024$ over the interval, including the boundary points. 
In two dimensions, a uniform $50\times50$ tensor-product grid is used over the domain, including the boundary points. 
To evaluate the numerical errors, we measure it under the continuous $L^2$-norm by using an accurate piecewise Gauss quadrature method. 
In 1D, we divide the interval into 1024 uniform subintervals and use a Gauss quadrature of order 5 in each subinterval.
In 2D, we divide the square domain into $50^2$ uniform subdomains and use a Gauss quadrature of order 5 in each subdomain.

\begin{figure}[H]
    \centering
    \includegraphics[width=0.325\linewidth]{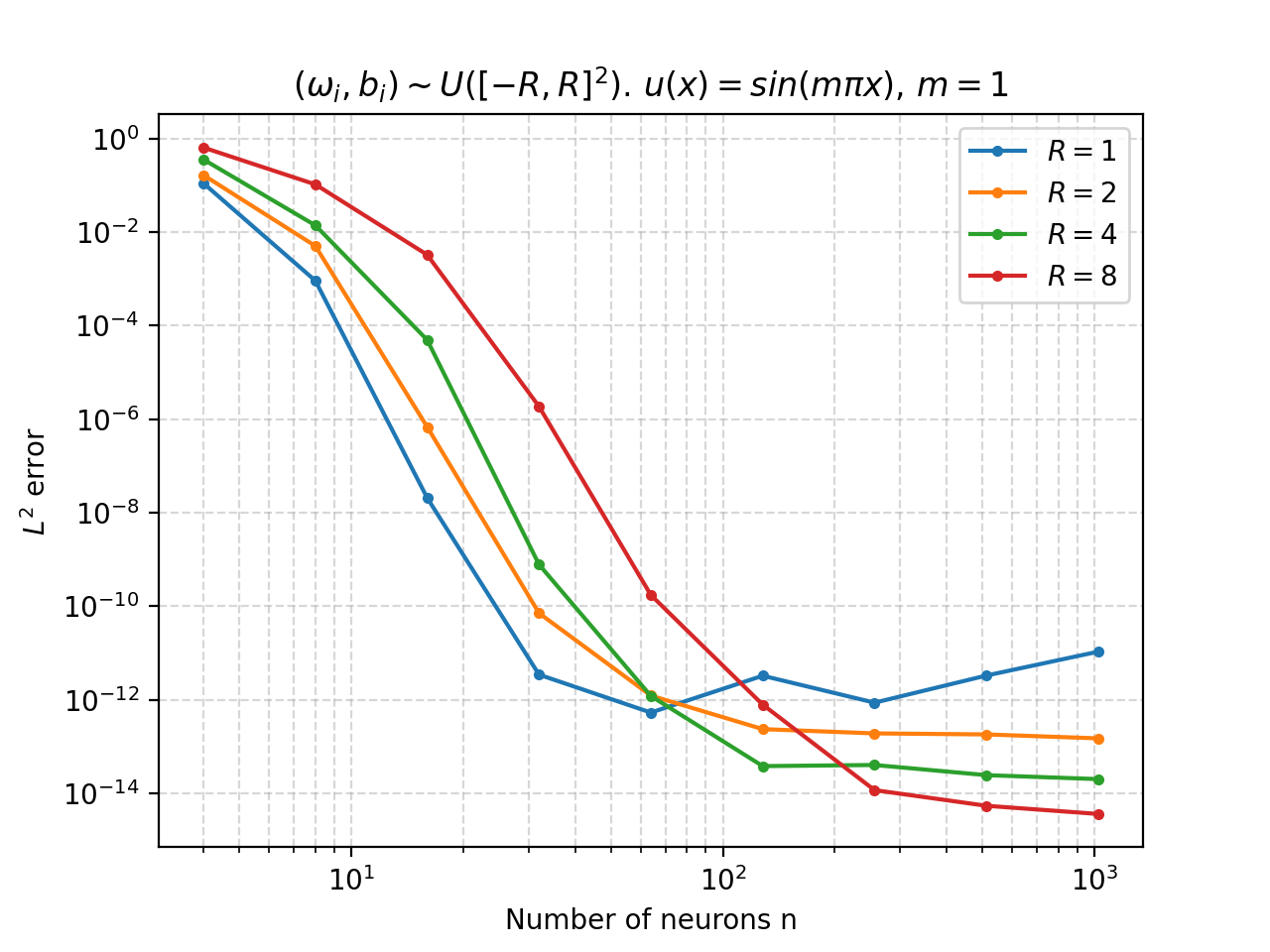}
    \includegraphics[width=0.325\linewidth]{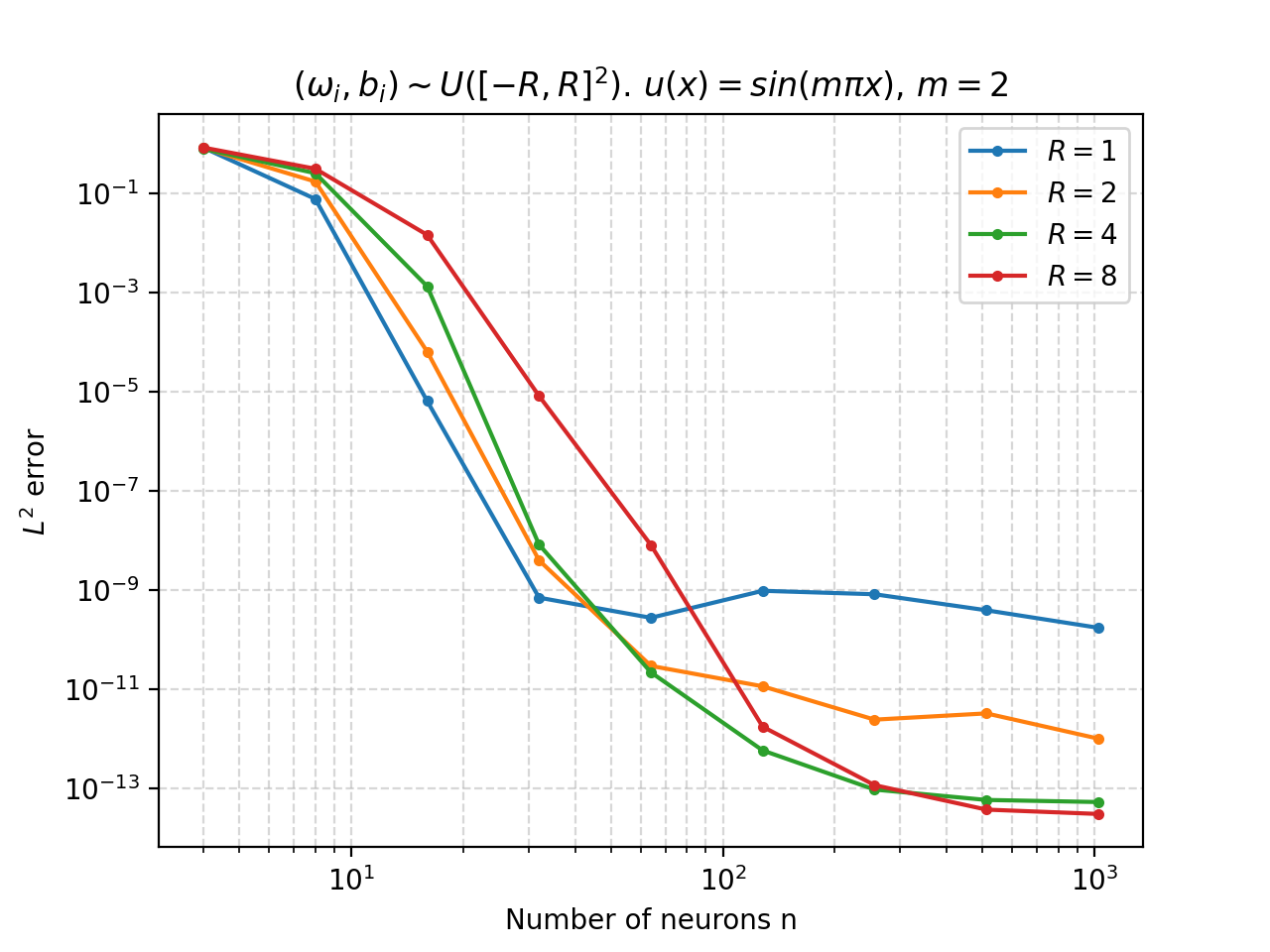}
    \includegraphics[width=0.325\linewidth]{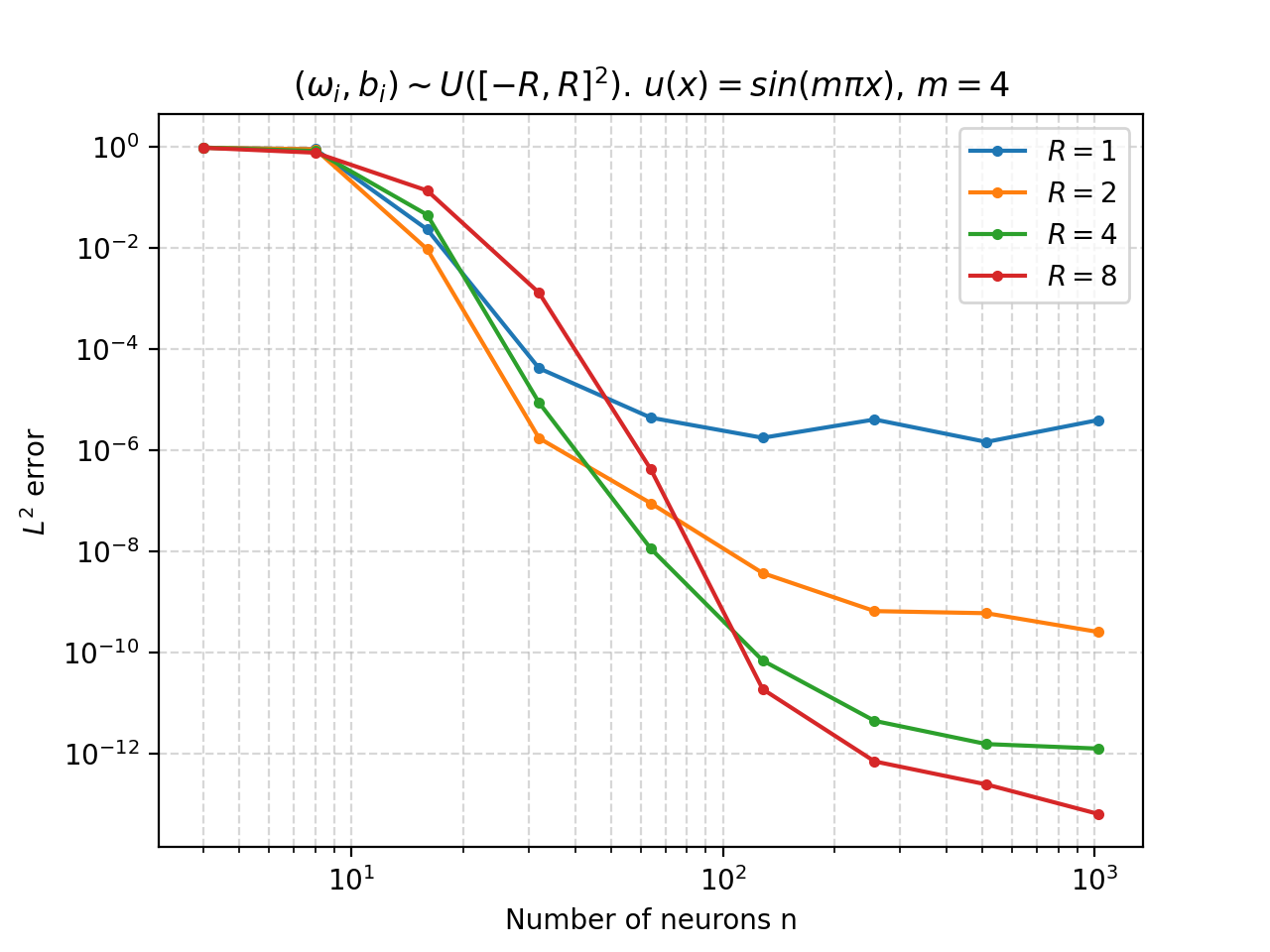}
    \caption{1D $L^2$ minimization with $\tanh$ activations. Collocation formulation. 
    Target function $u(x) = \sin(m\pi x)$ for $m=1,2,4$. }
    \label{fig:1dl2-tanh-col}
\end{figure}

\begin{figure}[H]
    \centering
    \includegraphics[width=0.325\linewidth]{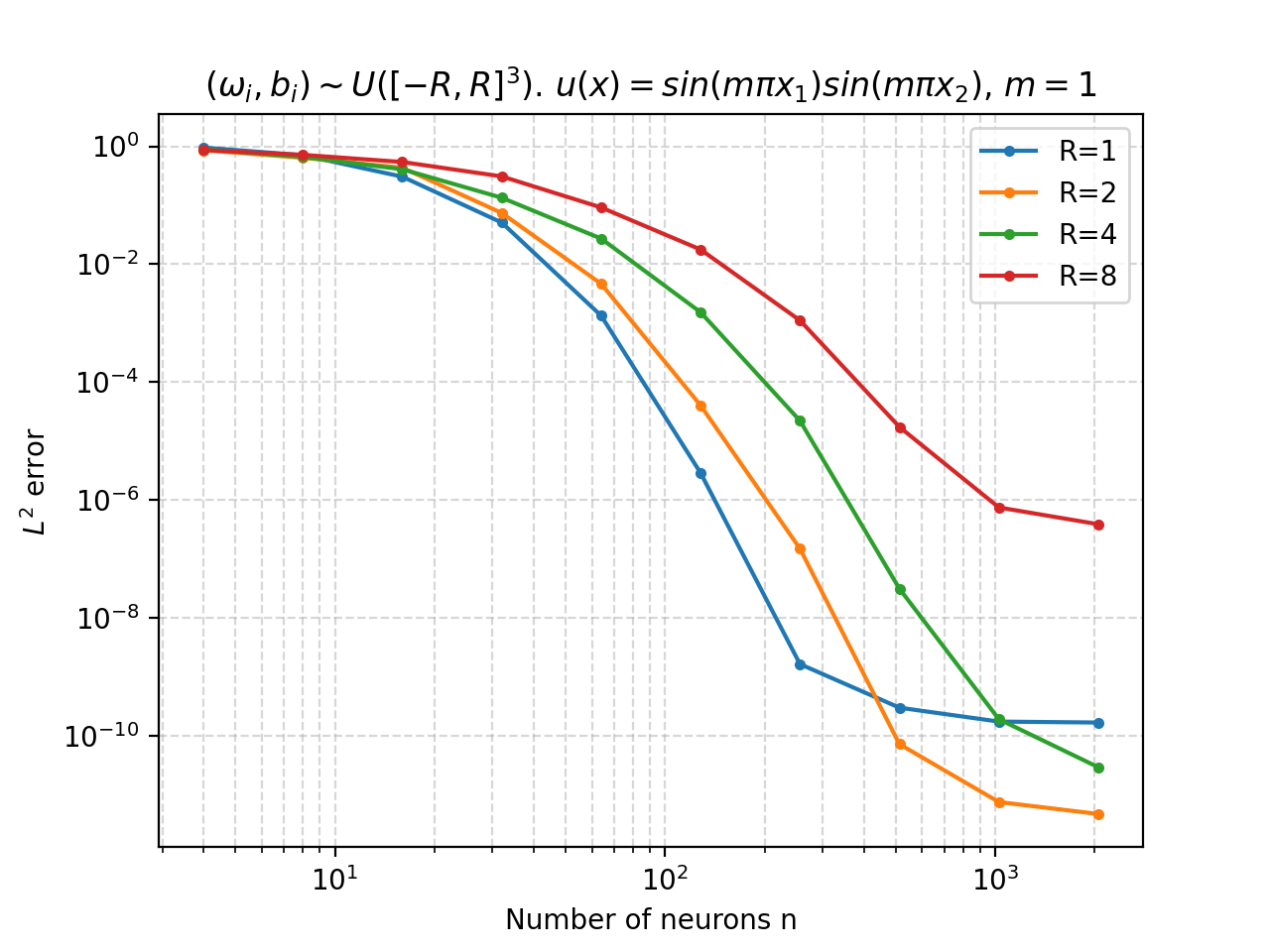}
    \includegraphics[width=0.325\linewidth]{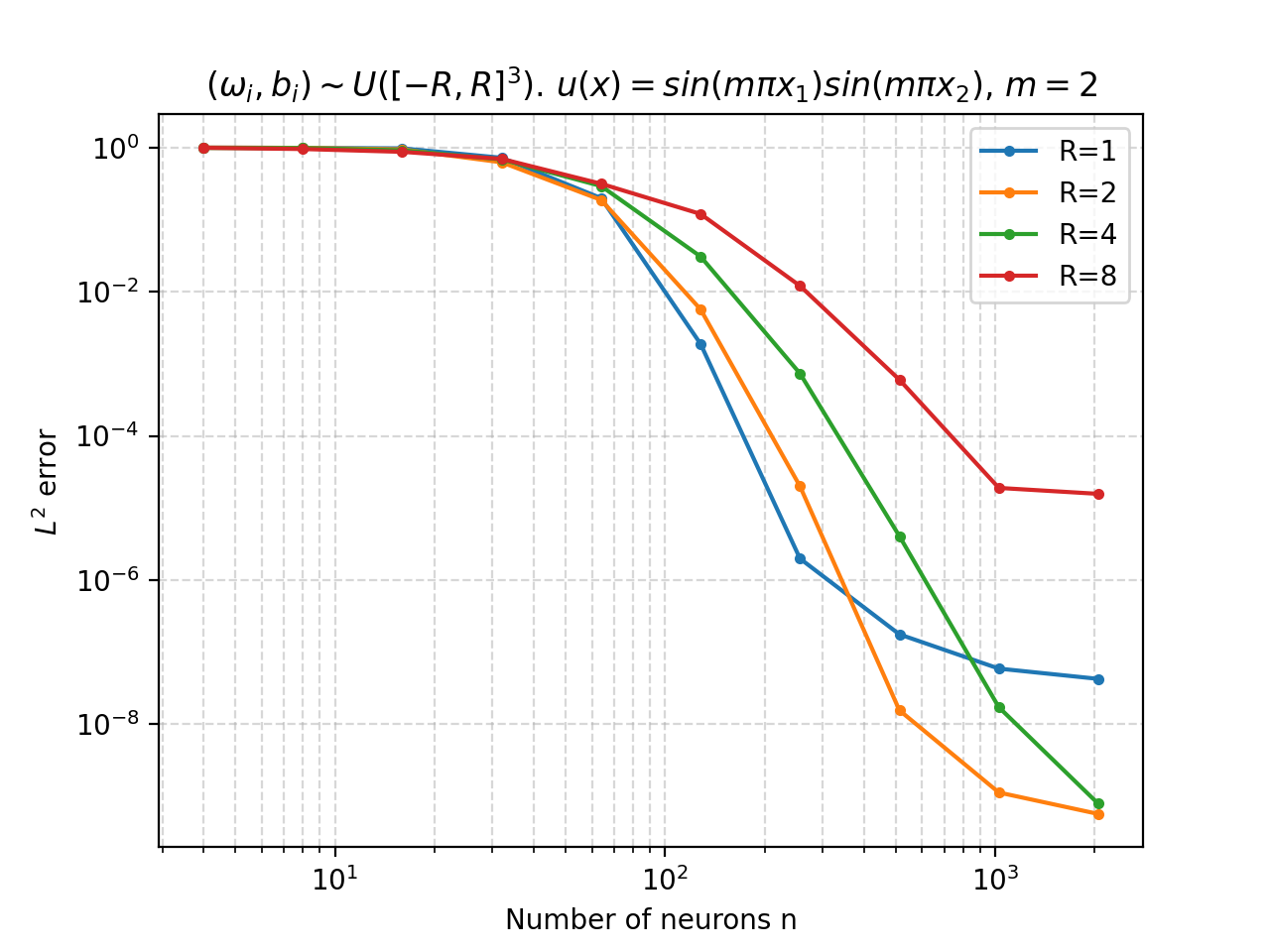}
    \includegraphics[width=0.325\linewidth]{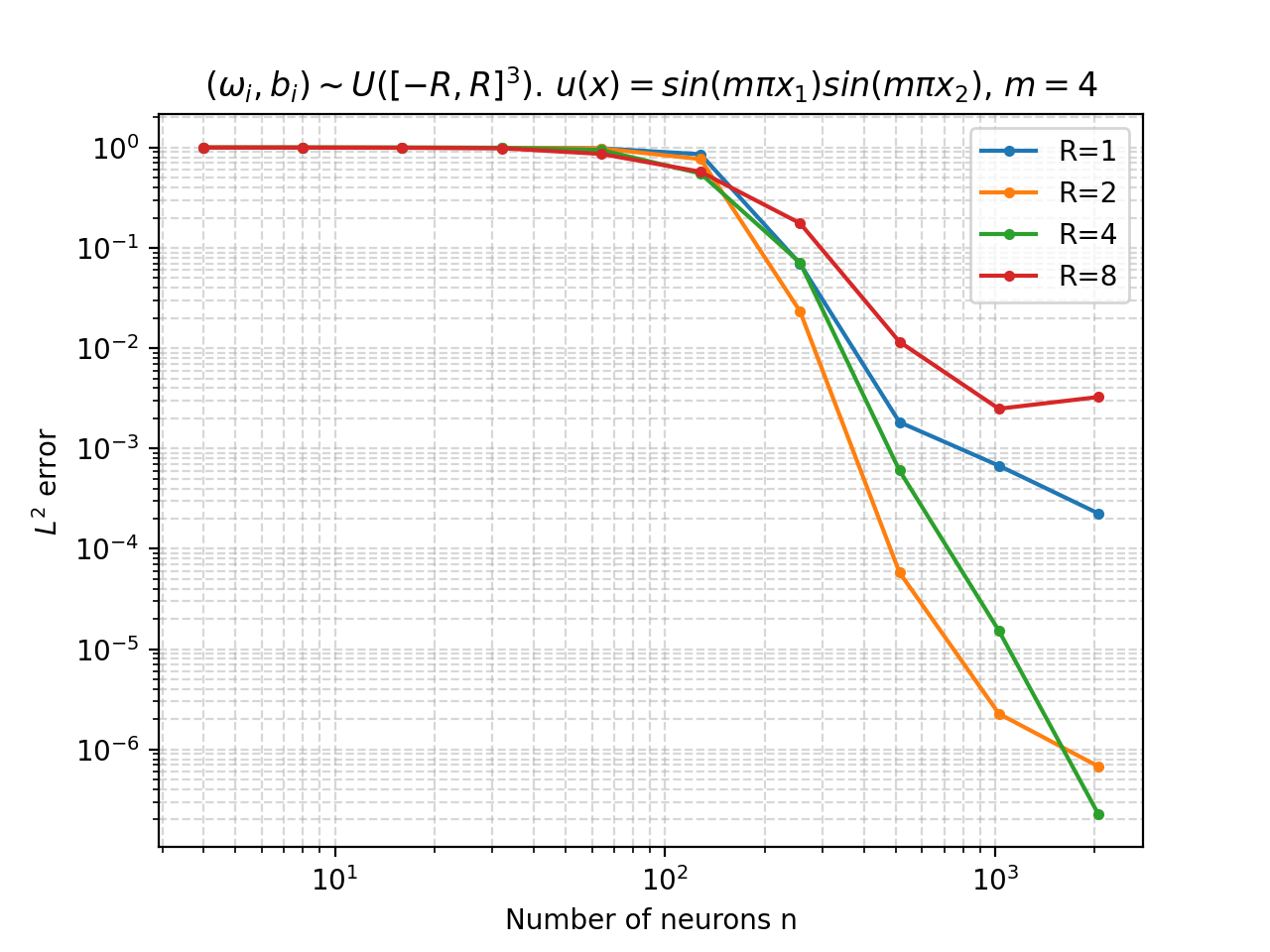}
    \caption{2D $L^2$ minimization with $\tanh$ activations. Collocation formulation. 
    Target function $u(x) = \sin(m\pi x_1)\sin(m\pi x_2)$ for $m=1,2,4$. }
    \label{fig:2dl2-tanh-col}
\end{figure}

\textbf{Observations.}  
For the $L^2$ minimization tasks in both one and two dimensions, the collocation formulation consistently achieves significantly higher accuracy than the variational formulation.
Using the collocation method, errors below $10^{-14}$ are attained in one dimension and below $10^{-10}$ in two dimensions with only a few hundred to a thousand neurons, demonstrating very rapid convergence. 
In contrast, when the variational formulation is used, the achievable accuracy is substantially lower. Errors of order $10^{-7}$ in one dimension and $10^{-5}$ in two dimensions are observed under comparable settings.

\subsubsection{Condition numbers and numerical instability}

We now examine the conditioning of the mass matrices arising in the variational formulation 
with $\tanh$ activations. Tables~\ref{tab:condition-tanh-1d} and~\ref{tab:condition-tanh-2d} 
report the condition numbers for one- and two-dimensional $L^2$ minimization problems, 
as the number of neurons $n$ increases. 
In both cases, the condition numbers grow extremely rapidly with~$n$, 
indicating severe ill-conditioning in the underlying linear systems. 

\begin{table}[H]
    \centering
    \begin{tabular}{c|c}
        $n$ & condition number \\ \hline \hline 
        4 & 1.8e05\\ \hline 
        8 & 5.5e13 \\ \hline 
        16 & 1.8e16 \\ \hline 
        32 & 2.4e17\\ \hline 
    \end{tabular}
    \caption{Condition numbers of the mass matrix for 1D $L^2$ minimization with $\tanh$ activations.}
    \label{tab:condition-tanh-1d}
\end{table}

\begin{table}[H]
    \centering
    \begin{tabular}{c|c}
        $n$ & condition number \\ \hline \hline 
        4 & 2.4e05\\ \hline 
        8 & 5.1e07 \\ \hline 
        16 & 5.7e10\\ \hline 
        32 & 1.6e16\\ \hline 
    \end{tabular}
    \caption{Condition numbers of the mass matrix for 2D $L^2$ minimization with $\tanh$ activations.}
    \label{tab:condition-tanh-2d}
\end{table}




\subsubsection{Deterministic parameter schemes }
In this subsection, we introduce two deterministic approaches for setting the nonlinear layer parameters of the $\tanh$ shallow neural network.

The first scheme is motivated by Petrushev’s result in \cite{petrushev1998approximation}, and hence we refer to it as Petrushev's scheme. 
The second scheme is inspired by the parameter selection methods used for ReLU$^k$ shallow neural networks in \cite{LMX2025}.
Since $\tanh$ is non-homogeneous, we set the nonlinear layer parameters on a quasi-uniform grid over $r \times S^d$ for a suitable radius $r$, where $S^d$ denotes the $d$-dimensional unit sphere in $\mathbb{R}^{d+1}$.
We refer to this as the sphere scheme.
The two schemes are described as follows.

\begin{enumerate}
    \item \textbf{Petrushev's scheme.}  
    Let $\Theta_1 := \{\omega_i \}_{i=1}^{n_1}$ denote a quasi-uniform grid on $r_1 S^{d-1}$, the $(d-1)$-sphere of radius $r_1$, and let $\Theta_2 := \{b_i\}_{i=1}^{n_2}$ denote a uniform grid on the interval $I := [-r_2, r_2]$ for some appropriate value of $r_2$.  
    We then form the tensor product of these two grids to obtain the parameter set  
    \[
    \Theta := \{ (\omega_i, b_j) \mid \omega_i \in \Theta_1,\, b_j \in \Theta_2 \}.
    \]
    This defines a set of points on the product domain $r_1 S^{d-1} \times [-r_2,r_2]$, which is used to set the nonlinear layer parameters.  
    \item \textbf{Sphere scheme.}  
    The parameters $(\omega_i, b_i)$ form a quasi-uniform grid on $r \times S^d$, where $S^d$ denotes the $d$-dimensional unit sphere in $\mathbb{R}^{d+1}$.  
\end{enumerate}

\textbf{$L^2$-minimization problems}

In the following, we solve the same $L^2$-minimization problems with target functions given by \eqref{eq:tanh-l2min-1d} in 1D and \eqref{eq:tanh-l2min-2d} in 2D using the collocation methods. 
The function space is chosen as $V_n = \mathrm{span}\{\phi_j\}_{j=1}^n$, 
where the neuron functions are defined as
\[
\phi_j(x) = \tanh(\omega_j \cdot x + b_j),
\]
and $(\omega_j,b_j)$ are set according to the two schemes introduced above. 
In one dimension, the collocation points are chosen as a uniform grid of size $1024$ over the interval, including the boundary points. 
In two dimensions, a uniform $50\times50$ tensor-product grid is used over the domain, including the boundary points. 
To evaluate the numerical errors, we measure it under the continuous $L^2$-norm by using an accurate piecewise Gauss quadrature method. 
In 1D, we divide the interval into 1024 uniform subintervals and use a Gauss quadrature of order 5 in each subinterval.
In 2D, we divide the square domain into $50^2$ uniform subdomains and use a Gauss quadrature of order 5 in each subdomain.

The numerical results obtained using Petrushev's scheme are plotted in Figures~\ref{fig:1dl2-tanh-col-petrushev} and~\ref{fig:2dl2-tanh-col-petrushev}, showing rapid convergence when appropriate values of $r_1$ and $r_2$ are chosen. 
The numerical results obtained using the sphere scheme are plotted in Figures~\ref{fig:1dl2-tanh-col-sphere} and~\ref{fig:2dl2-tanh-col-sphere}, also showing rapid convergence when an appropriate radius $r$ is selected. 

\begin{figure}[H]
    \centering
    \includegraphics[width=0.325\linewidth]{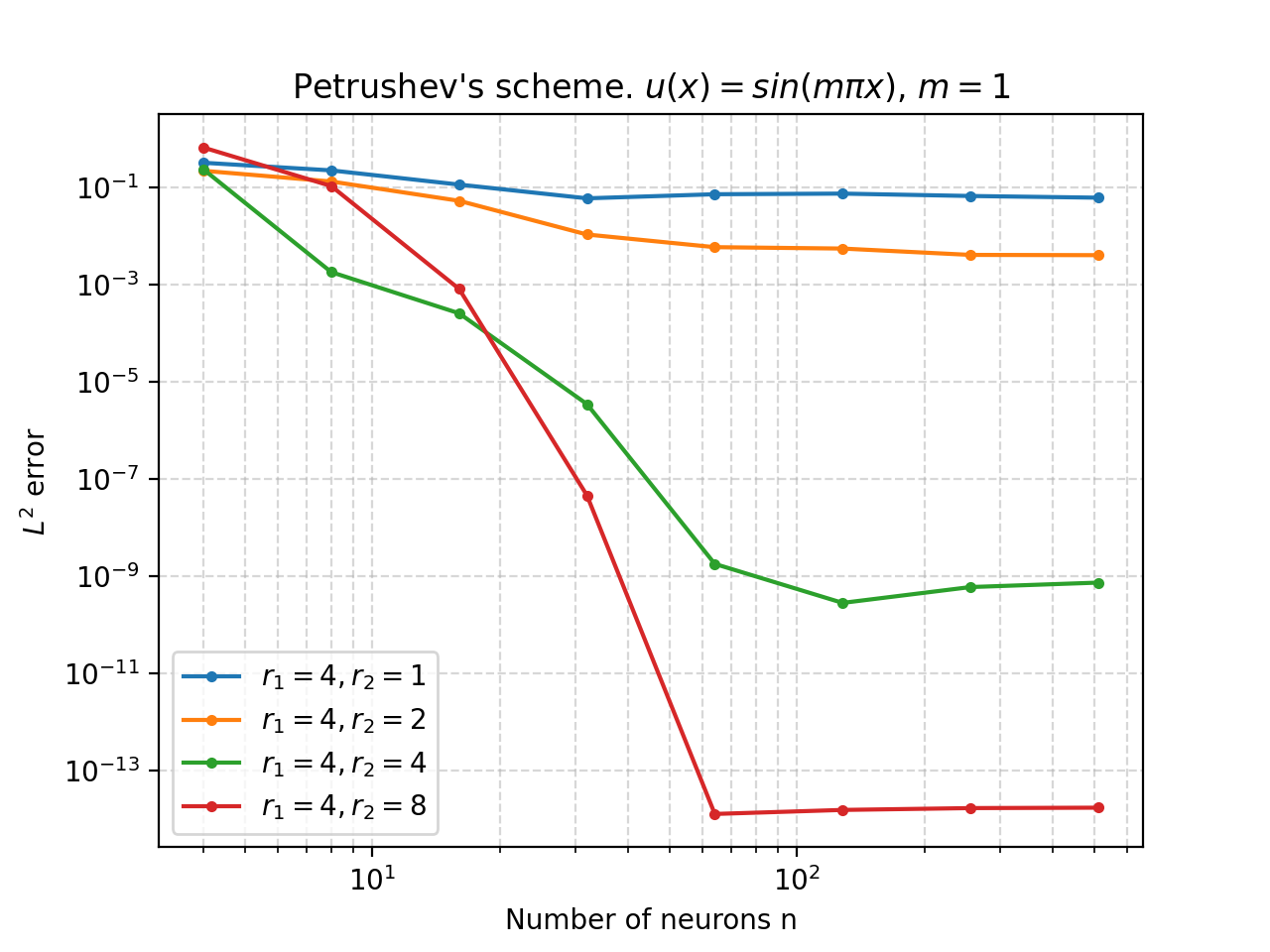}
    \includegraphics[width=0.325\linewidth]{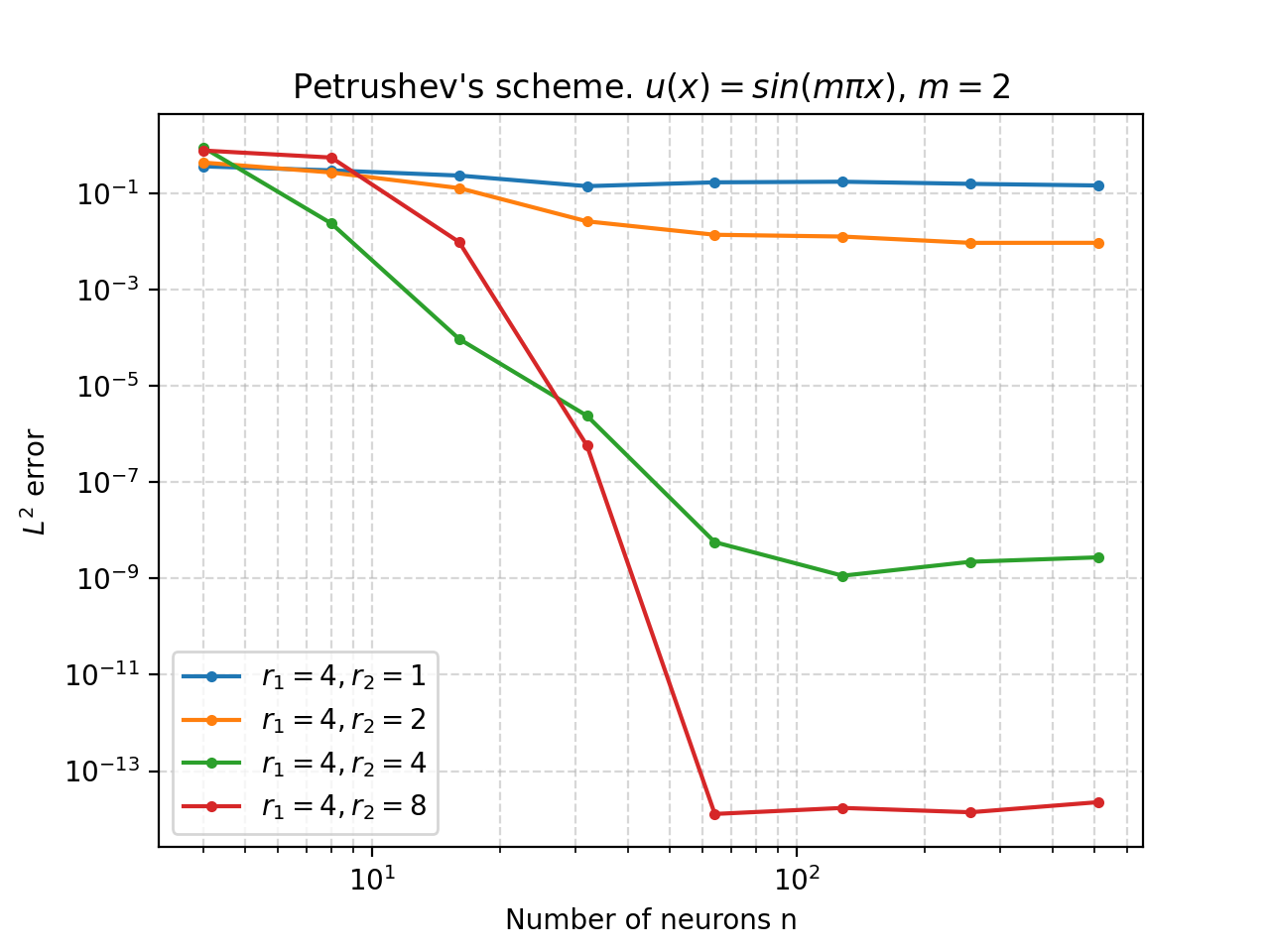}
    \includegraphics[width=0.325\linewidth]{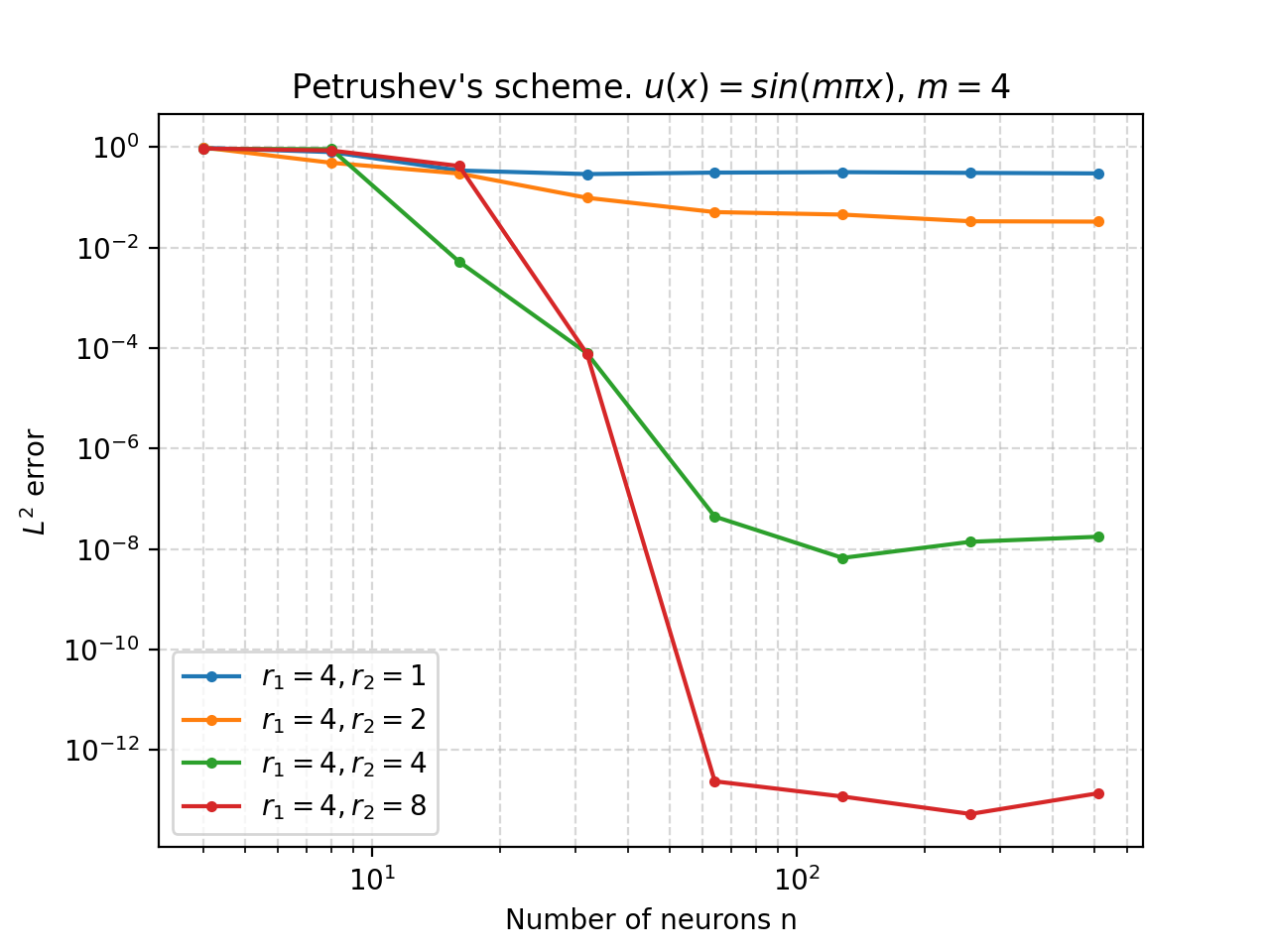}
    \caption{1D $L^2$ minimization with $\tanh$ activations. Collocation formulation. 
    Target function $u(x) = \sin(m\pi x)$ for $m=1,2,4$. }
    \label{fig:1dl2-tanh-col-petrushev}
\end{figure}

\begin{figure}[H]
    \centering
    \includegraphics[width=0.325\linewidth]{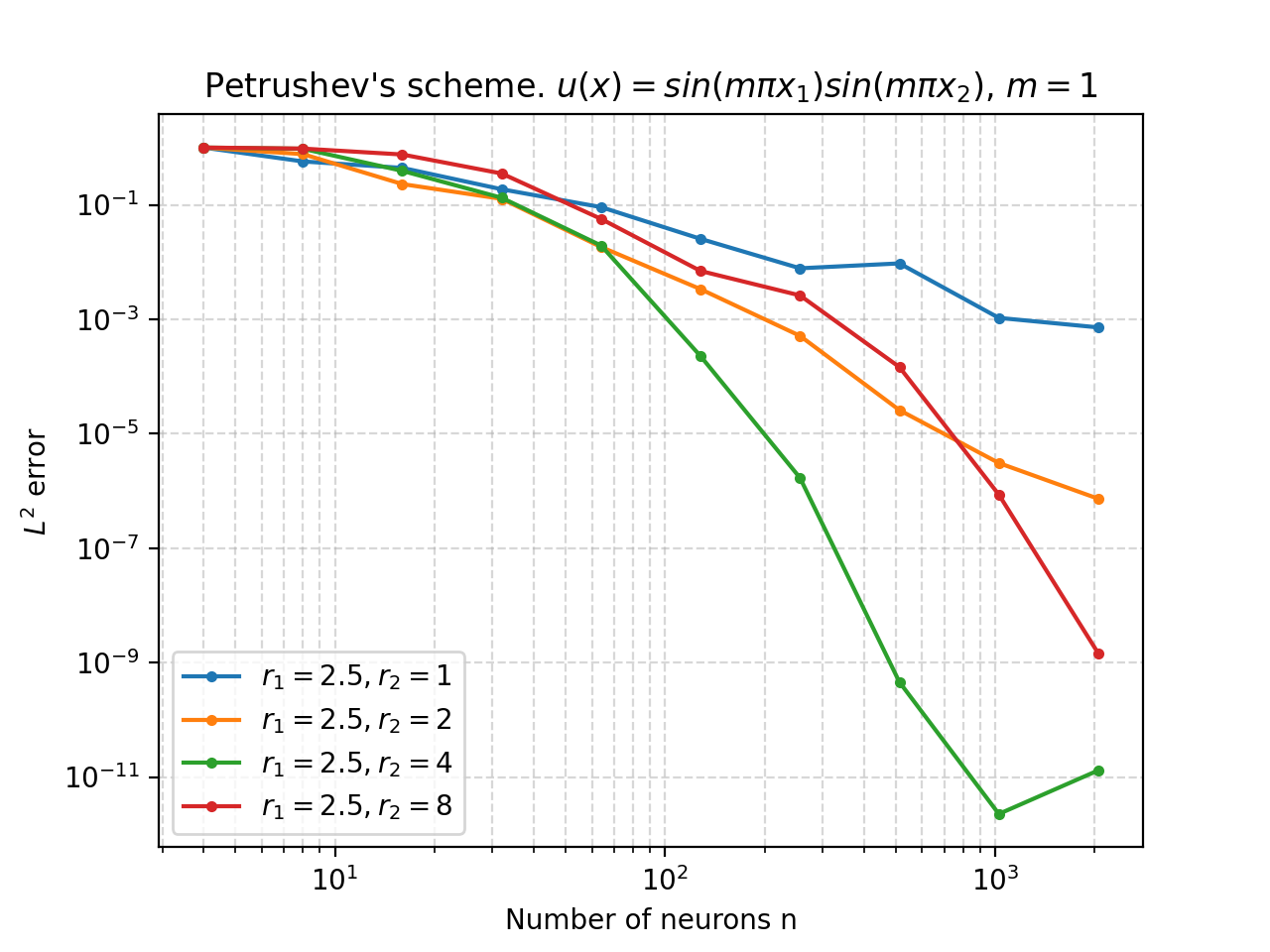}
    \includegraphics[width=0.325\linewidth]{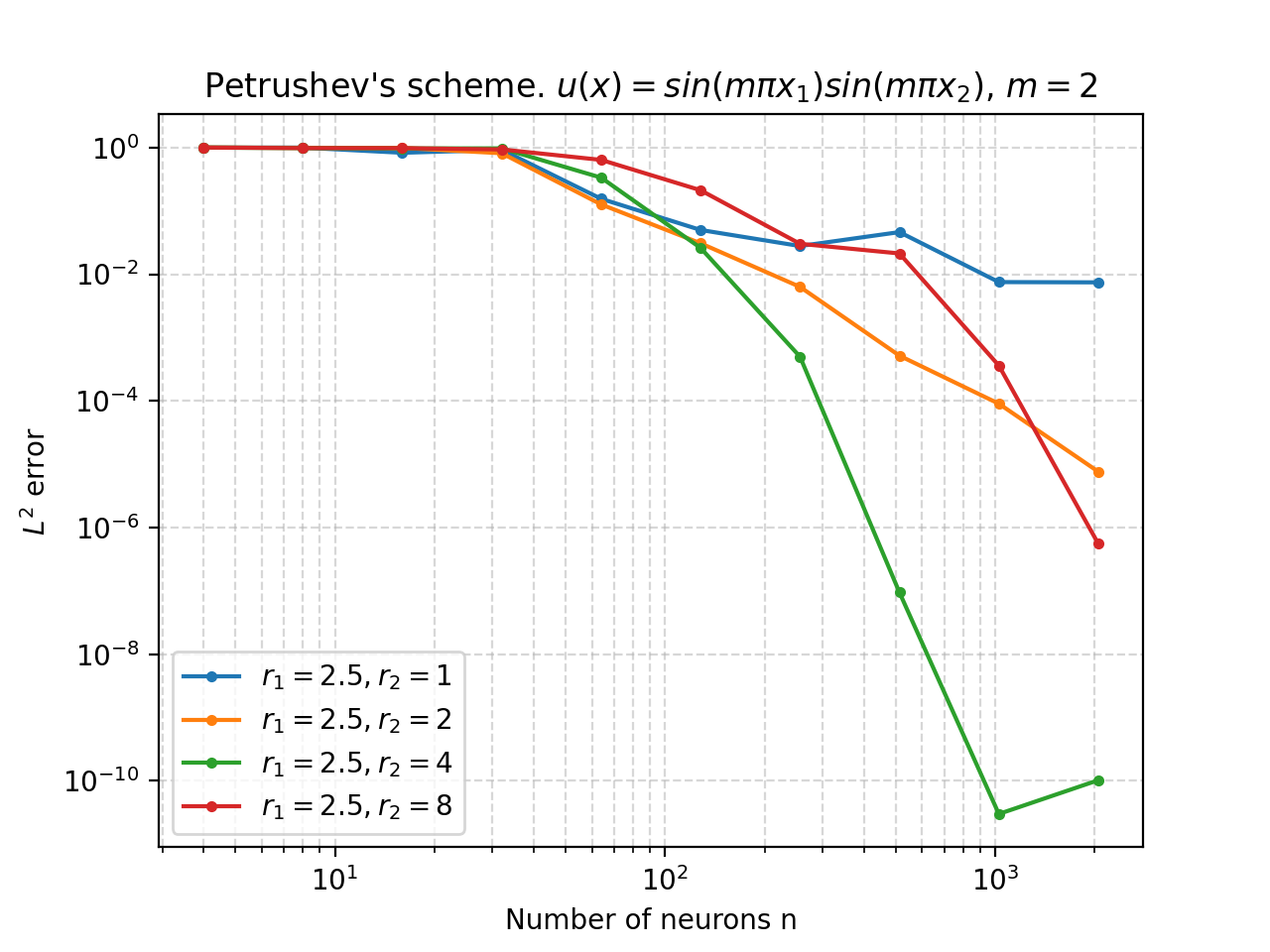}
    \includegraphics[width=0.325\linewidth]{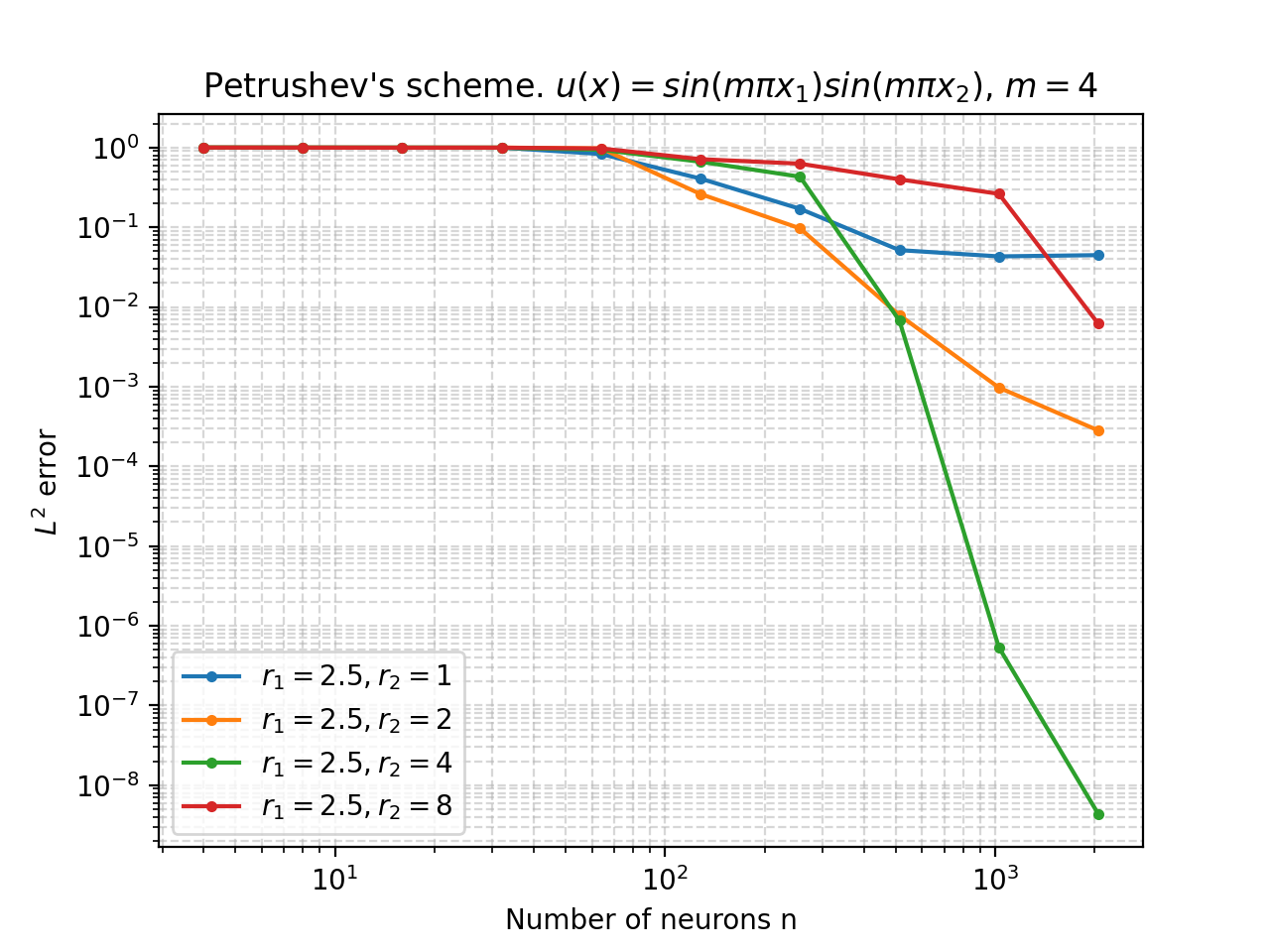}
    \caption{2D $L^2$ minimization with $\tanh$ activations. Collocation formulation. 
    Target function $u(x) = \sin(m\pi x_1)\sin(m\pi x_2)$ for $m=1,2,4$. }
    \label{fig:2dl2-tanh-col-petrushev}
\end{figure}

\begin{figure}[H]
    \centering
    \includegraphics[width=0.325\linewidth]{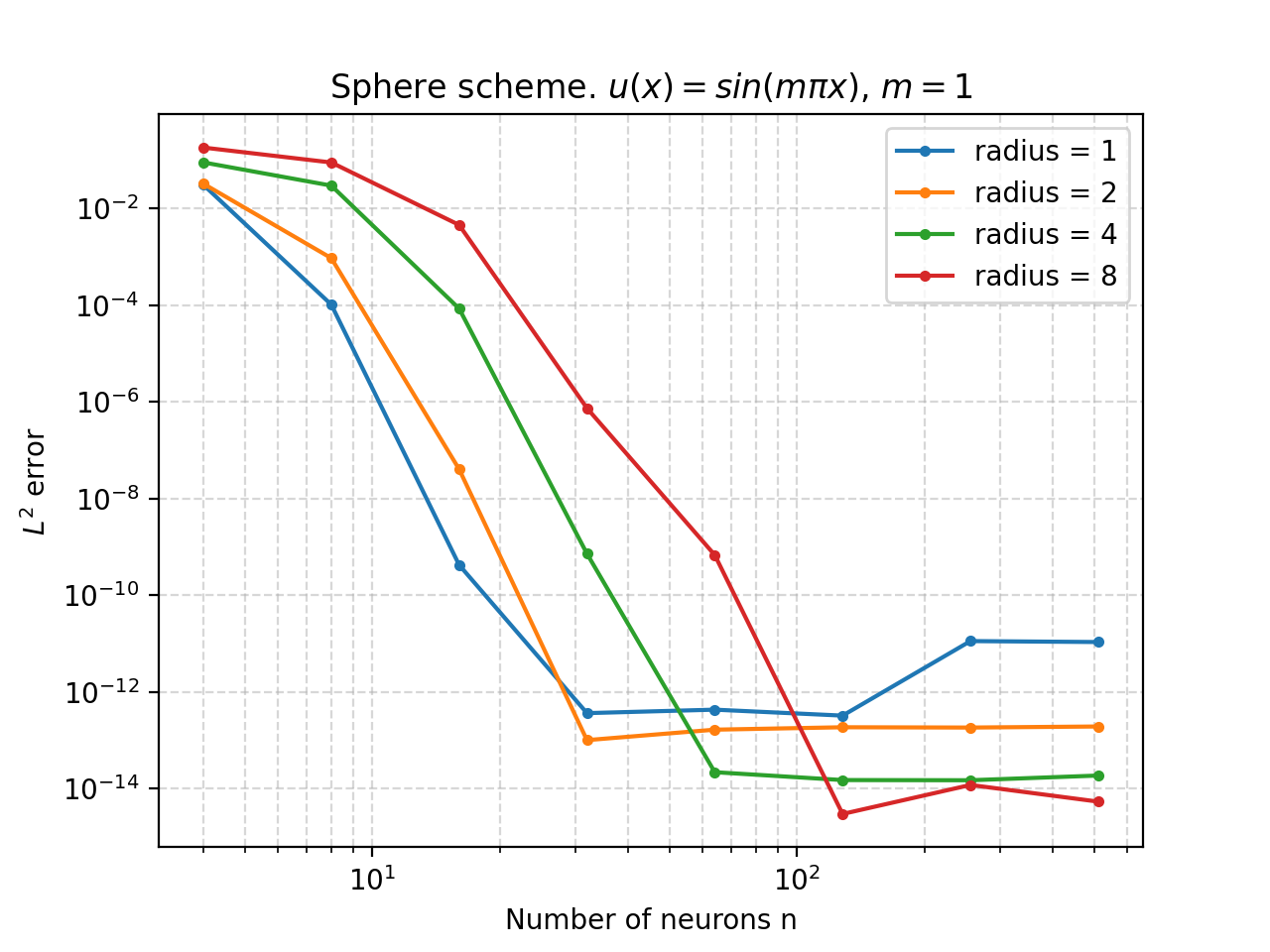}
    \includegraphics[width=0.325\linewidth]{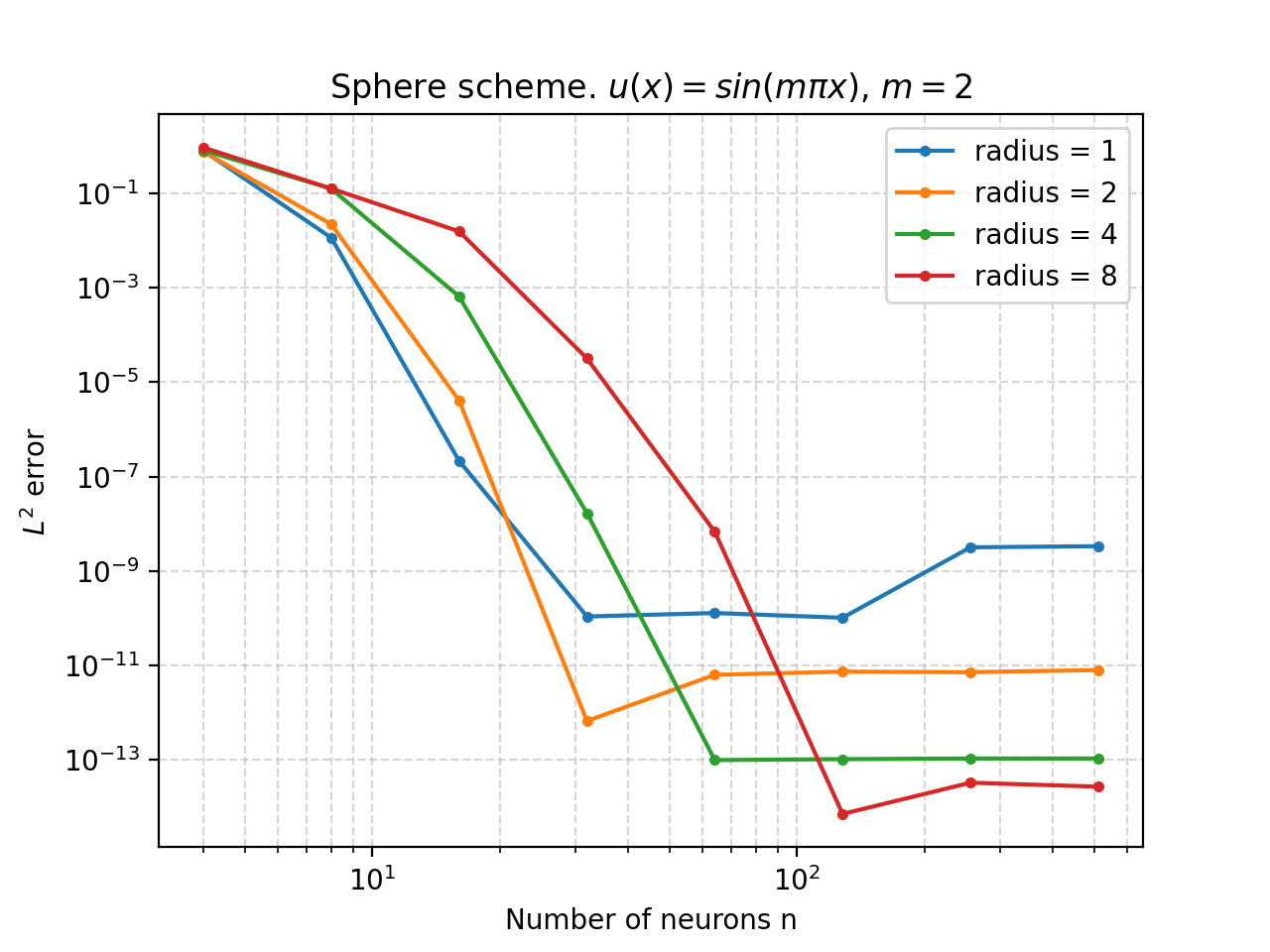}
    \includegraphics[width=0.325\linewidth]{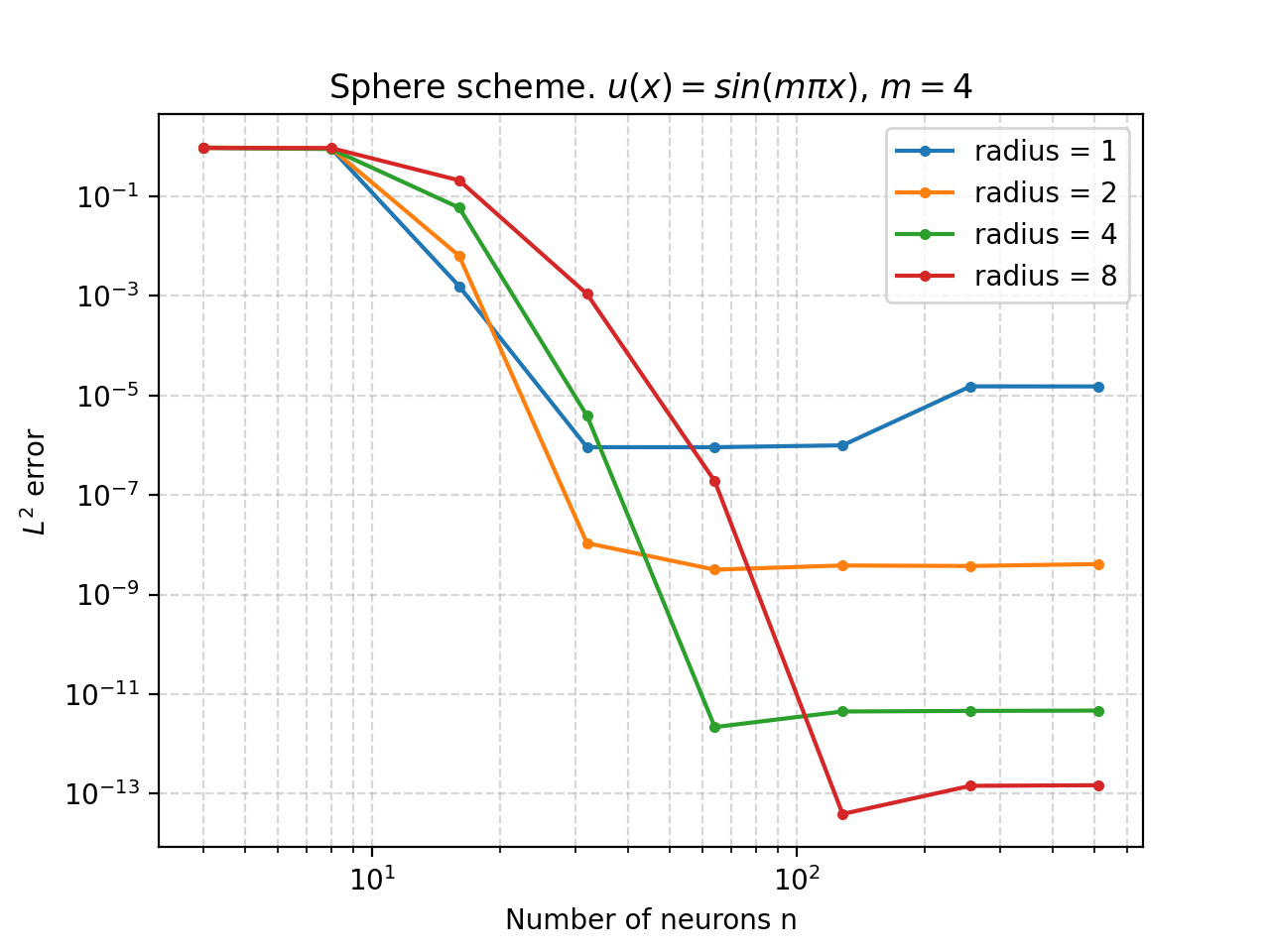}
    \caption{1D $L^2$ minimization with $\tanh$ activations. Collocation formulation. 
    Target function $u(x) = \sin(m\pi x)$ for $m=1,2,4$. }
    \label{fig:1dl2-tanh-col-sphere}
\end{figure}

\begin{figure}[H]
    \centering
    \includegraphics[width=0.325\linewidth]{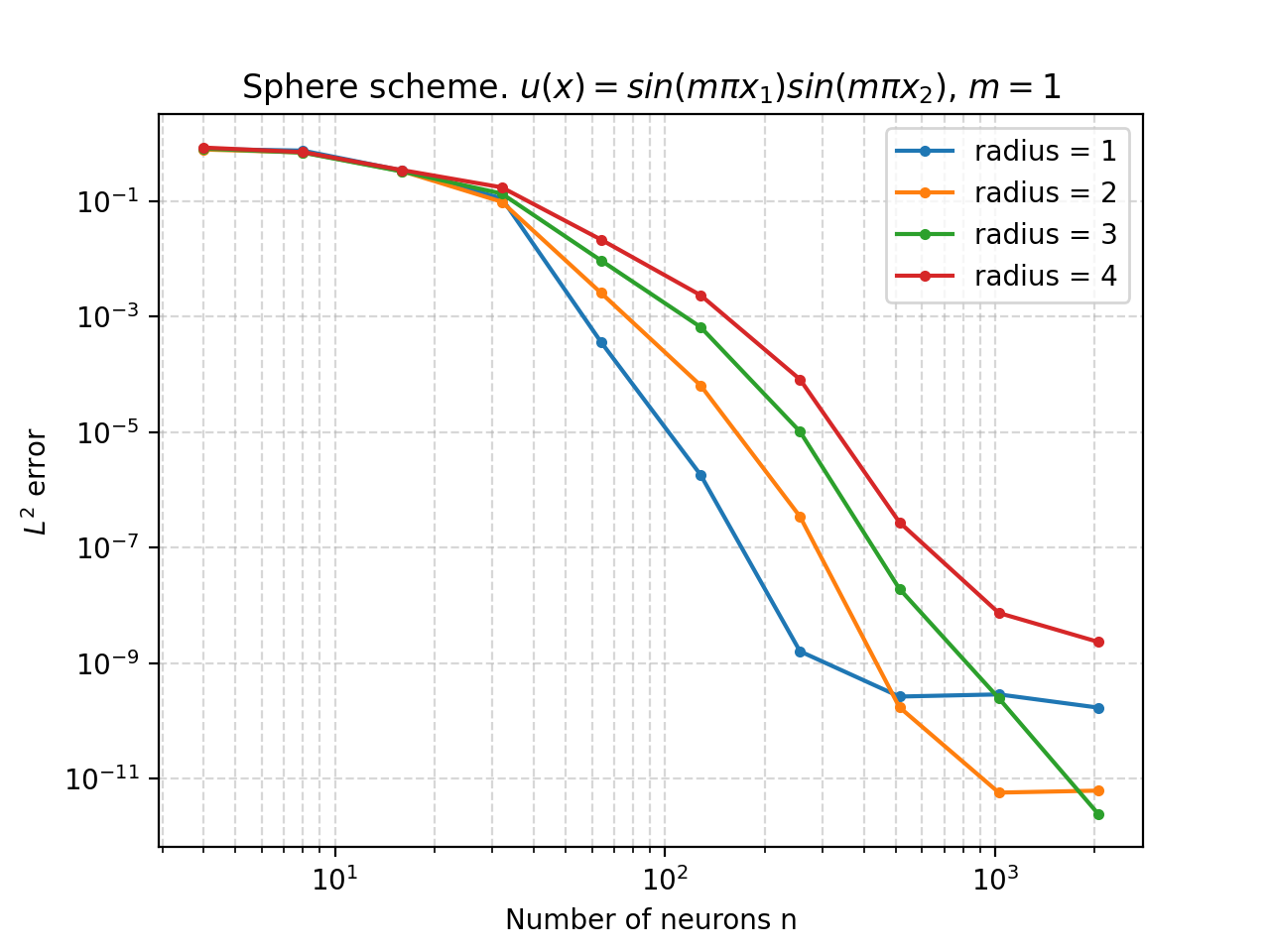}
    \includegraphics[width=0.325\linewidth]{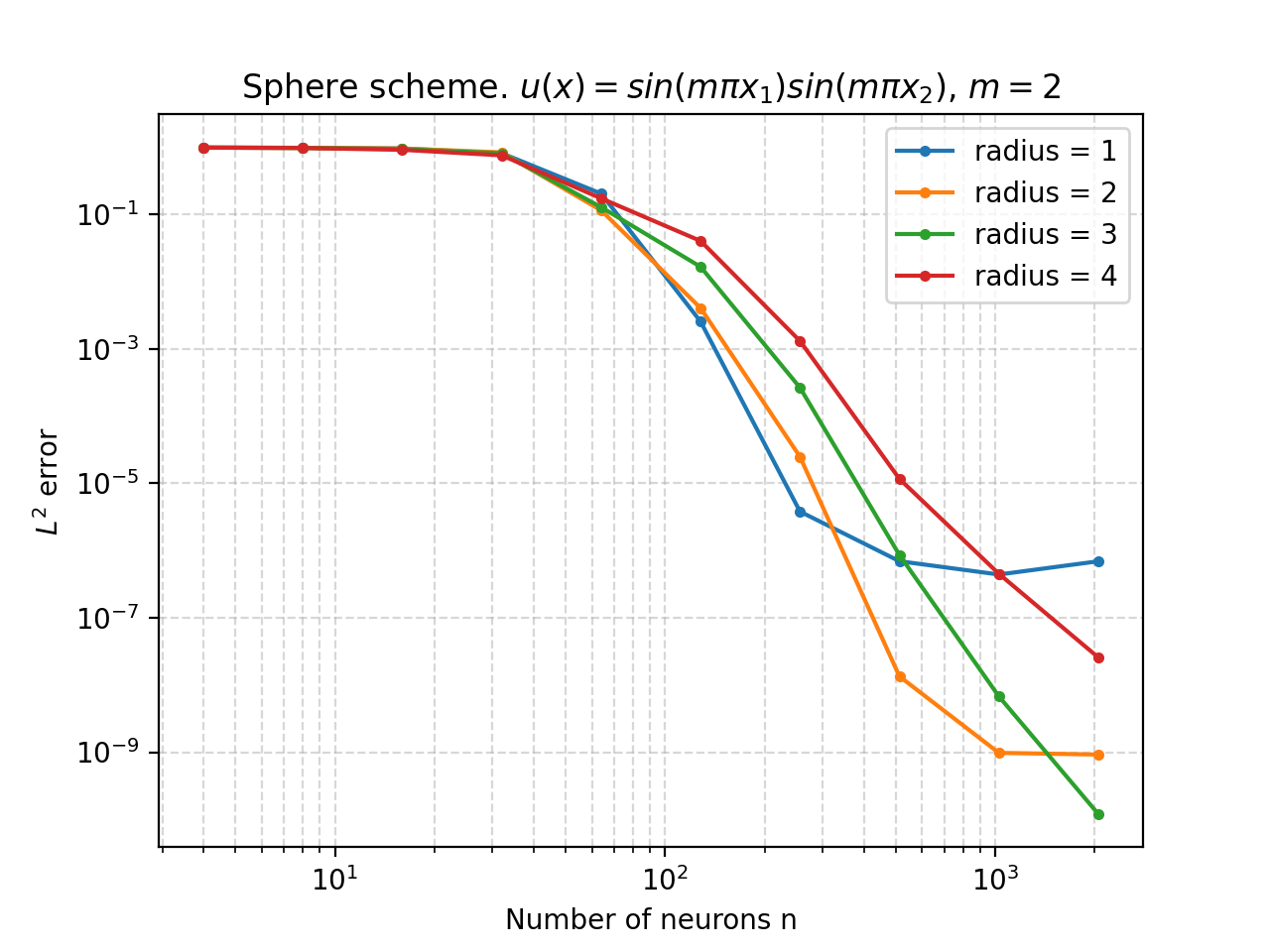}
    \includegraphics[width=0.325\linewidth]{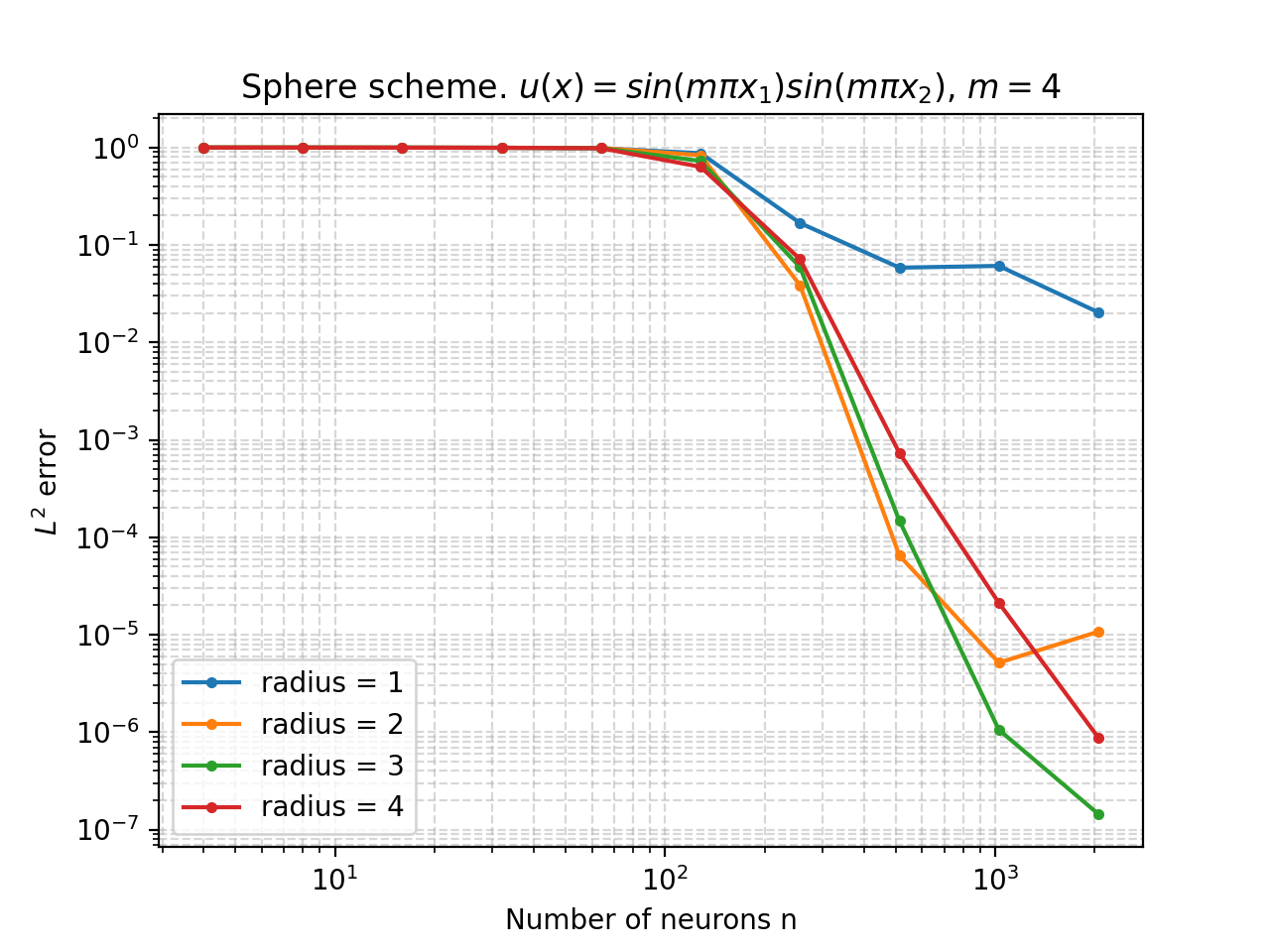}
    \caption{2D $L^2$ minimization with $\tanh$ activations. Collocation formulation. 
    Target function $u(x) = \sin(m\pi x_1)\sin(m\pi x_2)$ for $m=1,2,4$. }
    \label{fig:2dl2-tanh-col-sphere}
\end{figure}

\textbf{Second-order elliptic PDEs}
In the following, we solve a second-order elliptic PDE using linearized shallow neural network with $\tanh$ activation under the collocation formulation as described in Section~\ref{sec:methods}. 
The nonlinear layer parameters are fixed a priori using the
proposed sphere scheme, while the linear output layer is determined by solving a least-squares problem.

We first consider a one-dimensional second-order elliptic problem \eqref{eq:Neumann}
with Dirichlet boundary conditions. 
In one dimension, we choose the exact solution
\begin{equation}
\label{eq:1d-pde-sln}
u(x) = \sum_{i=1}^{3} \sin(m_i \pi x),
\quad m_1 = 1,\; m_2 = 2,\; m_3 = 4.
\end{equation}
The corresponding right-hand-side function is computed analytically from the PDE.
The collocation points are chosen as a uniform grid of 200 points on $[-1,1]$, including
the boundary points.

We initialize the nonlinear layer parameters of the $\tanh$ neural network using the
sphere scheme and test a range of different sphere radii. The approximation errors are evaluated using an accurate piecewise Gaussian quadrature rule, namely, the interval $[-1,1]$
is divided into 1024 uniform subintervals, and a Gauss quadrature rule of order 3 is applied on each subinterval.

We report both the $L^2$ error and the $H^1$ seminorm error, plotted against the number
of neurons in Figure~\ref{fig:1dNeumann_tanh_sphere}. As the number of neurons increases,
the method exhibits extremely fast convergence, with errors rapidly decreasing to the
order of $10^{-13}$, indicating a highly accurate numerical solution.

\begin{figure}[H]
    \centering
    \includegraphics[width=0.8\linewidth]{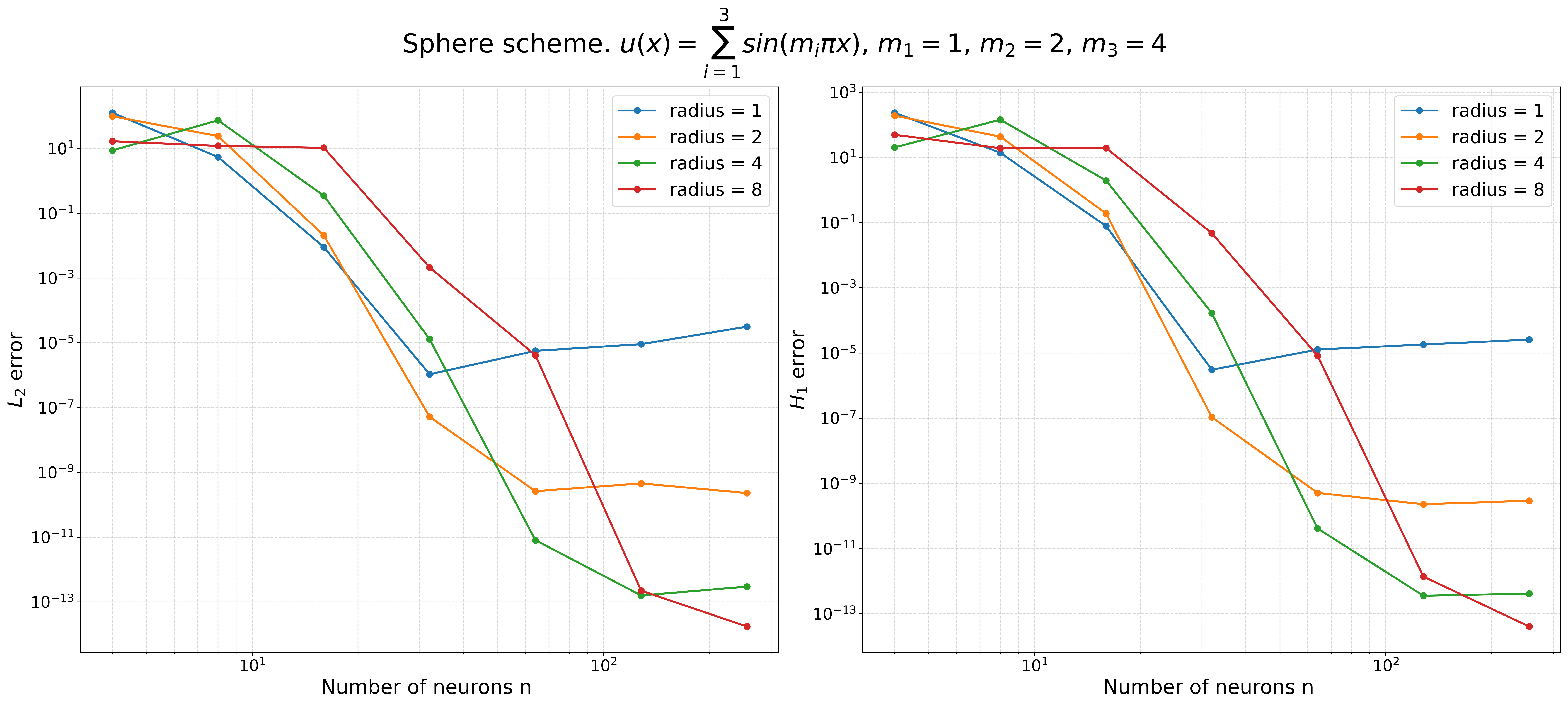}
    \caption{Numerical results for the one-dimensional second-order elliptic problem with exact solution \eqref{eq:1d-pde-sln}, solved using a collocation formulation with $\tanh$ activation. 
The nonlinear layer parameters are initialized by the proposed sphere scheme with different sphere radii, and the resulting $L^2$ and $H^1$ seminorm errors are shown.}
    \label{fig:1dNeumann_tanh_sphere}
\end{figure}

Next, we consider a two-dimensional second-order elliptic problem \eqref{eq:Neumann} with Dirichlet boundary conditions, solved using a linearized shallow neural network
with $\tanh$ activation. 
As in the one-dimensional case, the nonlinear layer parameters are fixed in advance using the proposed sphere scheme, and the linear output layer is
computed by solving a least-squares problem.

In two dimensions, we choose the exact solution
\begin{equation}
\label{eq:2d-pde-sln}
u(\mathbf{x}) = \sum_{i=1}^{3} \sin(m_i \pi x_1)\,\sin(m_i \pi x_2),
\end{equation}
with $m_1 = 1$, $m_2 = 2$, and $m_3 = 4$. 
The corresponding right-hand-side function is computed analytically from the PDE. 
The collocation points are chosen as a uniform $100\times100$ grid on the square domain, with boundary points included.

We initialize the nonlinear layer parameters of the $\tanh$ neural network using the sphere scheme and test a range of different sphere radii.
The approximation errors are evaluated using an accurate piecewise Gaussian quadrature rule: the square domain is partitioned into $50^2$ uniform subdomains, and a Gauss quadrature rule of order 5 is applied on each subdomain.

We report both the $L^2$ error and the $H^1$ seminorm error, plotted against the number of neurons in Figure~\ref{fig:2dNeumann_tanh_sphere}. As the number of neurons increases, the method continues to exhibit rapid convergence. In particular, the $L^2$ error reaches the order of $10^{-9}$, while the $H^1$ seminorm error reaches the order of $10^{-7}$, demonstrating the effectiveness of the proposed sphere scheme in two dimensions.

\begin{figure}[H]
    \centering
    \includegraphics[width=0.8\linewidth]{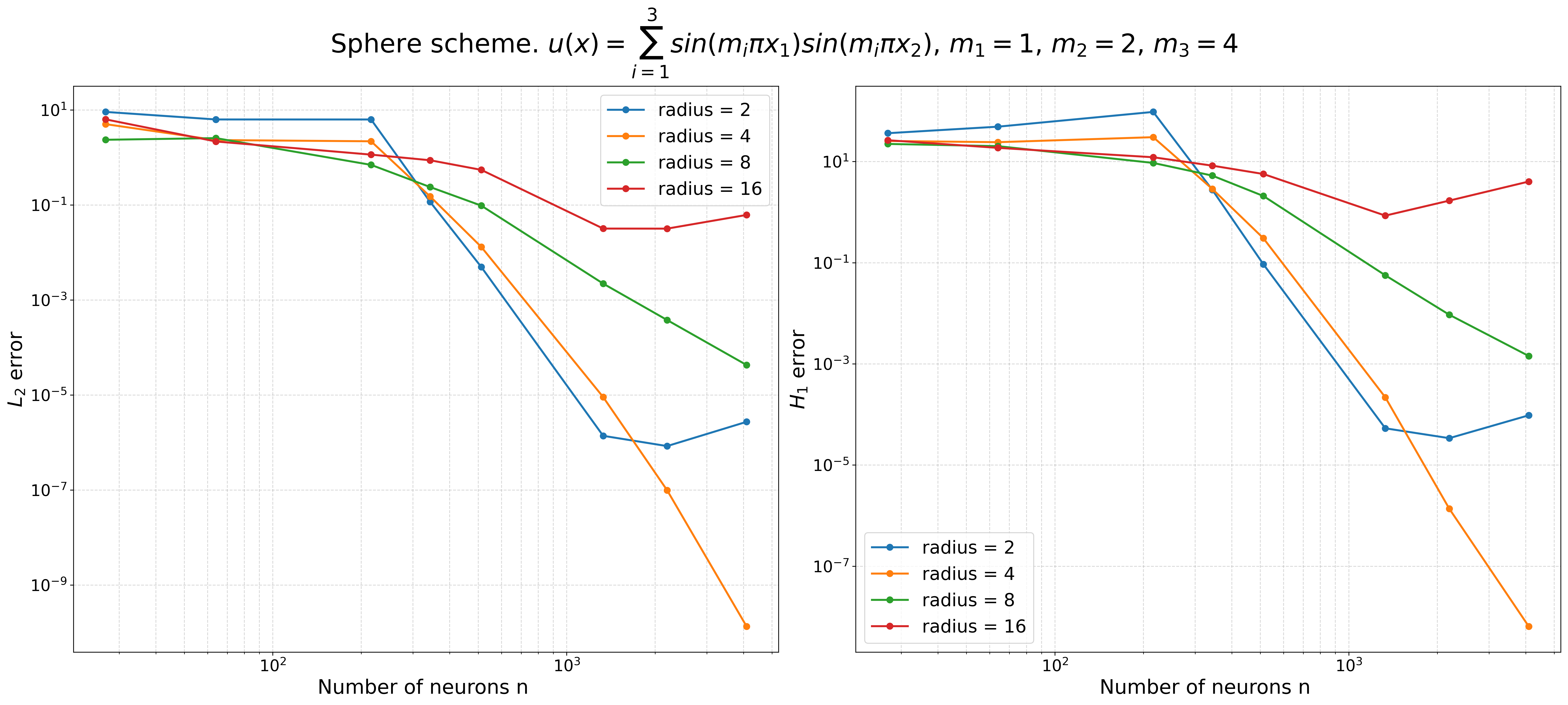}
    \caption{2D Neumann problem solved by a collocation method using $\tanh$ activation.
    The neurons are predetermined using the proposed sphere scheme with varying sphere radii.
    The corresponding errors are reported.}
    \label{fig:2dNeumann_tanh_sphere}
\end{figure}



\section{Conclusion}\label{sec:conclusion}

In this paper, we have presented a systematic numerical study of linearized shallow neural networks for the approximation of high-dimensional functions and the 
solution of elliptic PDEs. 
They are shallow neural networks with deterministically fixed hidden-layer parameters. 
Our results demonstrate that random sampling of the hidden-layer parameters, commonly used in RFM and ELM, is not necessary for achieving high accuracy. 
By employing quasi-uniform and quasi–Monte Carlo points to determine the hidden-layer features of ReLU$^k$ networks, we numerically demonstrate that the linearized ReLU$^k$ shallow neural networks achieve the optimal approximation rate.
We also propose two deterministic parameter schemes for linearized $\tanh$ shallow neural networks and achieve superior accuracy in function approximation and elliptic PDE problems. 

Furthermore, we carry out a numerical study of the variational and collocation 
formulations for linearized neural networks. 
The variational approach enjoys a clean theoretical foundation as a 
Galerkin projection, yet our experiments show that it suffers from severe 
ill-conditioning as the number of neurons increases. 
In contrast, the collocation 
formulation, implemented in a least–squares setting, is shown to be more numerically stable when solved with a robust least-squares solver, though its error analysis remains less complete. 
Across both formulations, conditioning emerges as a critical bottleneck: once the 
condition numbers of the associated Gram or stiffness matrices grow too large, further gains in accuracy become impossible, regardless of approximation capacity.


Overall, this work demonstrates both the promise and the limitations of linearized neural networks for high-dimensional function approximation problems and PDEs. 
The numerical evidence points clearly to conditioning as a major barrier to scalability in variational methods, while also showing that collocation can yield accurate results despite a lack of complete error analysis. 
Future research should therefore focus on developing effective preconditioning strategies and establishing rigorous error analysis for collocation-based approaches, so that linearized shallow networks can be used more reliably for challenging 
scientific computing problems.

\appendix
\section{Supplementary experiments on $L^2$-minimization} \label{app:supp-l2}
We report additional numerical results for ReLU activations to complement the ReLU$^k$ experiments presented in the main text. The experimental setup and evaluation criteria are the same as those in Section~\ref{subsec:relul2}, and the results exhibit similar error decay behavior across dimensions $d=1$--$6$.

\subsection{Continuous $L^2$-minimization}
\begin{figure}[H]
    \centering
    \includegraphics[width=0.325\linewidth]{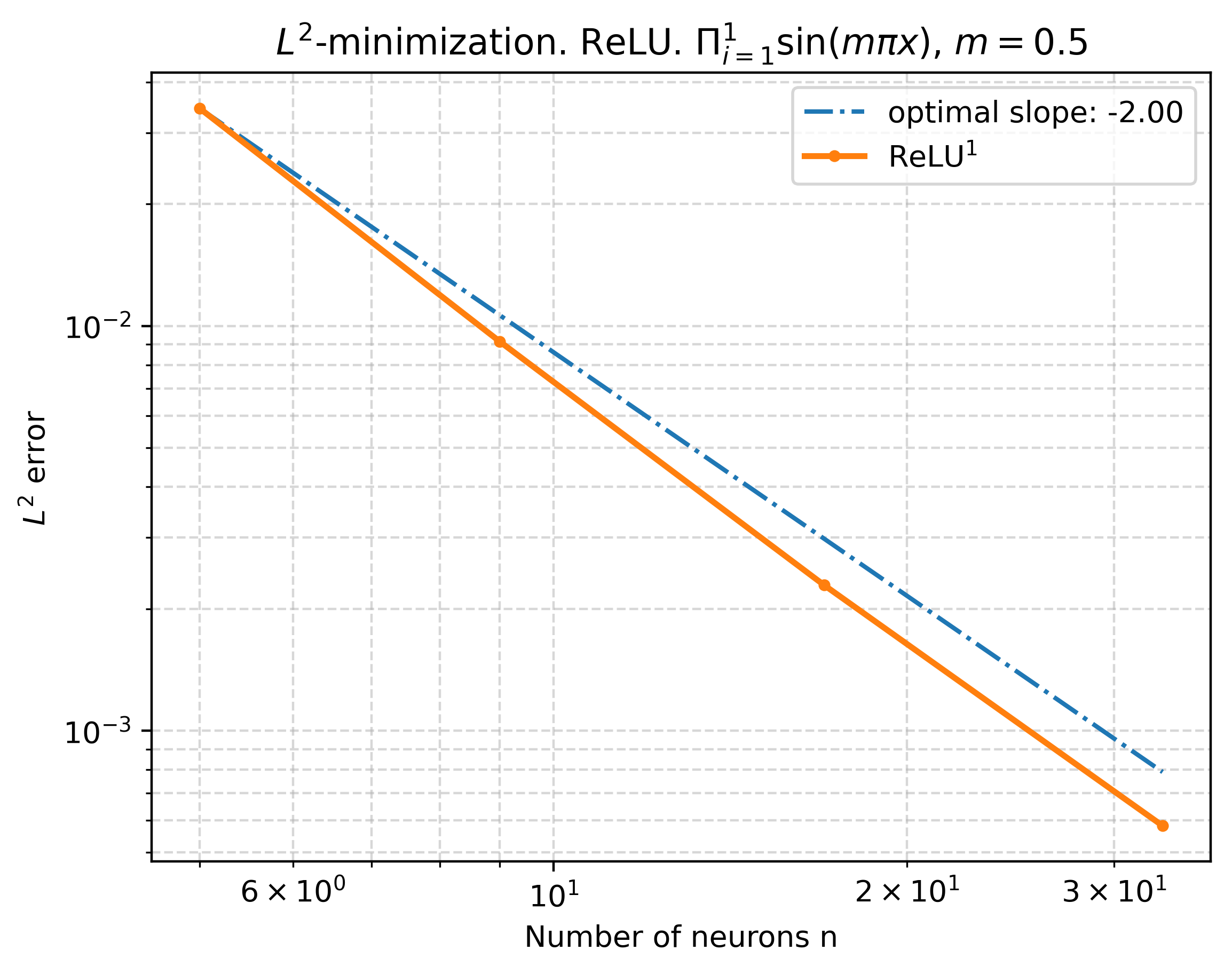}
    \includegraphics[width=0.325\linewidth]{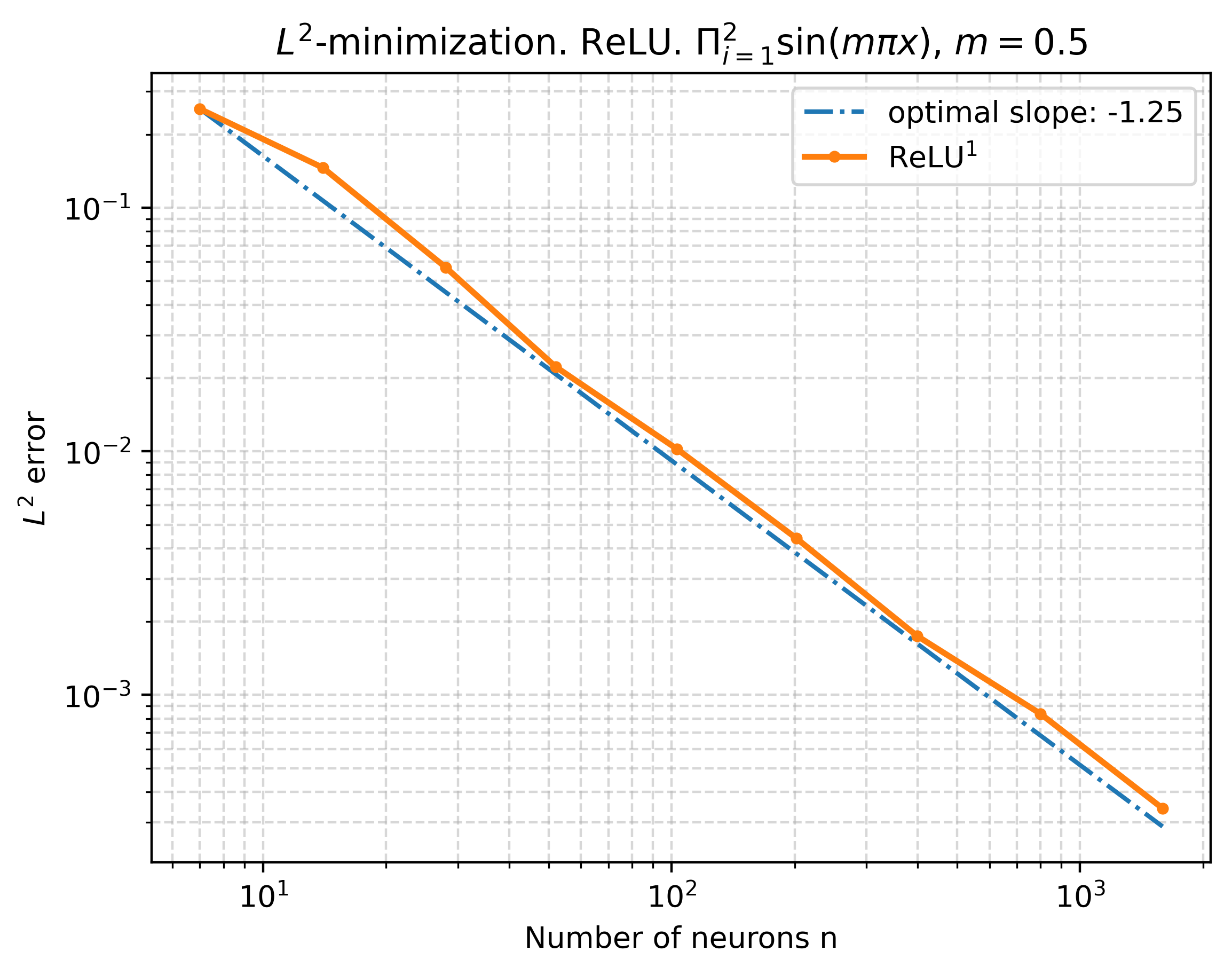}
    \includegraphics[width=0.325\linewidth]{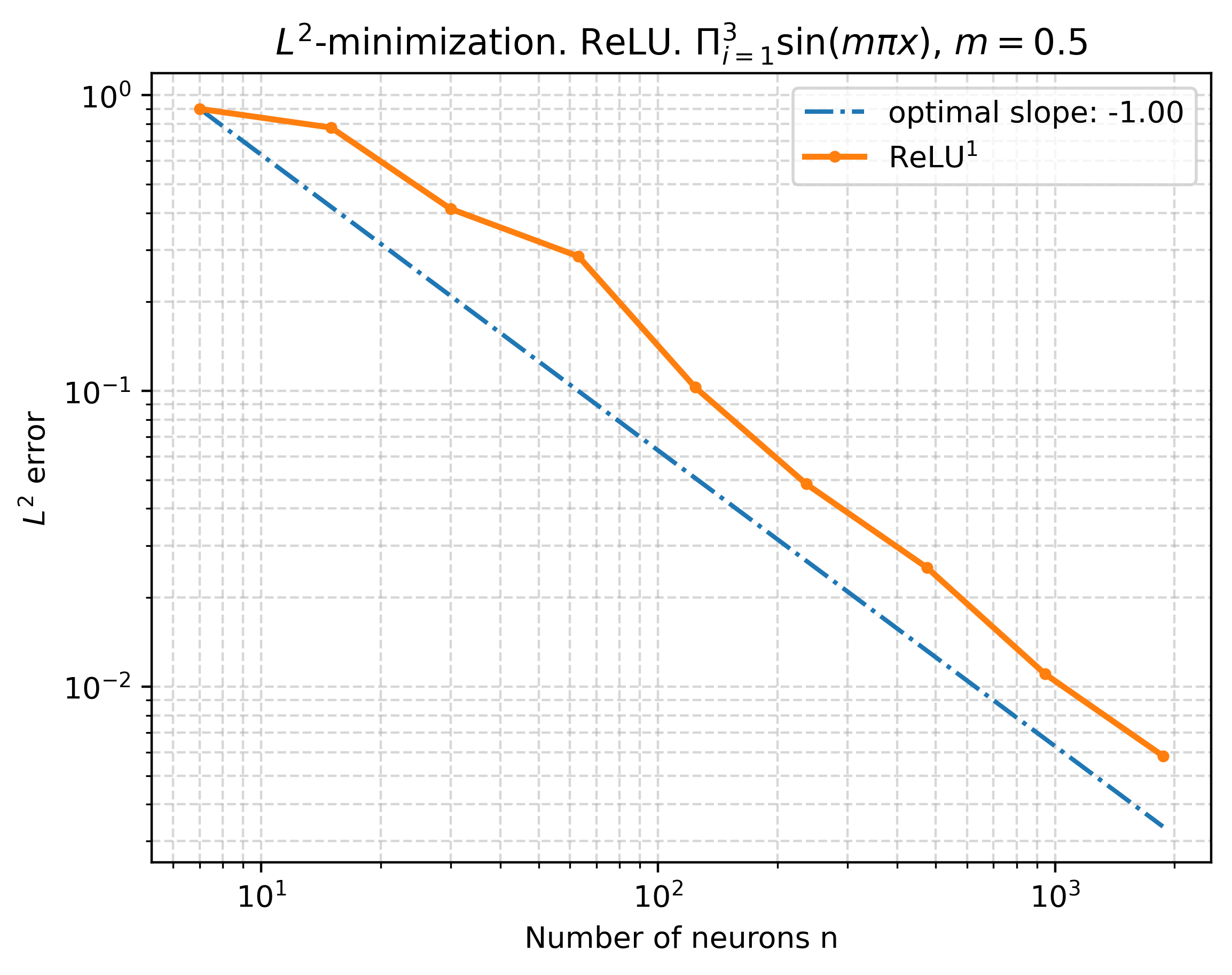} \\
    \includegraphics[width=0.325\linewidth]{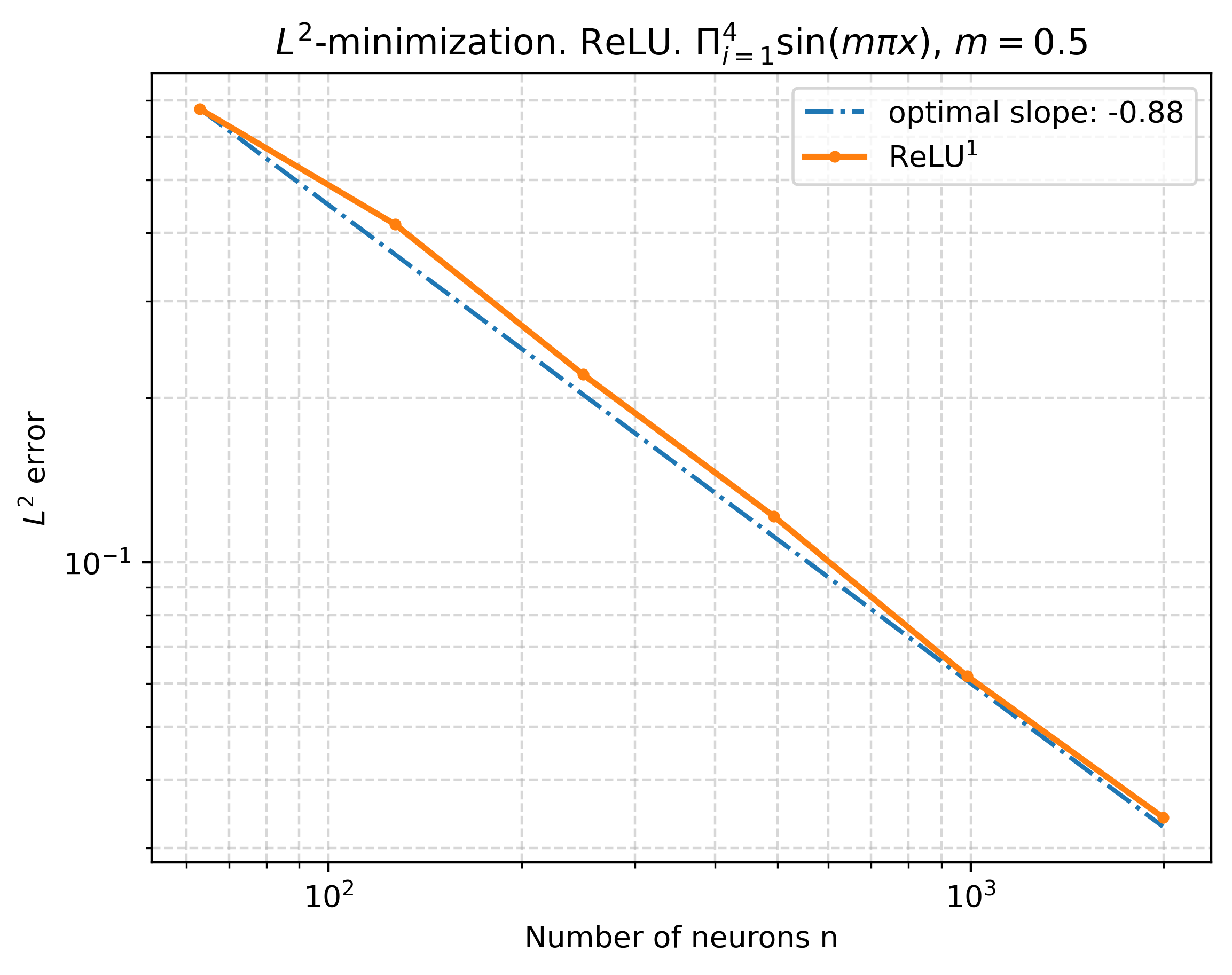}
    \includegraphics[width=0.325\linewidth]{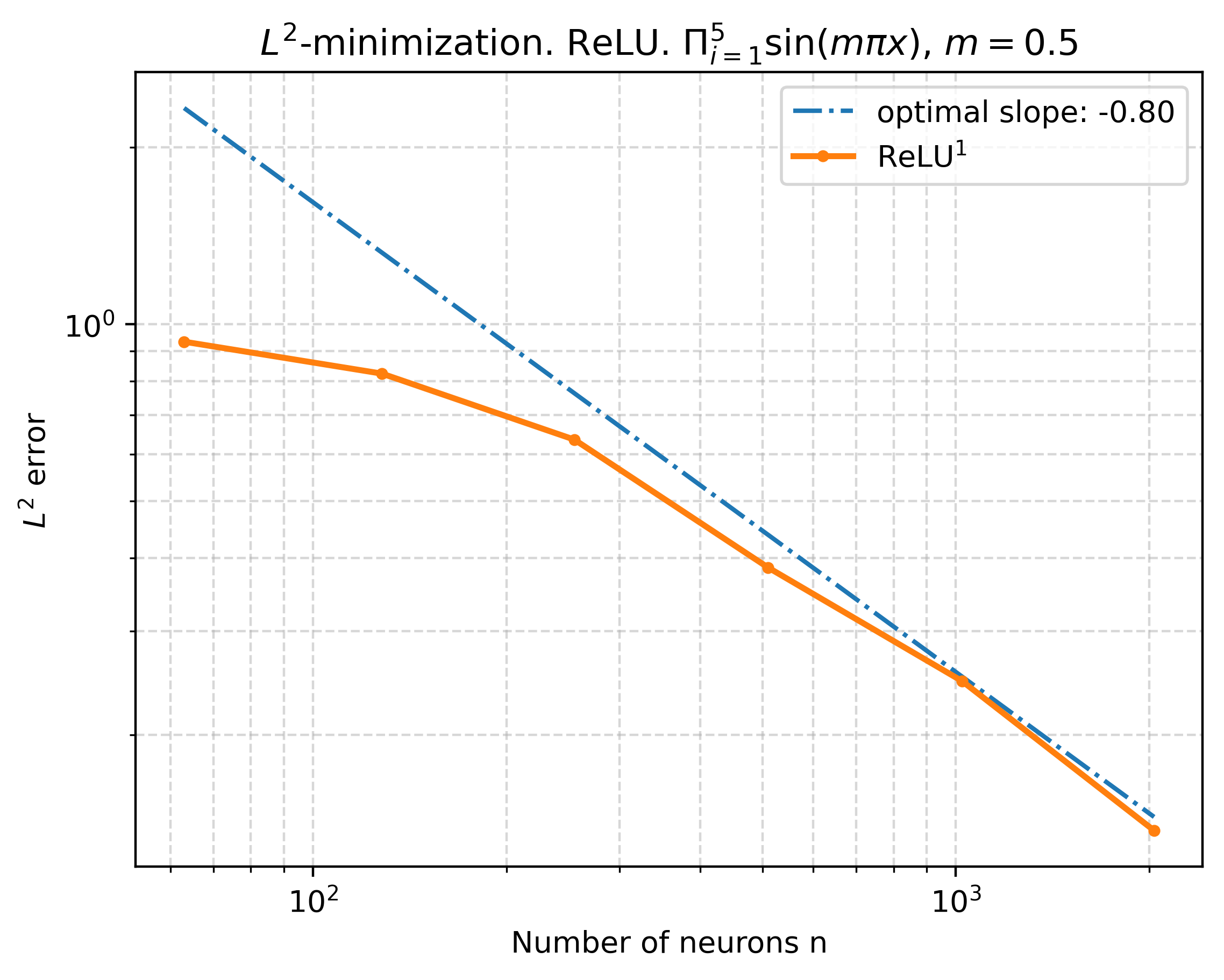}
    \includegraphics[width=0.325\linewidth]{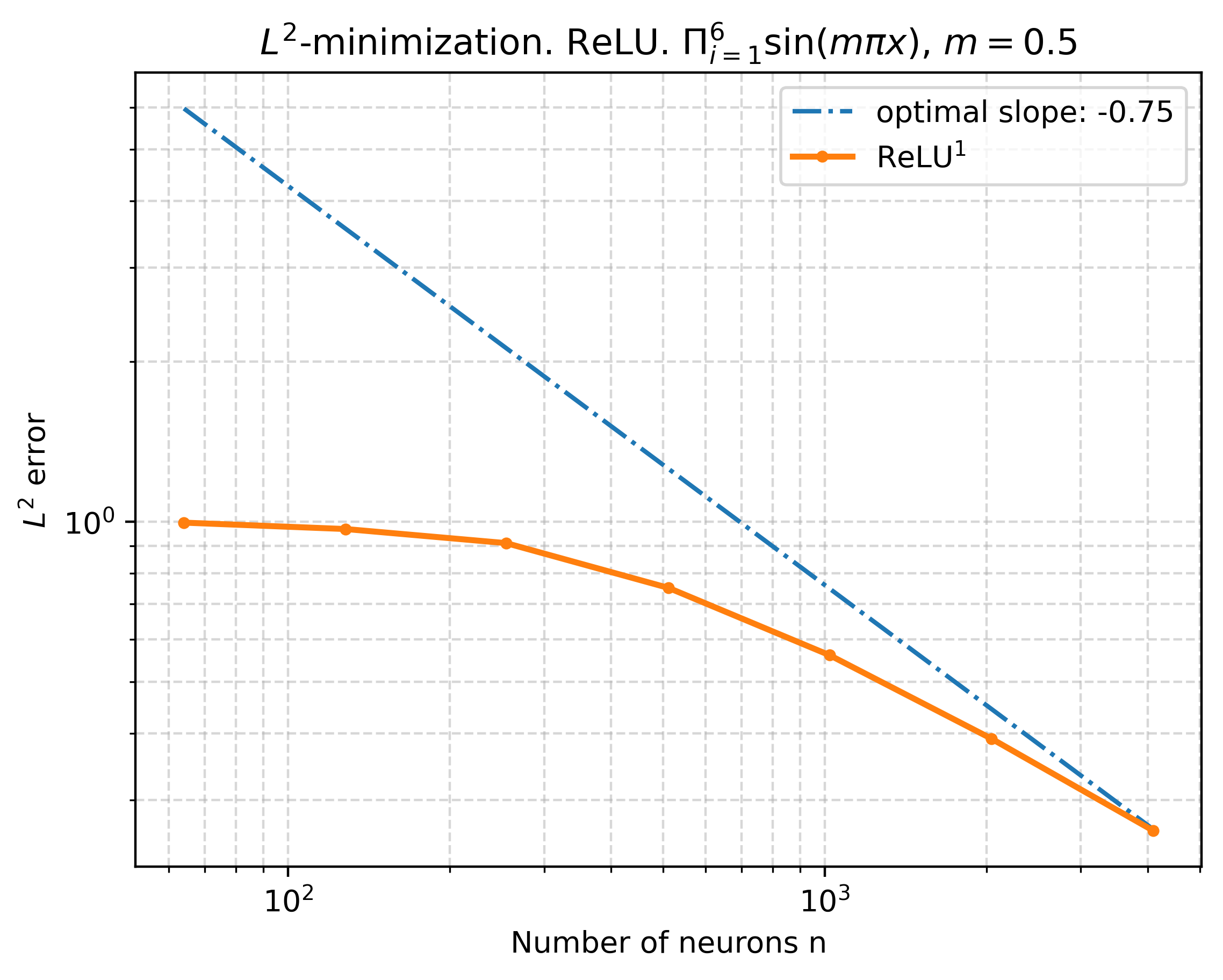} 
    \caption{Error decay for continuous $L^2$-minimization (variational formulation) 
    in dimensions $d=1$--$6$.}
    \label{fig:L2-variational-relu}
\end{figure}

\subsection{Discrete $\ell^2$-regression}

\begin{figure}[H]
    \centering
    \includegraphics[width=0.325\linewidth]{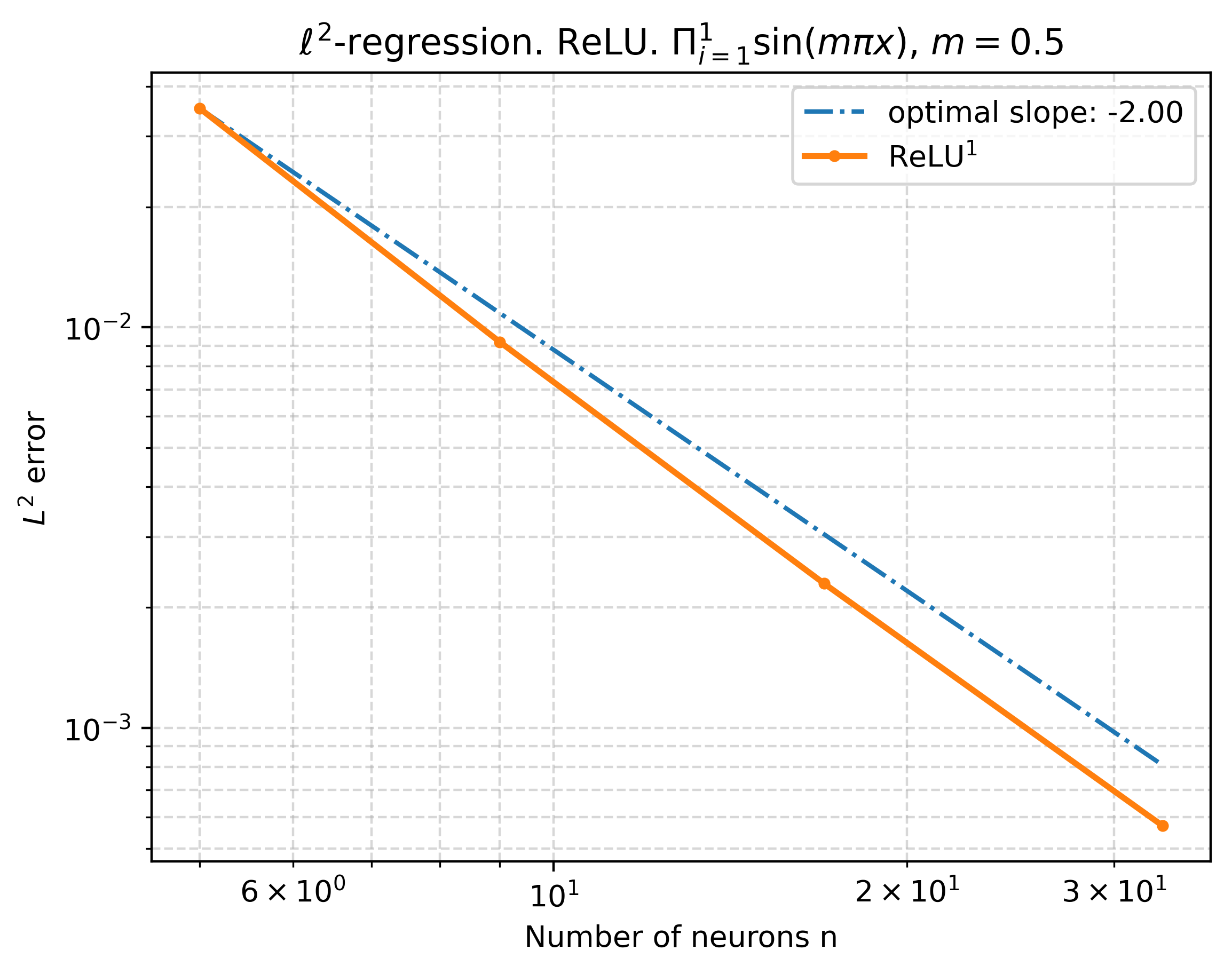}
    \includegraphics[width=0.325\linewidth]{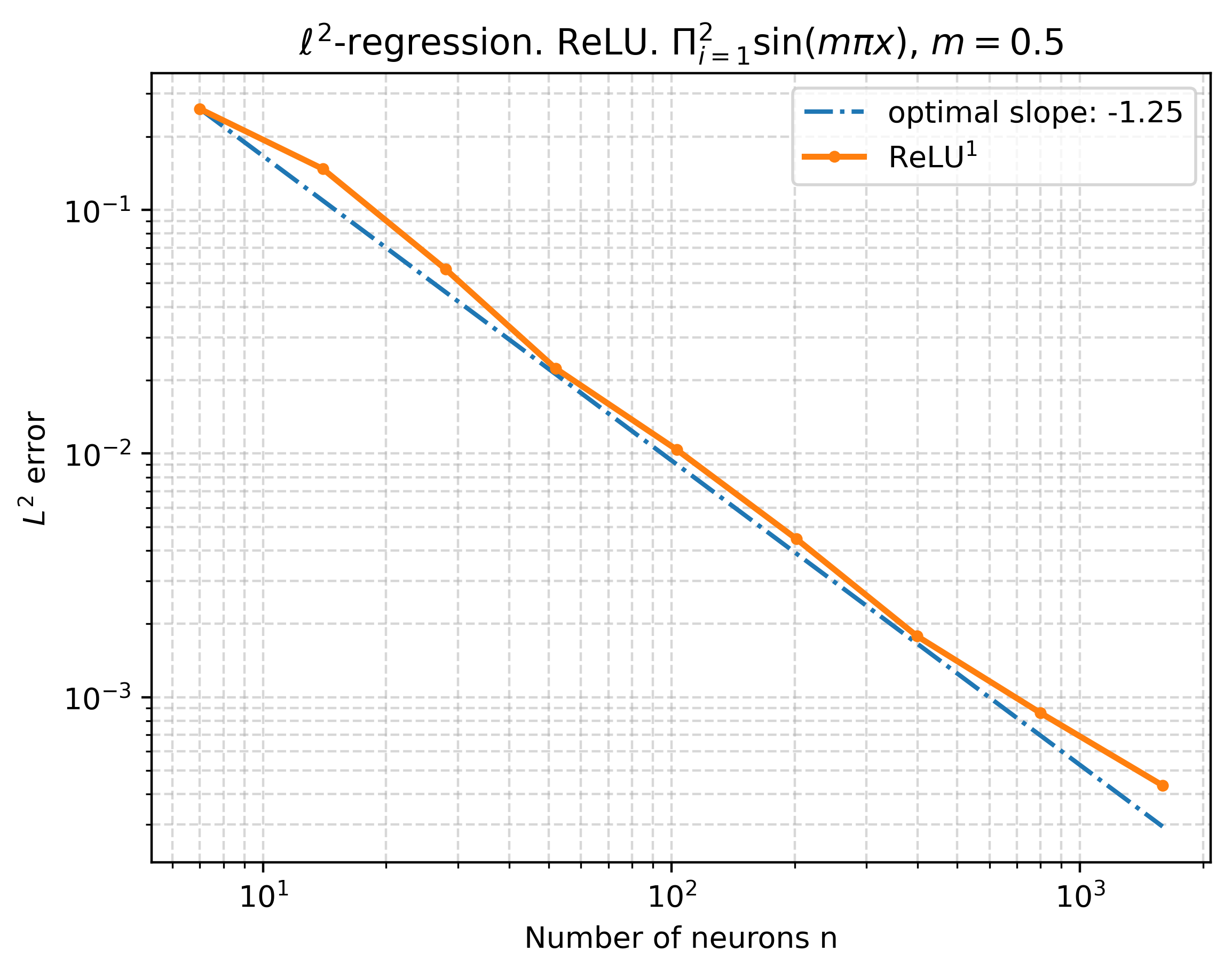}
    \includegraphics[width=0.325\linewidth]{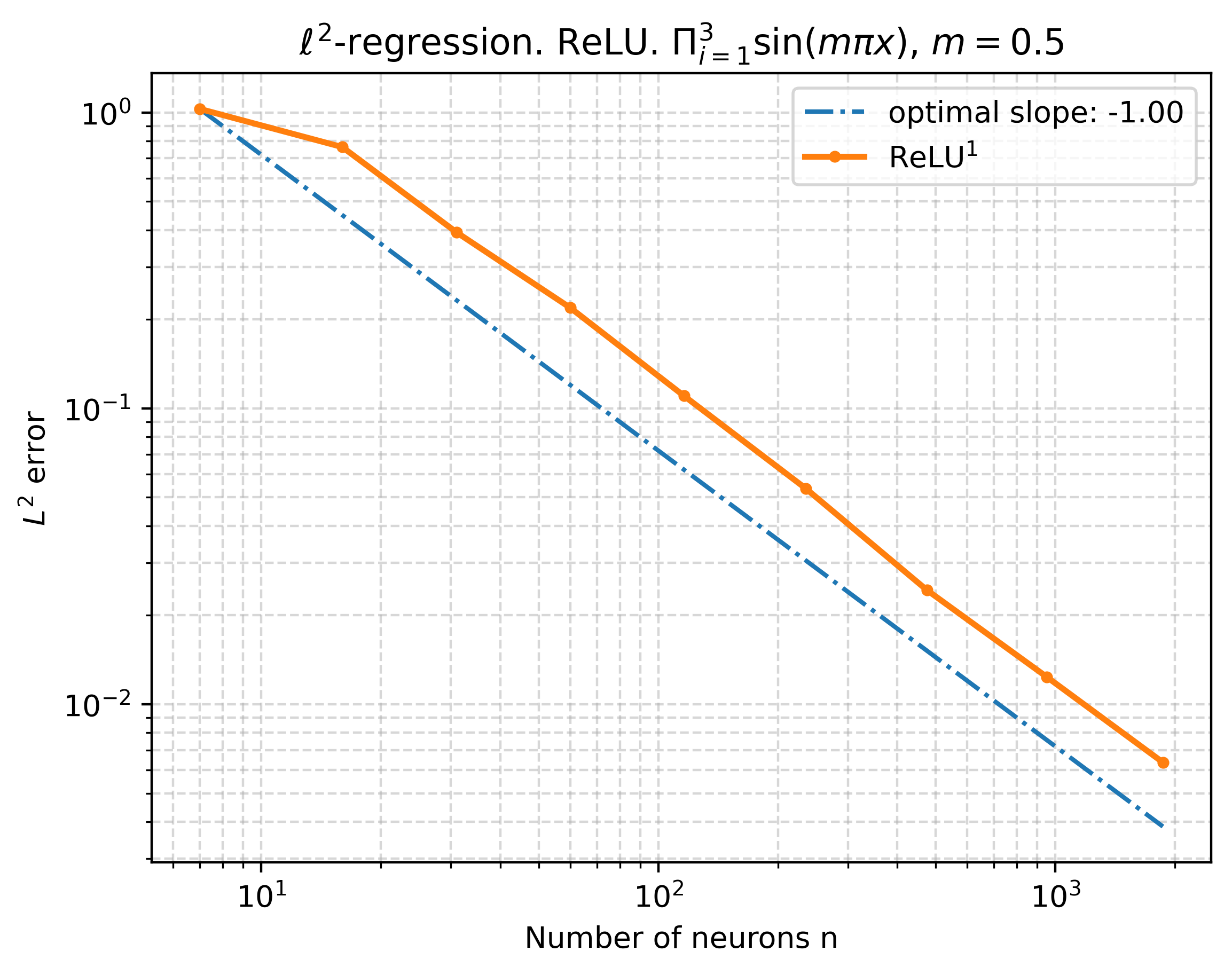} \\
    \includegraphics[width=0.325\linewidth]{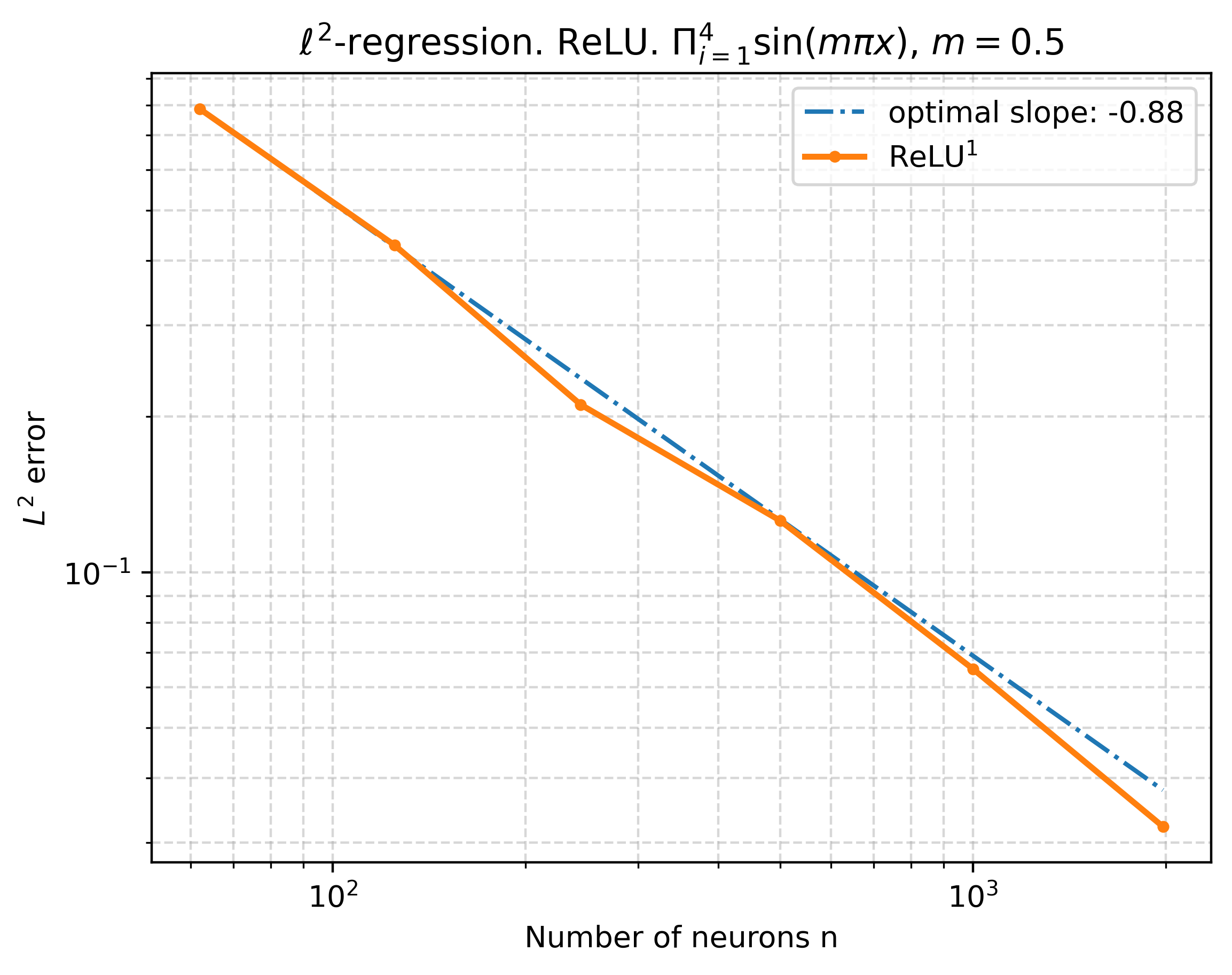}
    \includegraphics[width=0.325\linewidth]{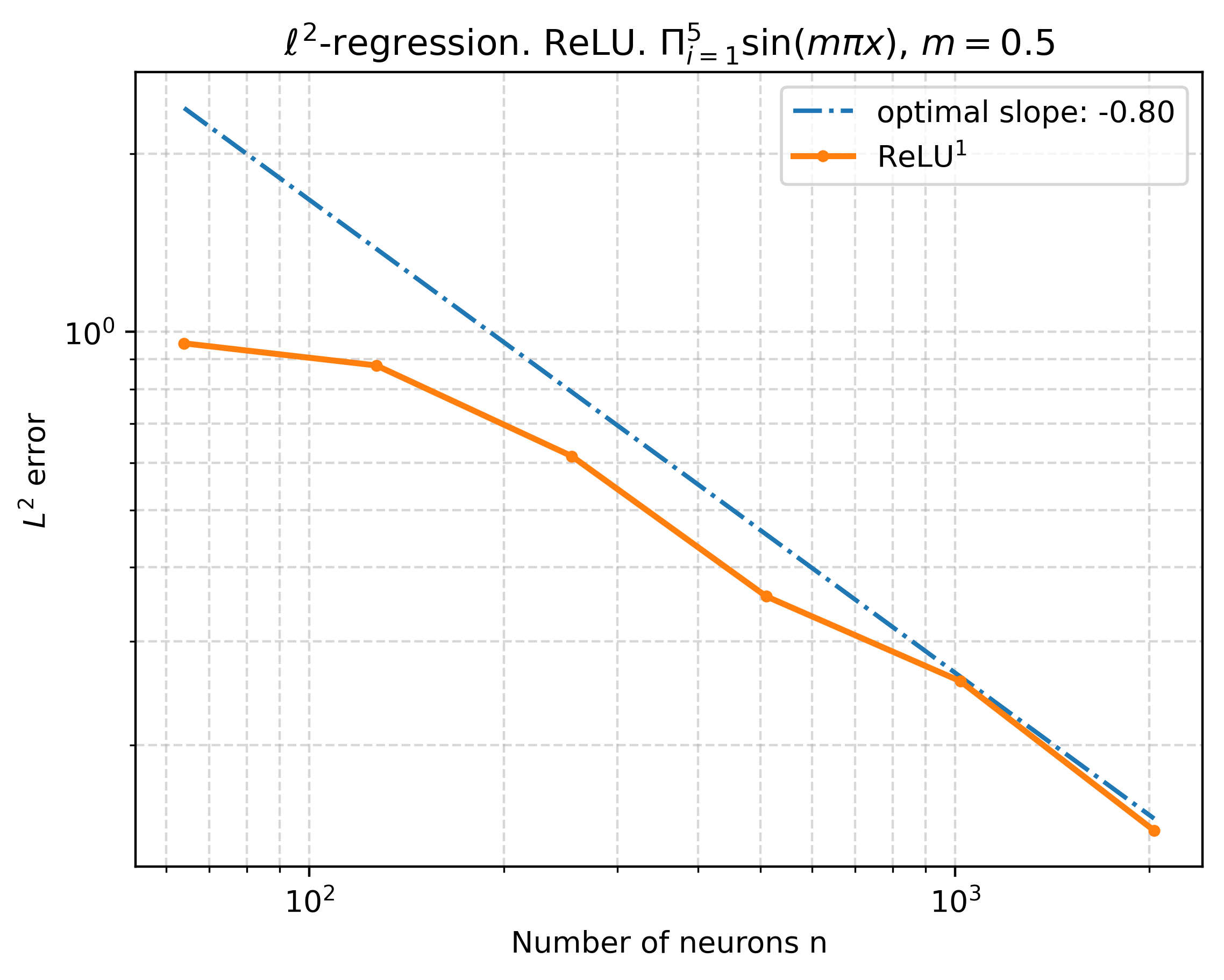}
    \includegraphics[width=0.325\linewidth]{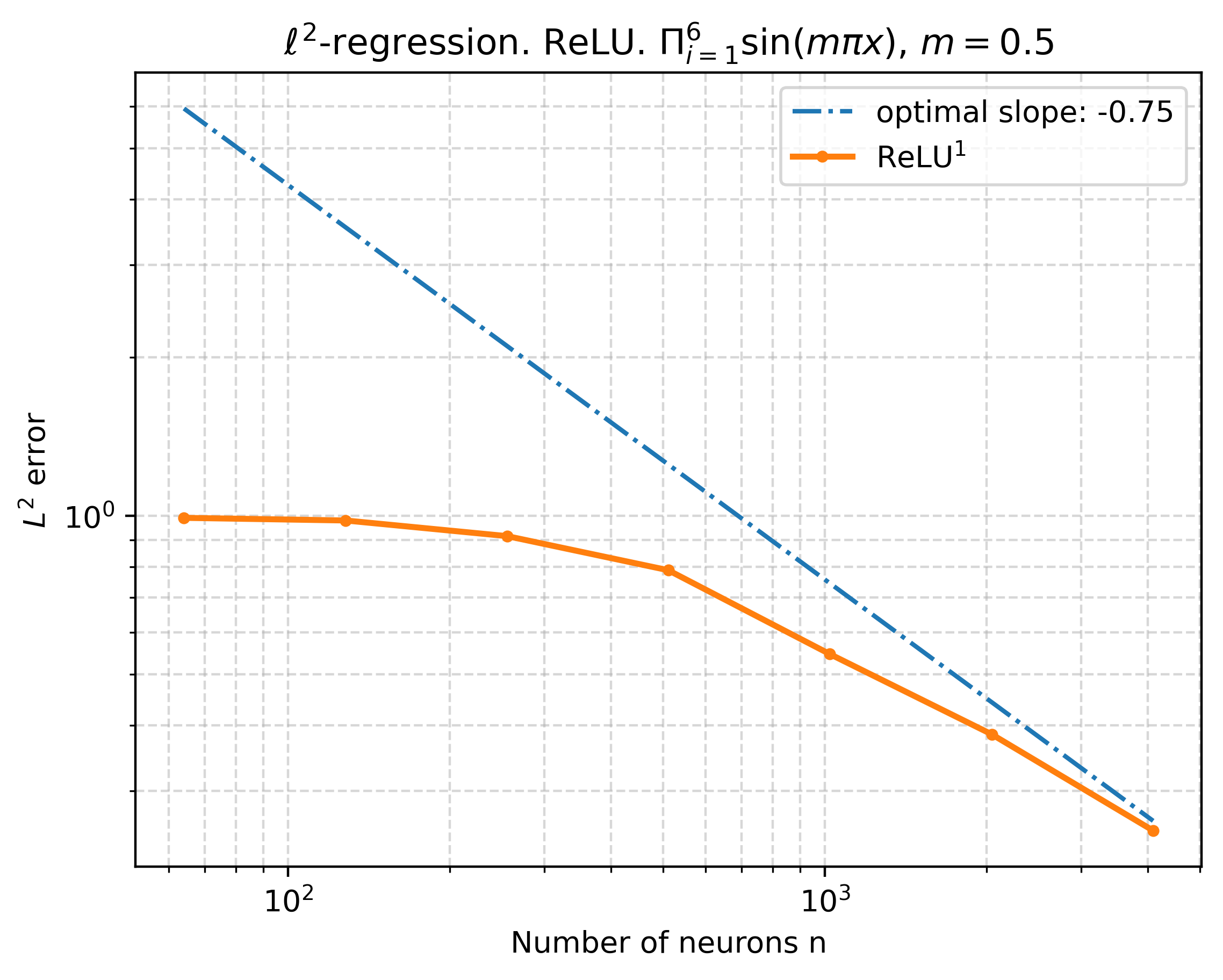} 
    \caption{Error decay for discrete $\ell^2$-regression (collocation formulation) 
    in dimensions $d=1$--$6$.}
    \label{fig:L2-regression-relu}
\end{figure}

\section{Variational least-squares formulation}
\label{app:least-square}

Now we consider the variational least-squares formulation for the $L^2$-minimization task.
Let $u$ be the target function, and let $\{ \phi_j \}_{j = 1}^n$ be the pre-determined neurons. 
We seek to solve the $L^2$-minimization problem on the subspace $\text{span}\{  \phi_j \}_{j = 1}^n$. 
Given a quadrature rule, $\int_{\Omega} f(x) dx \approx \sum_{i =1}^N w_i f(x_i)$, where $w_i$, $x_i$ are the quadrature weights and quadrature points, respectively,  the corresponding linear system for the $L^2$-minimization is
\begin{equation}
    M \mathbf{a} = \mathbf{b}, 
\end{equation}
where $M_{ij}= \sum_{k = 1}^N w_k \phi_j(x_k)  \phi_i(x_k)$, $\mathbf{b}_i = \sum_{k = 1}^N w_k u(x_k) $.
In a more compact matrix form, we have 
\begin{equation}
    M = \Phi^T W \Phi, \mathbf{b} = \Phi^T W \mathbf{u},  
\end{equation}
where
\begin{itemize}
    \item $\Phi \in \mathbb{R}^{N\times n}$ is the matrix of basis functions evaluated at all the quadrature points;
    \item $W = \text{diag}(w_1,\dots,w_N)$
    \item $\mathbf{u} = (u(x_1), \dots, u(x_N)) $ is the function values of the target function at all the quadrature points. 
\end{itemize} 
The above matrix form also has an equivalent minimization formulation given by 
\begin{equation}
    \min_{\mathbf{a} \in \mathbb{R}^n} \ \| \Phi \mathbf{a} - \mathbf{u} \|_{W}^2,
\end{equation}
where $\| \cdot\|_W$ denotes the norm induced by the weighted inner product $\langle u,v \rangle_W := \sum_{i = 1}^N w_i u_i v_i$ on $\mathbb{R}^N$. 
The following theorem establishes the equivalence.
\begin{theorem}
Any minimizer $\mathbf{a}$ of the weighted least-squares problem
\begin{equation}\label{eq:ls}
    \min_{\mathbf{a} \in \mathbb{R}^n} \ \| \Phi \mathbf{a} - \mathbf{u} \|_{W}^2
\end{equation}
satisfies the \emph{normal equations}
\begin{equation}\label{eq:ne}
    \Phi^\top W \Phi \, \mathbf{a} \;=\; \Phi^\top W \mathbf{u} .
\end{equation}
Conversely, any solution of \eqref{eq:ne} is a stationary point of \eqref{eq:ls}. 
If $W \succ 0$ and $\Phi$ has full column rank, the minimizer of \eqref{eq:ls} is a unique global minimizer and is given by
\begin{equation}\label{eq:sol}
    \mathbf{a} \;=\; (\Phi^\top W \Phi)^{-1} \Phi^\top W \mathbf{u} .
\end{equation}
\end{theorem}

\begin{proof}
Assume $W\in\mathbb{R}^{N\times N}$ is symmetric positive definite and let
\[
L(\mathbf{a})\;:=\;\|\Phi \mathbf{a}-\mathbf{u}\|_W^2=(\Phi \mathbf{a}-\mathbf{u})^\top W(\Phi \mathbf{a}-\mathbf{u}).
\]
Expanding and using $W^\top=W$,
\[
L(\mathbf{a})=\mathbf{a}^\top\Phi^\top W\Phi\,\mathbf{a}-2\,\mathbf{u}^\top W\Phi\,\mathbf{a}+\mathbf{u}^\top W\mathbf{u} .
\]
The gradient of $L$ is 
\[
\nabla L(\mathbf{a})=2\,\Phi^\top W(\Phi \mathbf{a}-\mathbf{u}).
\]
Hence any minimizer $\mathbf{a}$ of \eqref{eq:ls} must satisfy the first-order optimality
condition $\nabla L(\mathbf{a})=0$, i.e.
\[
\Phi^\top W\Phi\,\mathbf{a}=\Phi^\top W\mathbf{u},
\]
which are exactly the normal equations \eqref{eq:ne}.

Conversely, $L$ is convex because its Hessian is
\[
\nabla^2 L(\mathbf{a})=2\,\Phi^\top W\Phi\succ 0 .
\]
For a convex function, any stationary point is a global minimizer; therefore any
solution of \eqref{eq:ne} minimizes \eqref{eq:ls}. If, in addition, $W\succ 0$
and $\Phi$ has full column rank, then $\Phi^\top W\Phi\succ 0$, so $L$ is
strictly convex and the minimizer is unique and given by
\[
\mathbf{a}=(\Phi^\top W\Phi)^{-1}\Phi^\top W\mathbf{u},
\]
which is \eqref{eq:sol}.
\end{proof}

\begin{remark}[Numerical advantages of solving the least–squares form]
It is advantageous to solve the weighted least–squares problem
\[
\min_\mathbf{a} \|W^{1/2}(\Phi \mathbf{a}-\mathbf{u})\|_2
\]
directly, rather than forming and solving the normal equations
$(\Phi^\top W\Phi)\mathbf{a}=\Phi^\top W \mathbf{u}$. 
Forming the normal equations squares the singular values of
$W^{1/2}\Phi$, yielding
\[
\kappa_2(\Phi^\top W \Phi) = \kappa_2(W^{1/2}\Phi)^2.
\]
This squaring of the condition number not only amplifies numerical errors such as quadrature and roundoff errors, but also makes certain singular values too small to be reliably resolved numerically, leading to a more sensitive and less reliable linear system in practice.
\end{remark}
Finally, we plot the singular values of the least-squares system arising from the discrete $\ell^2$-regression with $\tanh$ shallow neural networks in
Figure~\ref{fig:svd-tanh}.
\begin{figure}[H]
    \centering
    \includegraphics[width=0.5\linewidth]{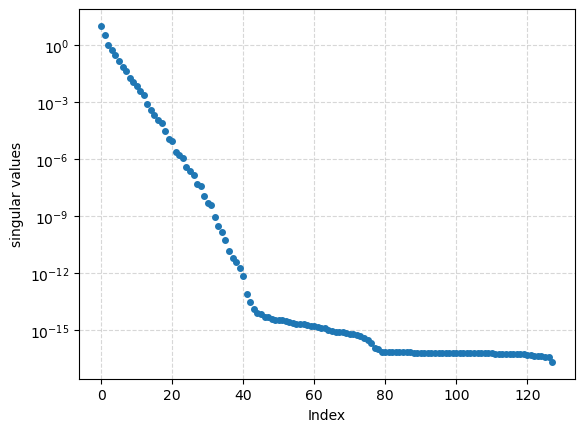}
    \caption{Singular values of the least-squares system for the discrete
    $\ell^2$-regression problem using a $\tanh$ shallow neural network with
    width $n = 128$, constructed via Petrushev's scheme. The rapid decay and subsequent flattening of the singular spectrum show that many singular values are extremely small, which explains the poor conditioning of the least-squares system and motivates avoiding the formation of normal equations.}
    \label{fig:svd-tanh}
\end{figure}

\nocite{siegel2023greedy, SX:2024, LS:2024, Temlyakov:2015greedyconvex}
\bibliography{reference} 
\bibliographystyle{ieeetr}

\end{document}